\documentclass[a4paper, 11pt]{article}
\usepackage{anysize}
\usepackage{mathptmx}

\marginsize{3,5cm}{2,5cm}{2,5cm}{2,5cm}
\usepackage{amssymb}
\usepackage{graphicx,psfrag,epsfig,multirow}
\usepackage{hyperref}
\usepackage{amsmath,amsbsy,amssymb,amscd}
\usepackage{t1enc}
\usepackage[cp1250]{inputenc}
\usepackage[british]{babel}
\usepackage[all]{xy}
\usepackage{color}
\usepackage{amsfonts}
\usepackage{latexsym}
\usepackage{amsthm}
\usepackage{mathrsfs}
\usepackage{hyperref}

\newtheorem{question}{Question}

\usepackage{color}
\def\bt{{t^*}}
\def\bk{{k^*}}
\def\cF{{\mathcal F}}
\def\cT{{\mathcal T}}
\def\cS{{\mathcal S}}

\def\bn{{n^*}}
\def\bl{{l^*}}

\definecolor{orange}{rgb}{1,0.5,0}

\def\orange{\color{orange}}

\definecolor{bluegray}{rgb}{0.4, 0.6, 0.8}
\usepackage{xcolor}
\def\bg{\color{bluegray}}
\def\blk{\color{black}}
\newcommand{\ignore}[1]{}

\usepackage{mathrsfs} 

\DeclareMathAlphabet{\mathpzc}{OT1}{pzc}{L}{it} 

\def\cal#1{{\mathcal #1}} 

\def\a{\alpha}

\newtheorem{definition}{Definition}[section]

\newtheorem{proposition}[definition]{Proposition}
\newtheorem{theorem}{Theorem}
\newtheorem{claim}{Claim}

\newtheorem{corollary}{Corollary}
\newtheorem{remark}{Remark}
\newtheorem{lemma}[definition]{Lemma}

\def\vphi{\varphi}

\def\geq{\geqslant}
\def\leq{\leqslant}
\def\R{\mathbb{R}}
\def\T{\mathbb{T}}
\def\eps{\varepsilon}

\def\Z{\mathbb{Z}}
\def\N{\mathbb{N}}

\def\J{J_*}
\def\Ji{J_i}
\def\cF{\mathcal F}
\def\tu{\tilde{u}}
\def\ts{\tilde{s}}
\def\tv{\tilde{v}}
\def\Nzero{\mathcal N_0}
\def\None{\mathcal N_1}

\newcommand{\bea}{\begin{eqnarray}}
  \newcommand{\eea}{\end{eqnarray}}
  \newcommand{\beab}{\begin{eqnarray*}}
  \newcommand{\eeab}{\end{eqnarray*}}
\renewcommand{\a}{\alpha}
  \newcommand{\be}{\begin{equation}}
  \newcommand{\ee}{\end{equation}}
  
\newcommand{\cI}{\mathcal I}
\newcommand{\cD}{\mathcal D}
\newcommand{\cE}{\mathcal E}

\title{Lebesgue spectrum of countable multiplicity for conservative flows on the torus}

\author{Bassam Fayad, Giovanni Forni and Adam Kanigowski}

\begin{document}
\baselineskip=14pt \maketitle

{\bg
\begin{abstract} \blk We study the spectral measures of conservative mixing flows on the $2$-torus having one degenerate singularity. We show that, for a sufficiently strong singularity,  the spectrum of these flows is typically Lebesgue with infinite multiplicity. 

 For this, we use two main ingredients: 1) a proof of  absolute continuity of the maximal spectral type for this class of non-uniformly stretching flows that have an irregular decay of correlations, 2) a geometric criterion that yields infinite Lebesgue multiplicity of the spectrum and that is well adapted to rapidly mixing flows.
\end{abstract}
}

\bg \section{Introduction} \blk

Smooth conservative, or area-preserving, flows on surfaces provide one of the fundamental examples  in the theory of dynamical systems.  {These flows are often called {\em multi-valued, or locally, Hamiltonian flows}, following the terminology introduced by S.~P.~Novikov \cite{novikov82},  who emphasized their relation with solid state physics \cite{novikov95}. 
In fact, smooth conservative surface flows preserve by definition a smooth area-form, hence they are generated by the
symplectic dual of a closed $1$-form, which is locally the exterior derivative of a multi-valued Hamiltonian function.}

Multi-valued Hamiltonian flows can be viewed as special flows above circle rotations, or more generally above IETs (interval exchange transformations). One can thus also view them as time changes of translation flows on surfaces. When the flow has fixed points, the ceiling function has singularities, that often appear  at the discontinuity points of the IET. 

 The study of conservative surface flows goes back to Poincar\'e, and it knew spectacular advances with the works of the Russian school starting from the beginning of the second half of last century till the early 90s. Recently, further substantial advances were made in their understanding  and they attracted a lot of attention due to their connections with billiards on rational polygons and Teichm\"{u}ller theory, as well as with parabolic dynamics such as the dynamics of horocycle flows and Ratner theory. 
  
Questions on the ergodic and spectral theory of conservative surface flows have a long history. 
The simplest setting to be examined is  that of smooth conservative flows {on the $2$-torus} without periodic orbits. This setting is reduced to that of reparametrizations (time changes) of minimal translation flows 
(see  for example the textbook \cite{Co-Fo-Si} by I.~P.~Cornfeld, S.~V.~Fomin  and Ya.~G.~Sinai).  A.~N.~Kolmogorov~\cite{kolmogorov} showed that such reparametrized flows are typically conjugated to translation flows, since it suffices for this that the slope of the translation flow belongs to the full measure set of  Diophantine numbers. 
He also observed that more exotic behaviors should be expected for the reparametrized flows in the case of Liouville slopes. M.~D.~Shklover indeed obtained in \cite{shklover} examples of real analytic reparametrizations of linear flows on the $2$-torus that were weak mixing (continuous spectrum). 
Not long after Shklover's result, A.~B.~Katok \cite{katok}, and later A.~V.~Kochergin~\cite{Koc1}, showed the absence of mixing for non-singular conservative flows on the $2$-torus, hence establishing that in a sense shear of nearby orbits near singularities is the only mixing mechanism available for smooth surface flows\footnote{This confirmed Kolmogorov's intuition about the absence of mixing for analytic reparametrizations of translation flows, but only in this two-dimensional setting. Indeed, mixing analytic reparametrizations of translation flows on $\T^3$ were obtained in \cite{Fa2}.}.
Note that analytic reparametrizations of Liouvillean irrational flows of the $2$-torus can have a mixed singular continuous and discrete maximal spectral type \cite{FKW,GP}.

\subsection*{ \bg Kochergin mixing flows on surfaces.} \blk The simplest mixing examples of conservative surface flows are those with one (degenerate) singularity on the $2$-torus produced by Kochergin in the 1970s \cite{Koc2}.  They are time changes of linear flows on the $2$-torus with an irrational slope and with a single rest point (see Figure 1 and the last section of this introduction for a precise definition of {\em Kochergin flows}).
Equivalently these flows can be viewed as special 
flows under a ceiling, or roof,  function with at least one power singularity (see Figures \ref{orbits} and \ref{sym} and the precise definition of special flows in Section \ref{sec.decay}). 
 
\begin{figure}[htb]
 \centering
  \resizebox{!}{5cm}{\includegraphics[angle=1]{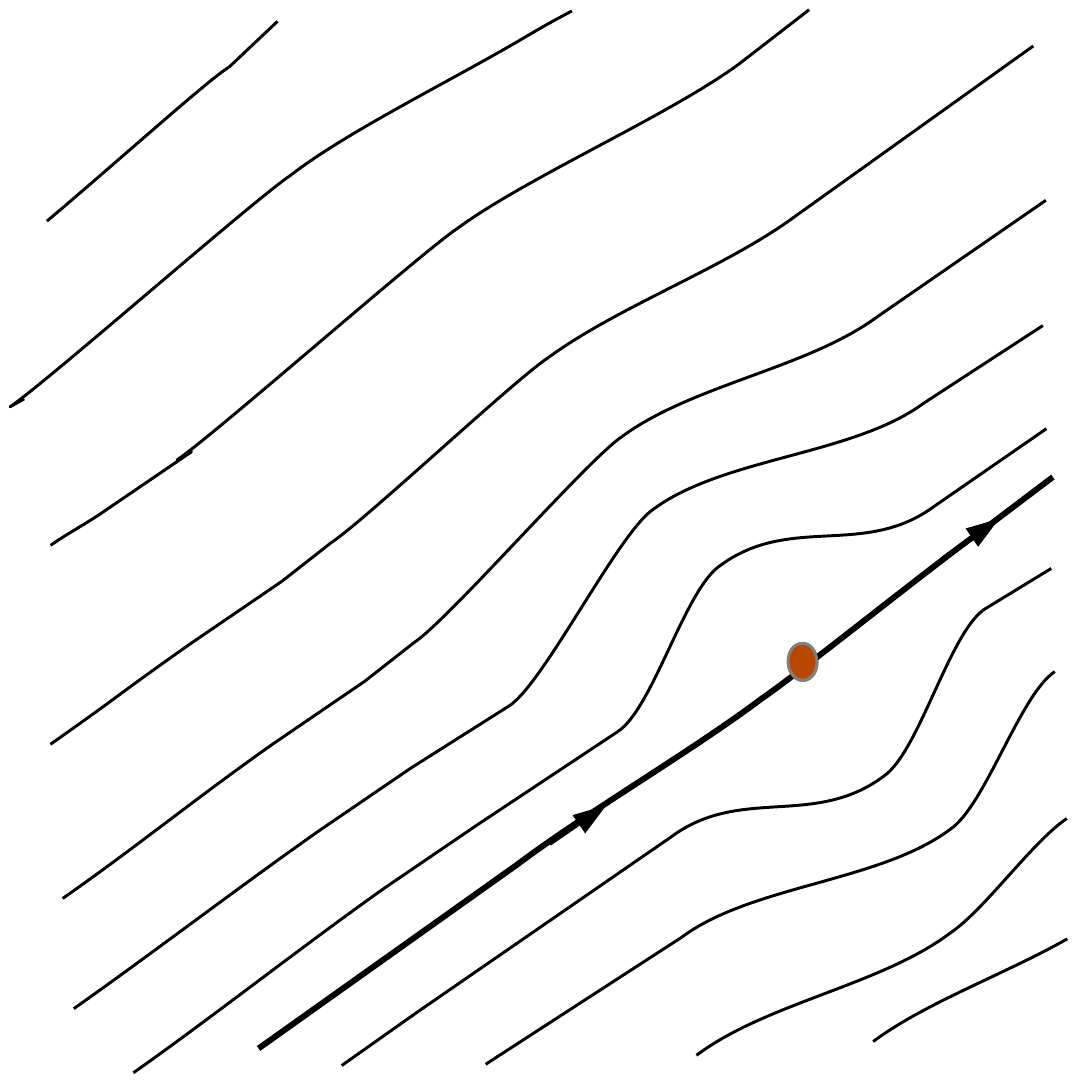}}
\caption{\small Torus flow with one degenerate saddle acting as a stopping point.} 
  \label{kochergin}
\end{figure}

Multi-valued Hamiltonian flows on higher genus surfaces can also be mixing (or mixing on an open ergodic component) in the presence of non-degenerate saddle type singularities that have some asymmetry  (see Figure 2).
Such flows are called {\em Arnol'd flows} and their mixing property, conjectured by V.~I.~Arnol'd in \cite{arnold},  was obtained by K.~Khanin and Ya. G.~Sinai \cite{KS} and, later, in more generality by Kochergin \cite{Koc4,Koc9}.  

Kochergin also proved  that for suspension flows under a roof function with {\it symmetric logarithmic} singularities over a circle rotation, mixing fails for almost every rotation number \cite{Koc3}. Many years later he proved in \cite{Koc8} that indeed mixing fails in this case for {\it all} rotation numbers. 

Ulcigrai substantially extended Kochergin's results by proving in \cite{CU3} that conservative flows with non-degenerate saddle singularities are generically not mixing (due to symmetry in the saddles). Recently, J.~Chaika and A.~Wright \cite{chaika} gave mixing examples with finitely many non-degenerate fixed points and no saddle connections on a closed surface of genus $5$.  

Although mixing was thoroughly studied for conservative surface flows, almost nothing was known about the spectral type {and spectral multiplicity} of the mixing examples (see for example the survey by Katok and 
J.-P.~Thouvenot \cite{KT} and the discussion therein). By spectral type of a flow $\{T^t\}$ we mean the spectral type of the associated Koopman operator $U_t : L^2(M,\mu)  : f \to f \circ T^t$. 

The nature of the spectral type {and multiplicity} of mixing surface flows naturally arose as soon as such mixing examples were obtained, especially since, at that time, the spectral theory of dynamical systems was a matter of major interest for the Russian school in that second half of last century (see for example the textbook \cite{Co-Fo-Si} or Kolmogorov's 1954 ICM address~\cite{Kicm}). Since then, the question about the possibility of a Lebesgue maximal spectral type for mixing surface flows appeared in many monographs and surveys (see for example the  discussions in \cite{KT}, \cite{LemEnc} or \cite{DF}). In the survey~\cite{KT}, by Katok and Thouvenot, it is remarked that ``Some estimate of correlation decay have been obtained but they are too weak to conclude that the spectrum is absolutely continuous.''  Finally,  Kochergin at the end of  his paper~\cite{Koc9}  asks about rate of mixing and absolutely continuous spectrum (Problem 4) and multiple mixing (Problem 6) for flows on surfaces.

In this paper we treat the simplest mixing examples that are Kochergin flows on the torus with a single degenerate rest point. 
 We give a general statement here that will be made more specific in the last section of this introduction. 

\begin{theorem} \label{statement} There exists a real analytic  conservative flow on $\T^2$ with exactly one singularity, with Lebesgue spectral type of countable multiplicity.
\end{theorem}

Note that besides their own interest, mixing conservative flows attracted an additional attention since they stood as the main and almost only natural class of mixing transformations for which higher order mixing has not been established, nor disproved.  The first and third author of this paper established multiple mixing only for a very special class of mixing Kochergin flows \cite{FK}. In our proof of Theorem \ref{statement} we will actually show that Kochergin flows with a sufficiently strong singularity have for almost every slope a countable Lebesgue spectrum. As a consequence, neither B.~Host's celebrated theorem that establishes multiple mixing for mixing systems with purely singular spectral type \cite{Ho}, neither \cite{FK}, can give a positive answer to the multiple mixing question for typical Kochergin flows.

Before we proceed to the precise statement of Theorem \ref{statement}, we make some comments about the new phenomenon that is enclosed in the above result and about the mechanisms that yield it.

\subsection*{\bg How chaotic can the lowest-dimensional, smooth, invertible dynamical systems be?} 

A circle diffeomorphism with irrational rotation number that preserves a smooth measure is smoothly conjugate to a rotation. It is hence rigid 
in the sense that the iterates along a subsequence of the integers converge uniformly to identity. Rigidity implies the absence of mixing between any two measurable observables. This absence of mixing actually holds for all smooth circle diffeomorphisms with irrational rotation number since, by Denjoy theory, they are topologically conjugated to rotations. Circle diffeomorphisms with rational rotation number are even farther from mixing, since their non-wandering dynamics are supported on periodic points.

The {\it lowest dimensional setting} that can be  investigated for dynamical complexity after circle diffeomorphisms is that of  multi-valued Hamiltonian flows on surfaces. In the absence of periodic orbits, these flows can be viewed as  reparametrizations of minimal translation flows on the torus. Combining Kolmogorov's result on the linearizability of Diophantine flows, and the theory of periodic approximations, A.~Katok~\cite{katok} proved that sufficiently smooth reparametrizations of linear flows on the torus are actually rigid. In particular, the maximal spectral type of smooth conservative flows of the torus without periodic orbits  is always purely singular.  

\begin{figure}[htb] 
 \centering
  \resizebox{!}{4cm}{\includegraphics[angle=0]{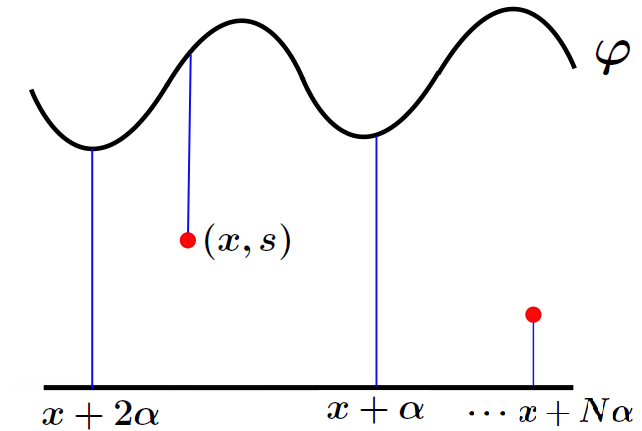}}  
\caption{\small The orbit of a point by the special flow above a rotation of angle $\a$ and under a bounded ceiling function $\varphi$. Smooth reparametrizations of linear flows on $\T^2$ are equivalent to such flows.}
\label{orbits}
\end{figure}

As a consequence of Katok's result, in order to go beyond the purely singular maximal spectral type for smooth conservative flows on the $2$-torus, one must allow the existence of singularities for the flow. {When there is just one singularity, the phase portrait  is actually similar to that of a minimal translation flow, apart from one orbit that contains the saddle point which acts as a stopping point  (see Figure \ref{kochergin}). Our result shows that in this situation the maximal spectral type can indeed be Lebesgue even in this `almost one-dimensional' setting. }
\medskip

\subsection*{\bg Quasi-minimal flows that are spectrally isomorphic to Bernoulli flows}

The two extremes in describing the stochasticity of a dynamical system from a spectral point of view are {simple} pure point spectrum on one end and {countable} Lebesgue spectrum on the other. {Translation flows on the torus have a simple pure point spectrum, while Bernoulli flows have countable Lebesgue spectrum.}

The proof of the countable Lebesgue spectrum property for geodesic flows on negatively curved surfaces by I.~M.~Gelfand and S.~V.~Fomin \cite{GF}, and later for general (open sets of) $K$-flows by A.~N.~Kolmogorov,  Ya.~G. ~Sinai and others, was considered as a major breakthrough
by the Russian school of dynamical systems in the second half of the twentieth century, because of what it implied on the similarities between some deterministic systems and stochastic flows \cite{sinai, Co-Fo-Si,kolmogorov3,pesinDSII, Kheritage}.

Parallel to this discovery was another major discovery made by Kolmogorov that quasi-periodic Diophantine motion is robust in many systems of mechanical origin (such as quasi-integrable Hamiltonian flows). This later developed into what is called today KAM theory (after Kolmogorov, Arnol'd and Moser). With these two phenomena in sight, stemming from the theory of K-systems and from KAM theory, Kolmogorov ventured in his ICM 1954 paper into the following interesting speculation : {\em It is not impossible that only these cases (a discrete spectrum with a finite number of independent frequencies and a countably-multiple Lebesgue
spectrum) are admissible for analytic transitive measures or that, in a sense,
only they alone are general typical cases.} This of course, is in stark contrast with the Halmos-Rokhlin general description of  invariant ergodic measures as being continuous and purely singular for the generic system in the weak topology, and Kolmogorov insisted in his speculation on restricting to the analytic setting, avoiding even the smooth category, for the above dichotomy to have some chances to hold. We know today, that even in the analytic category, and even in low dimensional systems such as reparametrized irrational flows of the $2$-torus, there are many other possibilities for the spectrum, including singular continuous, mixed, etc., but the validity of the dichotomy in some typical sense is still a possibility.

\medskip 
{ For systems with zero entropy,  for many decades progress on spectral questions was restricted to the  case of homogeneous flows, starting with O.~S.~Parasyuk's result \cite{Pa} on countable Lebesque spectrum for horocycle flows (see for example \cite{KlShSt} for a systematic exposition of ergodic theory of homogeneous flows and many references).  The only instance of smooth systems for which Lebesgue spectrum was established beyond hyperbolic and algebraic case is in \cite{FU} where the second author and Ulcigrai show that the maximal spectral type of smooth time-changes of the horocycle flow of a compact hyperbolic surface is Lebesgue.  Their work was motivated by a conjecture of A.~Katok and J.-P.~Thouvenot  that time-changes of horocycle flows should also have countable Lebesgue spectrum (see~\cite{KT},  Conjecture 6.8). Independently, R.~Tiedra \cite{Tie1}, \cite{Tie2} (following a different approach) obtained the absolute continuity of the spectrum for the same flows. }

Note that conservative flows on surfaces always have topological entropy zero\footnote{The situation is completely different for surface diffeomorphisms. Anosov automorphisms of the torus and their relatives constructed by A.~Katok~\cite{K-79} on the sphere and the disc are classical examples of conservative Bernoulli surface diffeomorphisms. Later, Bernoulli diffeomorphisms and flows were shown to exist on any compact manifold of dimension larger than $2$ and $3$ respectively~\cite{Dolgo-Pesin,HPT}. }.
Note also that conservative non-singular time changes of translation flows on the $2$-torus were presented by Kolmogorov~\cite{Kicm} as the basic context of real analytic systems in which discrete spectrum prevails.
In the particular case that we are considering of flows with just one singularity, the phase portrait is, as mentioned above, very similar to that of a minimal translation flow, except for the existence of one rest point. It is a striking fact that some of these quasi-minimal flows, as in Theorem~\ref{statement}, turn out to have a countable Lebesque spectrum and thus are
spectrally equivalent to ergodic Bernoulli flows.

 In the next two subsections we describe some aspects of these two steps as well as their relations to the existing literature. In a third subsection, we cast our results in a more general picture on conservative surface flows and describe their relations to the main recent advances in the field. In the last subsection of this Introduction, before we give the plan of the paper, we explain the shear mechanism that underlies mixing for conservative surface flows with singularities, and we precisely state our results .
 

\subsection*{\bg Non uniform shear and irregular decay of correlations} To prove the absolute continuity of the spectrum of a dynamical system, it is natural to look for a control on the decay of correlations by the flow.  The only result in the direction of getting power-like estimates for the decay of correlations of surface flows was obtained in~\cite{Fa}, where the first author proved a polynomial bound $t^{-\eta}$ on the decay of correlations (as functions of time $t>0$) for Kochergin flows with one power singularity and for the characteristic functions of rectangles. However, in that paper,  the power of the decay $\eta$ is bound to be less than $\frac{1}{4}$, so it is not possible to deduce from the decay anything about the spectral type of the corresponding flow.

However, and as it is often the case, characteristic functions of nice sets do not give the best rate of decay of correlations between observables.
 {Our work takes inspiration from that of \cite{FU}, especially in the use of coboundaries to estimate the decay of correlations, as well as in the proof of the equivalence to Lebesgue of the maximal spectral type. Our approach considerably refines the approach of \cite{FU} in two directions : 1) it handles non-uniformly parabolic flows, for which the correlation decay, even for coboundaries, is very irregular (not even bounded by $t^{-1/2}$), and 2) it gives a criterion for countable multiplicity, which applies to Kochergin flows, but also to a much wider class of mixing systems with square summable decay of correlations for a sufficiently rich class of functions. To prove the square summable decay of correlations for Kochergin flows, we also take inspiration from~\cite{Fa} where a quantitative approach to the mixing shear mechanism exhibited by Kochergin in~\cite{Koc2} is adopted to obtain a speed of mixing for these flows.}

Indeed, there is a 
{fundamental} difference between the decay of correlations for time-changes of horocycle flows and for mixing surface flows, that we will now explain. 

For time-changes of horocycle flows, the decay of correlations for coboundaries exploited in  \cite{FU}  is based on the {\it uniform shear} of geodesic arcs, linear with respect to time, as in B.~Marcus'  proof in \cite{Ma} of mixing for these flows. Such a shear can be readily derived from the commutation relations for the horocycle and the geodesic flows, and the unique ergodicity of the horocycle flow (and hence of all of its time-changes), first established by H.~Furstenberg \cite{Fu} (see also \cite{MaUE}). {The amount of shear is asymptotically linear with respect to time, since it is given by the ergodic integral of a function of non-zero mean}.

In the case of suspension flows above rotations, 
the shear of horizontal arcs is provided by the stretching of the Birkhoff sums of a ceiling function with a singularity (see Figure \ref{FigFlowMix}, {and the last subsection of this introduction for a precise description of the shear mechanism}). {In this case, the amount of shear is non-uniform since it is given by the ergodic integrals of a integrable function of zero mean, hence it depends on deviation of ergodic
integrals from the mean}. 
This {\it non-uniform shear} 
has a strength that depends on the asymptotics of the roof function at the singular point.  It is crucial
for our argument that the singularity be chosen strong enough so that, over most of the phase space, the 
{reciprocal} of the stretching is a square  integrable function of time. This means that our power singularity must be chosen with exponent in the interval
$(1/2, 1)$.  For {\it asymmetric} power singularities, the set where the {reciprocal} 
of the stretching of Birkhoff sums is not sufficiently small, that is,
not square-integrable, has very small measure and can be neglected in the argument. However, such suspension flows cannot be realized as smooth flows on a surface. For {\it symmetric} power singularities of exponent close to $1$, which indeed can be realized as smooth flows (see Remark \ref{rem.singularity} below), the set of insufficient stretching is not negligible anymore, and we have to deal with it in the argument. This is a significant difficulty, both conceptual and technical, and in fact the summability of the correlations even when their decay is not uniform, is a new phenomenon that, to the authors' best knowledge, does not arise in any of the proofs of absolutely continuous spectrum of dynamical systems available in the literature (see  \cite{FU} , \cite{Tie1}, \cite{Tie2}, \cite{Sim}). 

Indeed, we emphasize that in our situation, and in contrast with all the above-mentioned cases, in particular that of time-changes of horocycle flows investigated in  \cite{FU} ,  we have that for any smooth functions,  the correlation coefficients will not always be of order less than $t^{-1/2-\epsilon}$ as $t$ goes to infinity. To the contrary, along the subsequence $t_n$ given by the denominators of the irrational rotation, the correlation coefficients will in fact be as large as  $t_n^{-1/2+\epsilon}$, for some $\eps>0$, because there is a set of measure of order $t_n^{-1/2+\epsilon}$ on which the flow at time $t_n$ is almost equal to the identity. This {\em bad set} appears due to the cancellations in the stretching of the Birkhoff sums of the ceiling function that are caused by the symmetry at the singularity (a remnant of the Denjoy-Koksma property). The bad set is essentially a union of thin towers that follow in projection the orbit of the 
{rotation} on the base. Outside the bad set, the correlations are well controlled due to sufficiently strong uniform stretching. 
A crucial part of our argument, completely absent in the earlier works mentioned above, deals precisely with the bad set. Indeed, we use a bootstrap argument and the regular structure of the bad set, to show that for most of the times that are in a medium scale neighborhood of the time $t_n$, there is some {\it small power} decay of correlations on the bad set. This property, plus the smallness in measure of the bad set, plus the fast decay outside of this set, finally yield square summability of the total correlations (see Figures \ref{sym}, \ref{wykres2} and \ref{wykres1}). 

We think our method will be useful in treating other parabolic flows where mixing is due to shear and where the shear is often sufficiently strong but not uniformly in time and space.

\subsection*{\bg A criterion for countable Lebesgue multiplicity for parabolic flows} 

{Once we know that the maximal spectral type of Kochergin flows is absolutely continuous, two natural questions arise, one about the equivalence of the spectral type to Lebesgue measure on $\R$, and one  about the spectral multiplicity. 

The type and multiplicity of mixing  surface flows was often raised in connection with the question whether there exist flows with {\it simple} Lebesgue spectrum.
This is the flow version of the famous Banach's problem on the existence of a measure preserving transformation having simple Lebesgue spectrum. However, no tools were available to understand the multiplicity question for these flows. {{ A  criterion that gives an {\em upper bound} on the multiplicity of the spectrum of a flow, introduced by Katok and Thouvenot (\cite[Theorem 1.21]{KT}), does not apply to our Kochergin flows due to the strong shear near the singularity}. }

{We introduce here a geometric criterion based on rapid mixing that implies the pure Lebesgue and infinite multiplicity for flows that have an absolutely continuous maximal spectral type. It applies in particular to Kochergin flows with sufficiently degenerate power singularity and allows to complete the proof of Theorem \ref{statement}, building on the absolute continuity of the maximal spectral type, and on the estimates that implied it. 

The criterion, that we call CILS (Criterion for Infinite Lebesgue Spectrum), will be presented  in detail in Section 6.
Heuristically we see that if the flow admits a given number $n+1$ of functions, $n\geq 0$, such that each function is almost orthogonal to the cyclic space of any other one, and such that the spectral measures of the functions can be chosen to be not too small on any fixed bounded measurable set 
of $\R$, then the pure Lebesgue multiplicity of the flow is larger than $n+1$. In fact, in our formulation it is enough to construct $n+1$- functions such that the $(n+1)\times (n+1)$ matrix of Fourier transforms of their square-integrable mutual correlations has maximal rank equal to $n+1$ on any given positive measure subset of the real line.

The idea of constructing an arbitrarily large number of such independent functions for a rapidly mixing system is the following: one can choose  the functions to be supported on one or several  Rokhlin towers for the flow (or flow-boxes with an arbitrarily short base) and specify their values on these towers so that  the conditions of the criterion are satisfied for a finite, arbitrarily large time. Heuristically, such finite systems of functions are constructed to have  orthogonal cyclic subspaces on a large subinterval of the real line. Once more, it is in fact enough to control the Fourier transforms of all correlations of the functions in each finite system over a large time interval. The conditions of the criterion in the infinite complementary intervals are then derived from the mixing estimates, that is, from the square-integrability of the correlations (and their Fourier transforms).

Note that for $n=0$, only the condition on the spectral measure is required and yields the equivalence of the maximal spectral type to Lebesgue. Our criterion in that case reduces to the one used by the second author and Ulcigrai in~\cite{FU}  in the prof that  the maximal spectral type of smooth time-changes of horocycle flows is Lebesgue. Also, the construction of the function satisfying the criterion in that case $n=0$ is very similar to the construction in  \cite{FU}   but has to be adapted to our context of non-uniformly stretching flows.  


The main novelty in our CILS is the lower bound on the multiplicity.  Indeed, our CILS  gives an alternative to the much  stronger $K$-property introduced by Kolmogorov \cite{kolmogorov5}, Sinai \cite{sinai} and others to establish countable Lebesgue spectrum for uniformly hyperbolic systems. It is presented in a clear cut form that makes it applicable to a wide range of smooth mixing systems with a sufficiently fast rate of mixing for observables in some rich class of functions.  

Besides $K$-flows,  infinite Lebesgue spectrum was so far established only for homogeneous flows and other
systems of algebraic origin.
Even in one of the simplest non-algebraic cases, that of smooth time changes of horocycle flows, the {\it countable} Lebesgue spectrum property, conjectured, as we have recalled, by  Katok and Thouvenot (see~\cite{KT}, Conjecture 6.8), was still open. 


Our criterion allows to extend the work of the second author and Ulcigrai~ \cite{FU}  and  thereby complete the proof of the Katok-Thouvenot conjecture. However, the domain of applicability of our criterion is definitely wider than the class of {\em uniformly parabolic} flows (that is, flows with uniformly strong shear) such as horocycle flows and their time changes. Indeed we have applied it in this paper to the borderline case of mixing Kochergin flows, which are {\em non-uniformly parabolic}, with irregular decay of correlations, even for smooth coboundaries. 


{We believe that a systematic application of our CILS will allow to show that countable Lebesgue spectrum is a robust property in many non linear contexts, where many metric invariants (not just one metric invariant, like entropy in the case of K-systems) preclude the possibility of isomorphism classification.

\blk

\subsection*{\bg Other recent advances in the study of ergodic properties of surface flows}  



Further advances {in the ergodic theory of} flows on higher genus surfaces came only {in last couple of decades}
as a consequence of a deeper understanding of the behavior of ergodic sums (integrals) of Interval Exchange Transformations (translation flows), and several spectacular developments in that direction also brought renewed interest in  multi-valued Hamiltonian flows on surfaces. 

The second author studied {\em deviations of ergodic averages} for such flows~\cite{forni} and proved a substantial part of the conjectures formulated by M.~Kontsevich ~\cite{Kontsevich} and A.~Zorich \cite{Zorich1}, \cite{Zorich2}, \cite{Zorich3} on their deviation spectrum. {From these results, A.~Avila and the second author \cite{AF} derived the {\em weak mixing} property of non-toral translation flows and of Interval Exchange Transformations which are not rotations.} The proof of the Kontsevich--Zorich conjectures was later completed by A.~Avila and M.~Viana~\cite{AV}.  

Ergodic properties of multi-valued Hamiltonian flows on higher genus surfaces with non-degenerate saddle singularities  were then studied by C.~Ulcigrai, who established that such flows are generically {\em weak mixing}~\cite{CU2}, but {\em not mixing} \cite{CU3} (see also D.~Scheglov's paper \cite{scheglov}). 

For suspensions flows under ceiling functions with  {\it asymmetric logarithmic} singularities, Ulcigrai generalized in her thesis~\cite{CU1} the result of Khanin and Sinai \cite{KS} to suspensions with one singularity above generic IET's.  Only recently, D.~Ravotti \cite{ravotti} has carried out the argument for any number of singularities, thereby establishing mixing, with (at least) logarithmic decay of correlations, for  smooth flows of Arnol'd type on surfaces of higher genus. 

After Ulcigrai's work on {\em mixing properties}, a few major questions remained open in the ergodic theory
of flows on surfaces: whether mixing is at all possible for conservative smooth flows with non-degenerate saddles in higher genus, whether smooth flows on surfaces can have Lebesgue spectrum, and finally whether 
mixing implies multiple mixing (Rokhlin's question).  

As explained above, a better understanding of the various possible behaviors of IET's allowed J.~Chaika and A.~Wright \cite{chaika} to answer the first question in the affirmative. They proved the {\em existence of mixing  special flows} over non-generic, uniquely ergodic, interval exchange transformations,  in the case of a smooth ceiling function with symmetric logarithmic singularities at the interval endpoints.  Their result implies in particular
the existence of an (exceptional) mixing smooth flow with only Morse saddle singularities on a surface of genus $5$.

{Other important advances came from a better understanding of the similarities between the dynamics of uniformly parabolic flows, such as the horocycle flow, and locally Hamiltonian flows on surfaces.
The first and third author proved {\em multiple mixing} {for a class of such flows on the torus \cite{FK} (a restricted atypical class in the case of Kochergin flows, but a typical class for Arnol'd asymmetric flows)}. For this, they showed that these flows display a generalization of the so called {\rm Ratner property}  on slow divergence of nearby orbits (introduced by M.~Ratner \cite{ratner1}, \cite{ratner2}, \cite{ratner3}  in her study of ergodic properties of horocycyle flows), that implies strong restrictions on their joinings, which in turn yield higher order mixing. Multiple mixing was later generalized to many mixing flows on higher genus surfaces in~\cite{KKU}.  This was the first application of the Ratner property to prove multiple mixing outside its original context of horocycle flows.  Finally, a {\em disjointess criterion}, based on the Ratner property, has very recently been introduced  in \cite{KLU}, and systematically applied to disjointness results for time-changes of horocycles and Arnol'd flows (see also~\cite{FoK}~for an application of the criterion to Heisenberg nilflows and \cite{FlaFo} for a refinement of Ratner disjointness result for time-changes of horocycle flows).} As for the question on the spectral type and multiplicity of mixing surface flows, no results were known up to now.


\bigskip 

{In the next subsection, we describe  the mixing mechanism that comes from shear for surface flows with singularities, or for special flows above rotations.
We will also give the precise definition of the Kochergin flows for which we will establish the countable Lebesgue multiplicty. }

\subsection*{\bg Shear of Birkhoff sums and mixing}
Consider a section of a Kochergin's flow with one singularity on the torus, that is transversal to all orbits and does not contain the singularity. The dynamics can then be viewed as that of a special flow above an irrational  rotation of the circle with a return time
function (called a {\it ceiling} or {\it roof} function) having a power-like singularity (see Figure \ref{sym}). 
\begin{figure}[htb]
 \centering
   \resizebox{!}{4cm}{\includegraphics[angle=0]{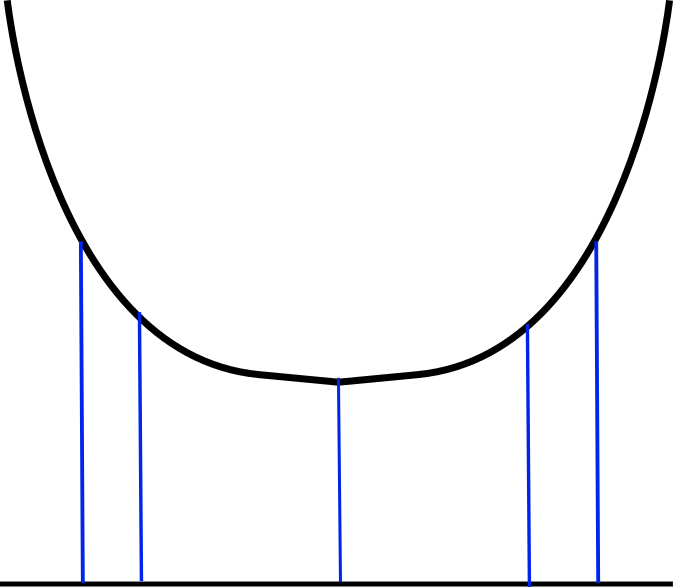}}
\caption{\small  Representation of a $2$-torus flow with one degenerate saddle  as a special flow under a ceiling function  (symmetric) power-like singularity.} 
  \label{sym}
\end{figure}

The singularity is precisely the last point where the section
intersects the incoming separatrix of the fixed point. The strength of the singularity depends on how abruptly the linear flow is slowed down in the neighborhood of the fixed point (see Remark \ref{rem.singularity}). 
 In the case of other surfaces and several singularities, the flows obtained by Kochergin are equivalent to special flows above interval exchange transformations (IET's) with ceiling functions having power-like singularities at the discontinuity points of the IET.

 The mechanism of mixing in Kochergin examples is, in part, the same as in the weak mixing examples of Shklover, namely the 
 stretching of the Birkhoff sums of the ceiling function above the iterates of the ergodic base dynamics. 
Whenever these sums are 
uniformly stretched above small intervals, the image of small rectangles by the special flow for large times decomposes into long and thin strips (see Figure \ref{FigFlowMix}). These strips are well distributed in the fibers due to uniform stretch, and well distributed in projection on the base because of ergodicity of the base dynamics.

\begin{figure}[htb]
 \centering
 \resizebox{!}{3.5cm}{\includegraphics{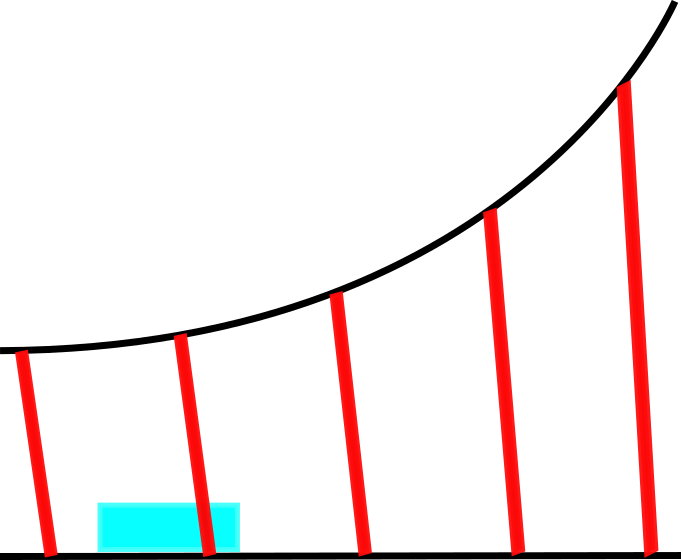}}
\caption{\small Mixing mechanism for special flows: the image of a rectangle is
a union of long narrow strips which fill densely the phase space.} 

\label{FigFlowMix}
\end{figure}

However, the reason behind the uniform stretching is different for Shklover's flows and Kochergin's ones. For the first ones, uniform stretching of the  Birkhoff sums of the ceiling function is due to a Liouville phenomenon of accumulation, along a subsequence of time, of the oscillations of the ceiling function due to periodic approximations. In the case of Kochergin's flows, it is the shear between orbits as they get near the fixed points that is responsible for mixing.
As a consequence, for the latter uniform stretching holds for {\it all} large  times, while for the former, the existence of Denjoy-Koksma (DK for short) times impedes mixing. Denjoy-Koksma times are integers for which the Birkhoff sums have an {\it a priori} bounded oscillation around the mean value on all or on a positive measure proportion of the base (see for example  the discussion around property DK in \cite{DF}).  Hence, a key fact behind Kochergin's result is that the Denjoy-Koksma property does not necessarily hold for ceiling functions having infinite asymptotic values at some singularities.

 A threshold is given by smooth ceiling functions having {\it logarithmic} singularities. When such a singularity is {\it symmetric}, it is known that for a typical irrational rotation a Denjoy-Koksma like property holds that prevents mixing of the special flow (see \cite{Lem1} and \cite[Section 8]{DF}).  In higher genus, Ulcigrai \cite{CU3}
 proved that, despite the presence of polynomial deviations of Birkhoff sums from the mean \cite{Zorich3}, \cite{forni}, for almost all IET's there are still sufficient cancellations to prevent mixing.  A different, special, cancellation mechanism was found slightly earlier by Scheglov \cite{scheglov} in genus $2$.  However,  as proven by Chaika and Wright \cite{chaika}, these cancellations do not happen for all IET's, as the speed of convergence of Birkhoff sums to the mean can be very slow, and this is why mixing is possible in some special cases.

The case of {\it asymmetric logarithmic} singularities is different.  In~\cite{arnold}, Arnol'd showed that multi-valued Hamiltonian flows with non-degenerate saddle points have a phase portrait that decomposes into elliptic islands (topological disks bounded by saddle connections 
and filled up by periodic orbits) and one open uniquely ergodic component. On this component, the flow can be represented as the special flow over an interval exchange map  of the circle and under a ceiling function that is smooth except for some logarithmic singularities. The singularities are typically asymmetric since the coefficient in front of the logarithm is twice as big on one side of the singularity as the one on the other side, due to the existence of homoclinic loops (see Figure  \ref{FigSaddle}).  As we mentioned above, Khanin and Sinai~\cite{KS}
proved that, as conjectured by  Arnol'd, this asymmetry produces mixing.

\begin{figure}[htb]
 \centering
\resizebox{!}{4cm}{\includegraphics{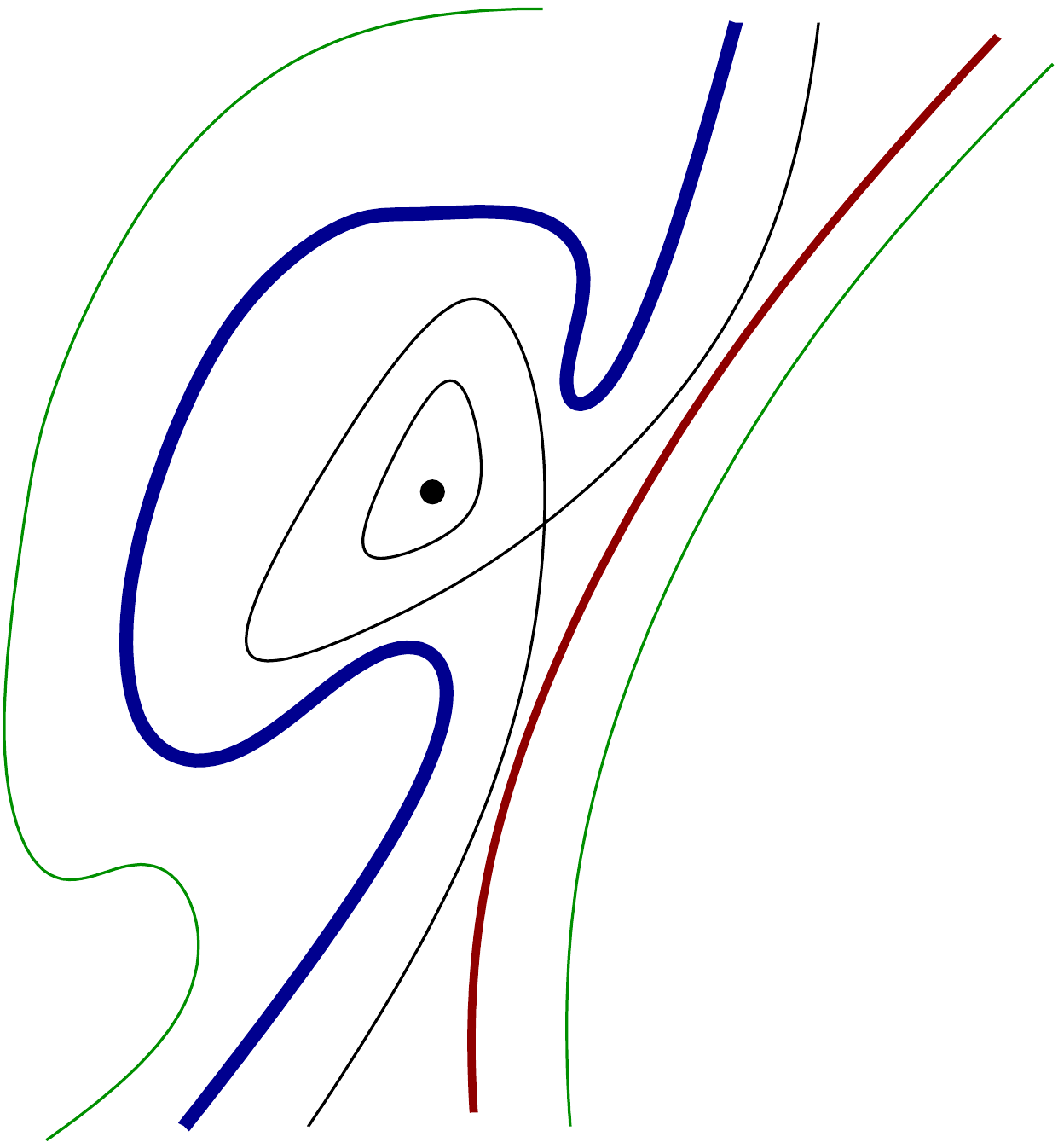}}
  \caption{\small Multivalued Hamiltonian flow. Note that the orbits passing to the left of the saddle 
spend approximately twice longer time comparing 
to the orbits passing to the right of the saddle and starting at the same distance from the separatrix since they pass near the saddle twice.}
 \label{FigSaddle}
\end{figure}

In this paper, we will show that Kochergin flows with a single sufficiently strong degenerate singularity typically have a Lebesgue spectral type with countable multiplicity. We now formulate our results more precisely. The flows which we will consider are special flows given by a base dynamics that is an  irrational rotation by $\a \in \T$, and a ceiling  function $\vphi \in C^2(\T\setminus\{0\})$, $\vphi>0$, with the following properties:
\begin{equation}\label{asu}\lim_{\theta\to 0^+}\frac{\vphi(\theta)}{\theta^{-(1-\eta)}}=M_1\quad \text{   and   }\quad\lim_{\theta\to 0^-}
\frac{\vphi(\theta)}{\theta^{-(1-\eta)}}=M_1
\end{equation}

\begin{equation}\label{asu2}\lim_{\theta \to 0^+}\frac{\vphi'(\theta)}{\theta^{-(2-\eta)}}=-N_1\quad \text{   and   }
\quad \lim_{\theta\to 0^-}\frac{\vphi'(\theta)}{\theta^{-(2-\eta)}}=N_1
\end{equation}

\begin{equation}\label{asu3}\lim_{\theta\to 0^+}\frac{\vphi''(\theta)}{\theta^{-(3-\eta)}}=R_1\quad \text{   and   }
\quad \lim_{\theta\to 0^-}\frac{\vphi''(\theta)}{\theta^{-(3-\eta)}}=R_1
\end{equation}
where $\eta$ is a small number,  $\eta\in(0,\frac{1}{1000})$, and $+\infty> M_1,N_1,R_1>0$. We refer to the beginning of Section \ref{sec.decay} for an exact definition of special flows. We assume that $\int_{\T}\vphi(\theta) d\theta=1$. We let $M=\{(\theta,s)\in\T\times\R\;:\; s\leq \vphi(\theta)\}$ and denote by $\mu$ the measure  equal to the restriction to $M$ of the product of the Haar measures $\lambda_\T$ on the circle $\T$ and $\lambda_\R$ on the real line~$\R$. This measure is the unique invariant measure for the special flow $T^t_{\alpha,\vphi}$ given by $(\alpha, \vphi)$. Our main result is the following. For $\xi>0$, we will say that $\a \in D_{\log,\xi}$ if and only if  there exists a constant $C(\a)>0$ such that for any $p \in \Z, q \in \Z^*$, 
$$ |\a-\frac{p}{q}| \geq \frac{C}{q^{2}\log^{1+\xi}q}.$$ 
It is a classical and easy to prove fact that for any $\xi>0$, $D_{\log,\xi}$ has full Haar measure in $\T$. 

\begin{theorem}\label{main} For $\a \in D_{\log,\xi}$, $\xi<\frac{1}{10}$, the dynamical system $(T^t_{\alpha,\vphi}, M, \mu)$ has Lebesgue spectral type with countable multiplicity. \end{theorem}

\begin{remark} \label{rem.singularity}  {\rm In \cite{Koc2}, the following method is adopted to obtain conservative flows on the torus with a degenerate saddle-node fixed point as in  \eqref{asu}--\eqref{asu3}. Consider first some Hamiltonian flow on $\R^2$ with the $x-$axis invariant and with a unique singularity at the origin. In the neighborhood of the origin, the orbits of such a flow are as described in Figure \ref{kochergin}. It is then possible to cut a small neighborhood of the origin and paste it smoothly inside the phase portrait of a linear flow of $\T^2$ with any given slope. As a result, one gets a multi-valued Hamiltonian flow that has a unique singularity of saddle-node type.  An easy calculation shows that if we consider the Hamiltonian given by $H_l(x,y)=y(x^2+y^2)^l$ then the corresponding special flow has a unique symmetric power-like singularity as in \eqref{asu}--\eqref{asu3} with $\eta$ arbitrarily close to $0$ as $l \to \infty$. 

One can also obtain analytic examples with one fixed point as in \eqref{asu}--\eqref{asu3}. To do so, one starts with the smooth construction of a multi-valued Hamiltonian described above.  Then, for an arbitrary $k>2l+4$, one considers a real analytic approximation of the smooth multi-valued Hamiltonian that continues to have the same slope and a unique singularity at $(0,0)$, with the same jets of order $k$ at $(0,0)$ (that is, those of $H_l$). From there it follows that the corresponding flow has a special flow representation with a ceiling function having a unique symmetric power-like singularity as in \eqref{asu}--\eqref{asu3}.  
}
\end{remark}

\medskip
We end this introduction with some of the questions that naturally arise from our result.

\begin{question} Do Kochergin flows always have Lebesgue spectral type (with countable multiplicity)?
\end{question}
To answer this question, one has to treat several singularities  and with smaller powers as well as general IET's on the base. 

\begin{question} What is the spectral type in the case of non degenerate saddles?
\end{question}

Arnol'd conjectured a power-like decay of correlation in the asymmetric case, but the decay is more likely to be logarithmic, at least between general regular observables or characteristic functions of regular sets such as balls or squares. Note that even a lower bound on the decay of correlations is not sufficient to preclude absolute continuity of the maximal spectral type. However, an approach based on slowly coalescent periodic approximations as in~\cite{Fa3} may be explored in the aim of proving that the spectrum is purely singular.

\bigskip 

\subsection*{\bg Plan of the paper} 

In Section \ref{sec.decay} we first give the formal definition of our special flows and we  describe the set of coboundary functions we will be interested in. 

The proof that the flow $T^t_{\a,\vphi}$ has an absolutely continuous maximal spectral type follows by a standard argument from Theorem \ref{cob} that states that the Fourier transforms of the spectral measures of  functions in our special dense set are square-integrable.

The proof of Theorem \ref{cob} splits in two parts. We consider a time $t \in [q_n,q_{n+1}]$ for some $n \in \N$. We further consider intervals of time of the type $t \in [l^{21/20},(l+1)^{21/20}]\subset [q_n,q_{n+1}]$.

First, a decay faster than $t^{-1/2-\eps}$ for some $\eps>0$ is established outside a bad set $\mathcal{B}_l$ of measure comparable to $t^{-1/2+\eps}$. This result is stated as Proposition  \ref{tdec}. Second, the squared correlations on the bad set $\mathcal{B}_l$ are controlled on average for $t \in [l^{21/20},(l+1)^{21/20}]$.  This is the content of Proposition \ref{bn}.

Section \ref{est.birk} is devoted to the proof of general stretching estimates for the Birkhoff sums of the ceiling function. 

In Section \ref{sec.partition} the bad set $\mathcal{B}_l$  is  constructed and the stretching properties outside this set are stated. This is the content of Propositions \ref{prop.badset}, \ref{good.dec} and~\ref{pr1}.

Section \ref{sec.correlations} explains the derivation of correlation decay estimates from  uniform stretching of Birkhoff sums. 
The main results of Section \ref{sec.uniformstretchandcorrelations} are Corollary \ref{correlation.good} that describes the fast decay of order  at least $t^{-1/2-\eps}$ on the good intervals that partition the complement of the bad set, and Corollary  \ref{alleb}  that describes the decay of order  $t^{-1/2+\eps}$ on general intervals (with the bad set $\mathcal{B}_l$ included). Corollary \ref{correlation.good} will directly yield the proof of Proposition \ref{tdec} on fast decay outside $\mathcal{B}_l$, given in Section \ref{Proof_tdec}, while Corollary  \ref{alleb} is crucial in the bootstrap argument that yields the averaged decay on the set $\mathcal{B}_l$ of Proposition \ref{bn}, given in Section \ref{Proof_bn}.

Finally, in Section \ref{subsec:CLS}, we complete the proof of Theorem \ref{main} and prove that the spectral type of Kochergin flows is Lebesgue with countable multiplicity. The proof that the spectral type is not just absolutely continuous, but indeed equivalent to the Lebesgue measure, is based on a new criterion for countable Lebesgue spectrum of smooth flows (Theorem \ref{thm:CLS}) and on the construction of an arbitrary number of observables, localized on an arbitrarily long flow-box, which have given arbitrary correlation functions on a finite, but arbitrary long, time interval. The control of the correlation functions beyond this time is guaranteed by the estimates on correlation decay obtained in Sections \ref{sec.partition} and \ref{sec.correlations}. 

The outline of the construction of the observables comes from the proof of the Lebesgue maximal spectral type for time changes of horocycle flows  \cite{FU} .  However, again in contrast with the case of time-changes of horocycle flows, whose phase space has dimension $3$, for this approach to work in the case of surface flows, that is, in dimension $2$, it is crucial that the constant in the estimates on the square integrals of correlations satisfy good bounds in terms of the smooth norms of the functions. For this reason, we will estimate carefully this dependence throughout the paper.

\bg \section{Special flows,  smooth coboundaries, and  decay of correlations} \label{sec.decay}
\blk 
 Let $R_\a:\T\to\T$, $R_\a(\theta)=\theta+\a \; \text{\rm mod }1$, where $\a\in \T$ is an irrational number with the sequence of denominators $(q_n)_{n=1}^{+\infty}$ and let $\psi\in L^1(\T,\mathcal{B},\lambda_\T)$ be a strictly positive function. We denote by $d_\T$ the distance on the circle. We recall that the special flow $T^t:=T^t_{\a,\psi}$ constructed above $R_\a$ and under $\psi$ is given by 
\begin{eqnarray*}
\T \times \R / \sim  &  \rightarrow &  \T \times \R / \sim  \\
 (\theta,s) & \rightarrow & (\theta,s+t), \end{eqnarray*}
where $\sim$ is the
identification 
\begin{equation}
\label{FlowSpace}
(\theta, s + \psi(\theta)) \sim (R_\a(\theta),s) \,.
\end{equation}
Equivalently (see Figure \ref{orbits}), this special flow  is defined for  $t+s \geq 0$ (with a similar definition for negative times) by 
$$T^t(\theta,s) = (\theta+N(\theta,s,t)\a , t+s-   \psi_{N(\theta,s,t)} (\theta)),$$ 
   where $N(\theta,s,t)$ is the unique integer such that 
\begin{equation}
\label{D-C}
0 \leq  t+s-   \psi_{N(\theta,s,t)} (\theta) \leq \psi(\theta+N(\theta,s,t)\a),\end{equation}    
and
$$
\psi_n(\theta)=\left\{\begin{array}{ccc}
\psi(\theta)+\ldots+\psi(R_\alpha^{n-1}\theta) &\mbox{if} & n>0\\
0&\mbox{if}& n=0\\
-(\psi(R_\alpha^n\theta)+\ldots+\psi(R_\alpha^{-1}\theta))&\mbox{if} &n<0.\end{array}\right.$$

Let $M$ denote the configuration space, that is, 
$$M:= \{(\theta,s)\in \T \times \R \;: s \leq \psi(\theta) \}.$$
In our case $\psi=\vphi$, where $\vphi$ has the properties stated in formulas~\eqref{asu}, \eqref{asu2} and~\eqref{asu3}. For a given $\zeta>0$, let us denote
\begin{equation}
\label{Mzeta}
M_\zeta:=\{(\theta,s)\in M\;:\; d_\T(\theta,0) >\zeta, \zeta<s<\vphi(\theta)-\zeta\}.
\end{equation}
We recall that $f$ is a smooth coboundary for the flow $T^t_{\a,\vphi}$ if there exists a smooth function 
$\phi$ such that, for any $a<b$, 
$$\int_a^b f(u,t) dt = \phi(u,b)-\phi(u,a).$$
 The space of smooth coboundaries is dense in the subspace $L^2_0(M) \subset L^2(M)$ of zero average functions, provided $T^t_{\a,\vphi}$ is ergodic (which is always the case if $\a$ is irrational). Moreover, it can be shown that the subspace $\cF$ of the space of all smooth coboundaries defined by the conditions that $f\in \cF$ if and only if $f$ is a smooth coboundary and there exists $\zeta>0$ with $f(x)=0$ for every $x\in M^c_\zeta$, is also dense in $L^2_0(M)$. Indeed, let us prove that the orthogonal space $\cF^\perp \subset L^{2}_0(M)$ contains only the zero function.
 In fact, every function $f\in L^2_0(M)$, which belongs to the orthogonal space $\cF^\perp \subset L^{2}_0(M)$, is by definition orthogonal to the Lie derivative along the flow of every smooth function with support contained in $M_\zeta$ for some $\zeta>0$. It follows that for every $t>0$ the function $f\circ T^t_{\alpha, \vphi} -f$ is orthogonal to all smooth functions with support in $M_\zeta$,
 for every $\zeta>0$,  hence it is orthogonal to all square-integrable functions, as the space of smooth functions with support contained in 
 $M_\zeta$ for some $\zeta>0$ is dense in $L^2(M)$. It follows that for any $t>0$, the function $f\circ T^t_{\alpha, \vphi} -f$ vanishes, hence 
 $f$ is invariant and constant by the ergodicity of the flow. As $f$ has zero average, it is equal to the zero function.  

\smallskip
Let $f\in \cF$ be a smooth coboundary and $g\in C^1(M)$. By definition, since $f\in \cF$, there exists $\zeta>0$ such that $f=0$ on $M^c_\zeta$.

\begin{theorem}\label{cob} Let $f$ be a smooth coboundary  for the flow $(T^t_{\alpha,\vphi})$ and let $g$ be a smooth function on $M$, both vanishing on some neighborhood of the boundary of $M$. Then the correlation function

\begin{equation}\label{inte}
\mathcal C_{f,g}(t) :=  \int_{M}f(T^t_{\alpha,\vphi}(x))g(x)d \mu \,, \quad \text{ for all } t>0 \,,
\end{equation}
belongs to the space $L^2(\R,d \lambda_\R)$ of square-integrable functions on the real line.
\end{theorem} 

The symbols $C_{f,g}$, $C'_{f,g}$, $C''_{f,g}$ will denote 
positive constants depending only on the $C^1$ norms of $f\in \cF$ and $g\in C^1(M)$ and on the $C^1$ norm of the transfer function $\phi$ for $f\in \cF$. Theorem \ref{cob} immediately follows from 
\begin{theorem}\label{decl} 
For every $f\in \cF$ and $g\in C^1_0(M_\zeta)$ there exists a constant $C_{f,g}>0$ such for all $l\in \N$,  we have
$$
\int_{l^{21/20}}^{(l+1)^{21/20}}\left\vert
\int_Mf(T^t_{\a,\vphi}(x))g(x)d\mu\right\vert^2 dt<  C_{f,g} \,l^{-1-\frac{\eta}{100}}.
$$
\end{theorem}

For simplicity, we denote $l_0=l^{21/20}$, $l_1=(l+1)^{21/20}$. Let $n\in \N$ be unique such that 
$$q_n<l_0<q_{n+1}\,.$$

 Theorem \ref{decl} can be derived from the propositions stated below.  
 
 \begin{proposition}\label{tdec} There exists a set $\mathcal B_l\subset M$, $\mu(\mathcal B_l)<q_n^{-1/2+6\eta}$ such that for every $t\in [l_0,l_1]$, we have
$$
\left|\int_{M\setminus \mathcal B_l}f(T^t_{\a,\vphi}(x))g(x)d\mu\right|<            C_{f,g}  \,t^{-1/2-\frac{\eta}{6}}.
$$
\end{proposition}
\begin{proposition}\label{bn}
We have 
$$
\int_{l_0}^{l_1}\left|\int_{\mathcal B_l}f(T^t_{\a,\vphi}(x))g(x)d\mu\right|dt< C_{f,g}\, \frac{(l_1-l_0)\mu(\mathcal B_l)}{q_n^{20\eta}}.
$$
\end{proposition}

The proofs of the two above propositions will be given later, in Sections~\ref{Proof_tdec} and~\ref{Proof_bn}, respectively. Let us show how they imply Theorem \ref{decl}, and therefore Theorem \ref{cob} and  the first part of Theorem \ref{main} on the absolute continuity of the spectrum.
\begin{proof}[Proof of Theorem \ref{decl}.]
{
Let $F(t):=\frac{1}{\mathcal B_l}\left|\int_{\mathcal B_l}f(T^t_{\a,\vphi}(x))g(x)d\mu\right|$ and let $G_l:=\{t\in [l_0,l_1]\;:\; F(t)\geq \frac{1}{q_n^{7\eta}}\}$. By Markov's inequality and  Proposition \ref{bn},  we have
$$
\left|G_l\right|\leq C_{f,g} \frac{l_1-l_0}{q_n^{13\eta}}.
$$
By splitting the integration below into $G_l$ and $G_l^c$, we get} 
\begin{multline*}\int_{l_0}^{l_1}\left\vert \int_{\mathcal B_l}f(T^t_{\a,\vphi}(x))g(x)d\mu \right\vert^2d t\leq C_{f,g}  \frac{(l_1-l_0)\mu(\mathcal B_l)^2}{q_n^{13\eta}}\\
\leq 
C_{f,g} \frac{(l_1-l_0)}{q_n^{1+\eta}}\leq  C'_{f,g} \frac{(l_1-l_0)}{q_{n+1}^{1+\eta/2}}\leq  C'_{f,g} \frac{(l_1-l_0)}{l_0^{1+\eta/2}} \\ \leq 
2C'_{f,g}\frac{l^{1/20}}{l^{21/20(1+\eta/2)}}< C''_{f,g} l^{-1-\frac{\eta}{2}}.
\end{multline*}

Using this and Proposition \ref{tdec}, we have 
\begin{multline*}
\int_{l^{21/20}}^{(l+1)^{21/20}}\left\vert
\int_Mf(T^t_{\a,\vphi}(x))g(x)d\mu\right\vert^2 dt \leq
 2 \int_{l_0}^{l_1}t^{-1-\frac{\eta}{5}} dt+\\
2 \int_{l_0}^{l_1}\left\vert\int_{\mathcal B_l}f(T^t_{\a,\vphi}(x))g(x)d\mu\right \vert^2d t\leq l^{-1-\frac{\eta}{10}},
\end{multline*}
which finishes the proof of Theorem \ref{decl}.
\end{proof}

\bg 
\section{Stretching of Birkhoff sums}
\label{est.birk} 
\blk

We collect in the section the necessary technical facts about the Birkhoff sums of the ceiling function $\vphi$ above $R_\a$. Some proofs 
that are not difficult, but probably a bit tedious, will be deferred to the Appendix \ref{app.birkhoff}. 

\smallskip
For simplicity, we will assume that in our main assumptions \eqref{asu}, \eqref{asu2}, \eqref{asu3} we have $M_1,N_1,R_1=1$ and that 
$\int_\T \vphi d \lambda_\T=1$. 
Throughout this section we suppose fixed $l_0=l^{21/20}$, $l_1=(l+1)^{21/20}$ and the unique integer $n$ such that $q_n<l_0<q_{n+1}$.

For every $x\in M$ we will  denote by $\bar x\in \T$ its first coordinate.  In particular, for any $t\in \R$, we will denote the first coordinate of 
$T^t_{\a,\vphi}(x) \in M$ by  $\bar T^t_{\a,\vphi}(x)$. Similarly, for any horizontal interval $I \subset M$, we will denote $\bar I \subset \T$ its vertical projection and by $\lambda (I)$ its (horizontal) Lebesgue measure, that
is, the Lebesgue measure $\lambda_\T( \bar I)$.

 Let $q_k\in [q_n\log ^{15} q_n, q_n\log^{20} q_n]$ (such $q_k$ exists by the Diophantine assumptions on $\a$) and consider the partition $\cI_k$ of $\T$ into intervals with endpoints $\{-i\a\}_{i=0}^{q_k-1}$. For any $\bar I \in \cI_k$ such that  $\bar I\cap [-\frac{1}{q_n^{3/5}},\frac{1}{q_n^{3/5}}]=\emptyset$, let $I_\vphi:=\{(\theta,s)\in M\,:\; \theta \in \bar I, 0\leq s\leq \min_{\theta\in \bar{I}}\vphi(\theta)\}$. Define 
\begin{equation}\label{defw}W:= \bigcup \{I_\vphi\;:\; \bar I \in \cI_k, \bar I\cap \left[-\frac{1}{q_n^{3/5}},\frac{1}{q_n^{3/5}}\right]=\emptyset\}.
\end{equation}
By a slight abuse of notations, we refer to $W$ as a set as well as a partial partition of $M$ into intervals. 
Define moreover
\begin{equation}\label{def:v}
V:=\{(\theta,s) \in M \;:\; 0\leq s \leq q_n^{3/5+1/10}\}.
\end{equation}

Notice that $M_\zeta\subset W$. 

 Notice that since $t\leq l_1\leq q_{n+2}$ and $\vphi>c>0$, we have
 $$
 N_t:=\sup_{x \in M}N(x,t)\leq \frac{q_{n+2}}{c}\ll q_k.
 $$ 
 Hence by the definition of the partition $\cI_k$, for every $I \subset W$
\begin{equation}\label{dif}
0\notin \bigcup_{i=0}^{N_t} R_\alpha^i(\bar I).
\end{equation}
As a consequence of \eqref{dif} the Birkhoff sum $\vphi_{N(x,t)}$ is (twice) differentiable on $I$, for every $x\in I$ and $t\leq l_1$. This fact will be used repeatedly in the proofs.

\subsection{Denjoy-Koksma estimates}

We start with some Denjoy-Koksma type estimates that allow us to give some control on the Birkhoff sums of $\vphi$ in function of the closest visit to the singularity.  

We will  adopt the following notation: for any $x\in M$ and $N\in \N$, we let
$$
x_{min}^N=\min_{0\leq j<N}d(\bar x+j\a,0).
$$

\begin{lemma}\label{koksi} For every $x\in M$ and every $ N\in [q_r,q_{r+1}]$,  we have 
\begin{equation}\label{koks0}
\vphi\left(x^N_{min}\right)+\frac{1}{3}q_{r}\leq \vphi_N(\bar x) \leq \vphi\left(x^N_{min}\right)+3q_{r+1}
\end{equation} 
\begin{equation}\label{koks1}
\vphi'\left(x^N_{min}\right)-8q_{r+1}^{2-\eta}<
|\vphi'_N(\bar x)|<\vphi'\left(x^N_{min}\right)+8q_{r+1}^{2-\eta}
\end{equation}
and
\begin{equation}\label{koks2}
\vphi''\left(x^N_{min}\right)\leq \vphi''_N(\bar x)<
\vphi''\left(x^N_{min}\right)+8q_{r+1}^{3-\eta}\,.
\end{equation}
\end{lemma}

\begin{proof}[Proof of Lemma \ref{koksi}] We will give the proof of \eqref{koks0}, the proofs of \eqref{koks1} and \eqref{koks2} are analogous. Let $\chi_{r}$ denote the characteristic
function of the interval $[-\frac{1}{3q_{r}},\frac{1}{3q_{r}}]$ and define $\bar{\vphi}_r:=(1-\chi_{r})\vphi$. By Denjoy-Koksma inequality, since $\int_\T\vphi d \lambda_\T=1$, we have
$$
\left(\bar{\vphi}_{r+1}\right)_{q_{r+1}}(\bar x)\leq q_{r+1}\int_\T\bar{\vphi}_{r+1}d\lambda_\T+ 4q_{r+1}^{1-\eta}\leq 3q_{r+1}.
$$
Therefore
$$
\vphi_N(\bar x)\leq \vphi_{q_{r+1}}(\bar x)\leq \vphi(x^N_{min})+\left(\bar{\vphi}_{r+1}\right)_{q_{r+1}}(\bar x)\leq \vphi(x^N_{min})+3q_{r+1}.
$$
This gives the upper bound. Analogously (by Denjoy-Koksma inequality for $\bar{\vphi}_r$), we get the lower bound. The proof is thus finished.
\end{proof}

The following lemma is a direct consequence of \eqref{koks0} and \eqref{koks1}, \eqref{koks2}.
\begin{lemma}\label{fi} For every $x\in M$ and $N\in \N$
\begin{equation}\label{ele}
|\vphi'_N(\bar x)|<(\vphi_N(\bar x))^{2+2\eta},
\end{equation}
\begin{equation}\label{ele2}
|\vphi''_N(\bar x)|>(\vphi_N(\bar x))^{3-\eta}\log^{-3}N
\end{equation}
As a consequence, we have that for every $x\in  M \cap (\T \times\{s\})$ and every $t\in \R$
\begin{equation}\label{con:ele}
|\vphi'_{N( x,t)}(\bar x)|<3s^{2+2\eta}+3t^{2+2\eta}
\end{equation}
and
\begin{equation}\label{con:ele2}
|\vphi''_{N( x,t)}(\bar x)|>(t+s-\vphi(\bar x+N( x,t)\a))^{3-\eta}\log^{-3}N( x,t).
\end{equation}
\end{lemma}

We have also the following bound on the discrepancies of the base rotation relative to intervals. 

\begin{lemma}\label{koksma2} Let $\bar J\subset \T$ be an interval. Then for every $N\in \N$ and every $\theta \in \T$
$$
|(\chi_{\bar J})_N(\theta)- N\lambda (J)|\leq 2C^{-1} \log^{2+\xi} N\,.
$$
\end{lemma}
\begin{proof} Notice that by Denjoy-Koksma inequality, for every $j\in \N$ and $\theta\in \T$, we have
\begin{equation}\label{coj}|(\chi_{\bar J})_{q_j}(\theta)- q_j\lambda (J)|\leq 2.
\end{equation}
To conclude, we write $N=\sum_{j=0}^ra_jq_j$, where $0\leq a_j\leq\frac{q_{j+1}}{q_j}$ (it is called Ostrowski expansion of $N$) use the cocycle identity, the bound in \eqref{coj} for $j=r,r-1,\dots,0$ and the fact that by our Diophantine condition $a_j\leq C^{-1} (\log q_j)^{1+\xi}$ for all $j\in \N$.
\end{proof}

\subsection{Stretching estimates}

Uniform stretching of the Birkhoff sums requires a lower bound on the derivatives of the Birkhoff sums and an upper bound on their second derivatives (see for example Definition \ref{int.good} below).  For any interval $I \subset W$, we therefore introduce the notation 
\begin{equation}\label{ui}u_I:=\sup_{t\in[l_0,l_1]}\sup_{x\in I}|\vphi ''_{N( x,t)}(\bar x)|.
\end{equation}

\begin{lemma}\label{sec.fir} Let $I\subset W.$ If $u_I\geq q_n\log^9q_n$, then for every $t\in [l_0,l_1]$
and every $x\in I\cap T^{-t}_{\a,\vphi}(W)$, we have 
\begin{equation}\label{xmin}
x_{min}^{N( x,t)}\leq \frac{1}{q_n\log^2q_n}
\end{equation}
and
\begin{equation}\label{eq:sec}
|\vphi'_{N( x,t)}(\bar x)|\geq\left(\frac{1}{2x^{N( x,t)}_{min}}\right)^{2-\eta}\;\;\text{   and   }
\;\;\;\;\;\;\;\;|\vphi ''_{N( x,t)}(\bar x)|\leq \left(\frac{2}{x^{N( x,t)}_{min}}\right)^{3-\eta}.
\end{equation}
\end{lemma}

In what follows, for simplicity, we will denote $N( x):= N( x,t)$. 

\begin{lemma}\label{xze} Let $x_0,x\in I\subset W$ with $|\bar x-\bar x_0|\geq \frac{1}{q_n^{3/2-2\eta}}$ satisfy $T^{t}_{\a,\vphi}(x)\in V$  and let 
$$|\vphi'_{N( x_0)}(\bar x_0)|\leq q_n^{7/4+\eta} \quad \text{ and } \quad |\vphi''_{N( x_0)}(\bar x_0)|\leq q_n^{3-\eta}\log^{10}q_n.$$
Then for some $A_{x,x_0}\geq \frac{q_n^{3-\eta}}{\log^5 q_n}$ we have
$$
|\vphi'_{N( x)}(\bar x)-\vphi'_{N( x_0)}(\bar x_0)-A_{x,x_0}(\bar x-\bar x_0)|\leq \frac{A_{x,x_0}}{10}|\bar x-\bar x_0|. 
$$
\end{lemma}

The proofs of Lemmas \ref{sec.fir} and \ref{xze} will be given in Appendix \ref{app.birkhoff}. Lemma \ref{xze} has the following straightforward consequence.

\begin{corollary}\label{smder} If $|\vphi'_{N( x_0)}(\bar x_0)|<3q_n^{3/2+\eta}$ and $|\vphi''_{N( x_0)}(\bar x_0)|<q_n^{3-\eta}\log^{10}q_n$ 
for some $x_0 \in W$, then for every $x\in I$ such that $|\bar x-\bar x_0|\geq \frac{1}{q_n^{3/2-3\eta}}$ either $T^t_{\a,\vphi}(x)\in V^c$
or if $x$ satisfies $T^{t}_{\a,\vphi}(x)\in V$,  then 
\begin{equation}\label{large.first}
|\vphi'_{N( x)}(\bar x)|\geq \frac{q_n^{3-\eta}}{2\log^5 q_n}|\bar x-\bar x_0|.
\end{equation}
\end{corollary}

\bg
\section{Mixing rate on intervals, construction of $\mathcal B_l$} \label{sec.partition}
\blk
 In what follows $I\subset W$ will be a horizontal interval  (such that  $\bar I \in \cI_k$) and $h=q_n^{3/5}$. Then we 
 know that the iterates $R_\a^i(\bar I)$ for $i=0,\dots,h$ are all disjoint and do not contain $0$. Recall the notation 
\begin{equation*}u_I:=\sup_{t\in[l_0,l_1]}\sup_{x\in I }|\vphi ''_{N( x,t)}(\bar x)|.
\end{equation*}

Moreover whenever $I_t:= I\cap T_{\a, \vphi}^{-t}W \not=\emptyset$, we define
\begin{equation}\label{rti}
r^t_I=\inf_{x\in I_t}|\vphi'_{N( x,t)}(\bar x)|
\end{equation}
 (if $I_t=\emptyset$ we may define $r^t_I=+\infty$). We also let
 
 \begin{equation}\label{ri}r_I=\inf_{t\in[l_0,l_1]}r^t_I.
\end{equation}

\begin{definition}[Complete towers]\label{compl}  Fix a horizontal interval $ I \subset  M \cap (\T \times \{s\})$ centered at $z$ and a number $h>0$.  
A complete tower of `height' $h$ above the interval $I$ is the set:
$$
\bigcup_{i=0}^{N(z,h)} (R_\a^i (\bar I))_\vphi \setminus \cup_{t=0}^s T^t_{\a, \vphi}(\bar I \times \{0\}).
$$
\end{definition}

We now describe the bad set for correlations $\mathcal B_l$ (see Figure \ref{wykres2}).

\begin{proposition}\label{prop.badset} There exists a set $\mathcal{B}_l\subset M$ with the following properties:
\begin{enumerate}
\item[$(B_1)$]  $\mathcal B_l=U_1\cup \dots \cup U_m$ where $U_i$ are disjoint complete towers with heights $h=q_n^{3/5}$ over intervals $
B_i \subset W$ with horizontal measure $\lambda ( B_i)=\frac{2}{q_n^{3/2-5\eta}} $;
\item[$(B_2)$]  $\mu(\mathcal B_l)\leq q_n^{-1/2+6\eta}$; 
\item[$(B_3)$]  for every interval $I\subset W$,  we have  $I=J_1 \sqcup J_2 \sqcup I_{bad}$ where  
either  $I\cap \mathcal{B}_l= \emptyset$ and $I_{bad}$ and $J_2$ are empty, or $I_{bad}$ is a level of some $U_i$ and $J_1,J_2$ are intervals. When $I_{bad}$ is not empty, we denote by  $x_{bad}$ its center.
\item[$(B_4)$]  for every interval $I\subset W$ and every $t\in [l_0,l_1]$, we have one of the following 
\begin{itemize} \item[$(B_4.i)$] $r^t_I\geq q_n^{3/2+\eta}$,

\item[$(B_4.ii)$] $r^t_I<q_n^{3/2+\eta}$, $I_{bad}\neq \emptyset$, $u_I\leq q_n^{3-\eta}\log^9q_n$ and 
for every $x \in J_1 \sqcup J_2$ s.t. $T_{\a,\vphi}^t x \in W$  
$$|\vphi '_{N( x,t)}{(\bar x)}|\geq   \frac{q_n^{3-\eta}}{\log ^6 q_n}|\bar x-\bar x_{bad}|$$ 
\end{itemize} 
 \item[$(B_5)$]  For every $t \in [l_0,l_1]$, for every $i \in [1,m]$,
 there exists a complete tower $\cT_{t,i}$ over an interval $B_{t,i}=[\theta_{t,i}-\frac{1}{q_n^{3/2-5\eta}},\theta_{t,i}+\frac{1}{q_n^{3/2-5\eta}}]  \times \{s_{t,i}\} \subset
  M$ of height $h_{t,i}\geq q_n^{3/5-1/50}$ such that 
 $$\mu(\left(T^{t}_{\a,\vphi}(U_i)\triangle \cT_{t,i}\right)\cap M_\zeta)\leq q_n^{-1 +10\eta}.$$
 \end{enumerate}
\end{proposition}

For a horizontal interval $I \subset W$ such that  $T^{t}_{\a,\vphi}I\cap W\neq \emptyset$, the quantity that measures uniform stretching on $I$ is the ratio
\begin{equation}\label{mj}
S_I^t:=\inf_{x \in I_t} \frac{(\vphi '_{N( x,t)}(\bar x))^2}{\vphi ''_{N( x,t)}(\bar x)},
\end{equation}
where $I_t=I\cap T^{-t}_{\a,\vphi}(W)$ (we set $S_I^t=+\infty$ if $I\cap T^{-t}_{\a,\vphi}W= \emptyset$).

\smallskip
We recall that the integer $l$, hence the integers $l_0=l^{21/20}$, $l_1= (l+1)^{21/20}$,  and the integer $n$ such that
$q_n <l_0 < q_{n+1}$, are fixed throughout this section. 

\begin{definition}\label{int.good}   An interval $J =[u,v] \subset I \subset W$ is called {\em good}  if for every $t \in [l_0,l_1]$, at least one of the following holds:

\begin{equation}
\label{verygood} 
S_J^t \geq t^{\frac{1}{2}+ 2 \epsilon}
\end{equation}    
or for some choice of $x^*\in I$  and for every $x \in J$ such that  $T^t_{\a,\vphi}x\in W$, we have 
\begin{equation}
\label{good}
|\vphi ''_{N( x,t)}{(\bar x)}|<q_n^{3-\eta}\log ^9 q_n\text{  and }
|\vphi '_{N( x,t)}{(\bar x)}|\geq   \frac{1}{2}q_n^{3/2+\eta} +  \frac{1}{2}\frac{q_n^{3-\eta}}{\log ^6 q_n}|\bar x-\bar x^*|.
\end{equation}
When we check~\eqref{verygood} or~\eqref{good} for a given $t$, we say that $J$ is $t$-good. 

\end{definition}

\begin{proposition}\label{good.dec} In the decomposition $I=J_1 \sqcup J_2 \sqcup I_{bad}$ of $(B_3)$, we have that $J_1$ and $J_2$ are good. 
\end{proposition}

\begin{proof}[Proof of Proposition \ref{good.dec}] Let $t \in [l_0,l_1]$.  If $r^t_I<q_n^{3/2+\eta}$ then  \eqref{good} holds on $J_1$ and $J_2$ (with $x^*=x_{bad}$) due to Lemma~\ref{sec.fir},  Proposition~\ref{prop.badset}, part $(B_4.ii)$, and the fact that for $x \in J_1 \cup J_2$ we have that 
$|\bar x-\bar x_{bad}|\geq q_n^{-3/2+5\eta}$. 

Now, if $r^t_I\geq q_n^{3/2+\eta}$, then we will actually establish that all of $I$ is $t$-good (which in particular implies the conclusion of Proposition \ref{good.dec} in this case):
\begin{lemma}\label{pr1} For any $t \in [l_0,l_1]$, 
if $r^t_I\geq q_n^{3/2+\eta}$, then $I$ is $t$-good. 
\end{lemma}

\begin{proof}[Proof of Lemma \ref{pr1}]

\subsubsection*{\bf Case 1: $u_I\geq q_n^{3-\eta}\log ^9 q_n$.}  In this case we do not use the assumption $r^t_I\geq q_n^{3/2+\eta}$. 
We use Lemma \ref{sec.fir} and get for every  $t\in [l_0,l_1]$ and every $x\in I\cap T^{-t}_{\a,\vphi}(W)$
$$ S ^t_I = \inf_{x\in I_t} \frac{\left(\vphi '_{N( x,t)}(\bar x)\right)^2}{|\vphi ''_{N( x,t)}(\bar x)|}\geq \frac{2^{-7}}{(x^{N( x,t)}_{min})^{1-\eta}}\geq q_n^{2/3} 
\geq t^{1/2+\epsilon}.$$
The last inequality holds because of $t<q_{n+2}$ and the Diophantine assumptions on $\a$.
This shows that $I$ satisfies \eqref{verygood} and hence finishes the proof of Lemma \ref{pr1} in this case.

\subsubsection*{\bf Case 2:  $u_I< q_n^{3-\eta}\log^9q_n$.}
 Notice first that if $r^t_I\geq q_n^{7/4+\frac{\eta}{2}}$ (see \eqref{rti} for the definition of $r^t_I$), then either $x\in T^{-t}_{\a,\vphi}(W^c)$ or
$$
S ^t_I = \inf_{x\in I_t} \frac{(\vphi '_{N( x,t)}(\bar x))^2}{\vphi ''_{N( x,t)}(\bar x)}\geq
\frac{q_n^{7/2+\eta}}{q_n^{3-\eta}\log^9 q_n}\geq q_n^{1/2+\eta}\geq
 t^{1/2+\eps},
$$
where the last inequality again holds because of $t<q_{n+2}$ and assumptions on $\a$. Therefore \eqref{verygood} holds for $I$ and the proof is 
finished in this case .\\

Let us consider only $x\in I$ such that $T^t_{\a,\vphi}(x)\in W$.
 If $r^t_I<q_n^{7/4+1/2\eta}$, let $x_0\in I$ be such that $|\vphi'_{N( x_0,t)}(\bar x_0)|=r^t_I$. Let us assume WLOG that $\vphi'_{N( x_0,t)}(\bar x_0)>0$.  Then by Lemma \ref{xze},  whenever $\bar x\geq \bar x_0+\frac{1}{q_n^{3/2-2\eta}}$, we have
\begin{equation}\label{exs}
|\vphi'_{N( x,t)}(\bar x)|\geq \frac{q_n^{3-\eta}}{2\log^5 q_n}|\bar x-\bar x_0|\geq  \frac{2q_n^{3-\eta}}{\log^6 q_n}|\bar x-\bar x_0|.
\end{equation}

If $\bar x<\bar x_0-\frac{1}{q_n^{3/2-2\eta}}$,
then $\vphi'_{N( x,t)}(\bar x)<0$. Indeed, otherwise by Lemma \ref{xze} we have
$$
0\leq\vphi'_{N( x,t)}(\bar x)<\vphi'_{N( x_0,t)}(\bar x_0)+\frac{q_n^{3-\eta}}{2\log^5 q_n}(\bar x-\bar x_0)\leq \vphi'_{N( x_0,t)}(\bar x_0)- q_n, 
$$
which is a contradiction with the choice of $x_0$. Therefore we have $\vphi'_{N( x,t)}(\bar x)<0$ and, by Lemma~\ref{xze} and by the definition of $x_0$, we derive
\begin{equation}\label{exs2}|\vphi'_{N( x,t)}(\bar x)|\geq \frac{q_n^{3-\eta}}{4\log^5 q_n}|\bar x-\bar x_0|\geq \frac{2q_n^{3-\eta}}{\log^6 q_n}
 |\bar x-\bar x_0|.
\end{equation}
Then by \eqref{exs} and \eqref{exs2} and since $r_I\geq q_n^{3/2+\eta}$,  we get that \eqref{good} is satisfied with $x^*:=x_0$. This finishes the proof in Case 2.  and Lemma \ref{pr1} is established. \end{proof} 

The proof of Proposition~\ref{good.dec} is hence finished. \end{proof}

\subsection{Construction of the bad set $\mathcal{B}_l$}\label{been}

Recall that the partition $\cI_k$ is given by two towers i.e. disjoint sets of the form $\{B+i\a\}_{i=0}^{q_k}$ and $\{C+i\a\}_{i=0}^{q_{k-1}}$ where $B,C$ are intervals around $0$ of length $\|q_{k-1}\a\|, \|q_k\a\|$ respectively. Denote $D_1=B+\a, D_2=C+\a$ (the shift comes from the fact that we want to stay away from the singularity). The following construction works for $D=D_1,D_2$. We will present it for the tower above $D=D_1$, the other case being analogous. Consider a complete tower $\cD$ of height $H_k=q_k-1$ over $D$. Notice that $\cD\cap W$ is a union of horizontal intervals of length $\lambda (D)$. Moreover there is a natural order on horizontal intervals in $\cD\cap W$ (coming from the order on $\cD$): each interval in $\cD\cap W$ is of the form $D(h)$ for some $0\leq h\leq H_k$ (with $D(0)=D$). 

 Let $0\leq h_1\leq H_k$ be the smallest real number such that 
$D(h_{1})\subset \cD\cap W$ and $r_{D(h_{1})}\leq 2q_n^{3/2+\eta}$. Let $t_1 \in [l_0,l_1]$ and $x_1:= (\theta_1, s_1)\in D(h_1)$  be such that 
$$ T^{t_1}_{\a,\vphi}x_1\in W \text{  and  } \vphi'_{N(\theta_1,t_1)}(\theta_1)\leq 2q_n^{3/2+\eta}.
$$
Let $U_1$ be the complete tower of height $q_n^{3/5}$ over 
$B_1:=\left([-\frac{1}{q_n^{3/2-5\eta}}+\theta_1,\theta_1+\frac{1}{q_n^{3/2-5\eta}}]\times\{s_1\}\right)\cap \cD.$ 
Let $k_2$ be the largest number such that $D(k_2)\subset \cD \cap W$.

Now inductively  let $H_k\geq h_{i}\geq k_i$ be the smallest real number such that 
$D(h_{i})\subset \cD \cap W$ and $r_{D(h_{i})} \leq 2q_n^{3/2+\eta}$. Let $t_i \in [l_0,l_1]$ and $x_i:=(\theta_i, s_i)
 \in D(h_i)$   be such that 
\begin{equation}\label{txi} T^{t_i}_{\a,\vphi} x_i\in W \text{  and  } \vphi'_{N(\theta_i,t_i)}(\theta_i)\leq 2q_n^{3/2+\eta}.
\end{equation}
We define $U_i$ to be the complete tower of height $q_n^{3/5}$ over
$B_i:=\left([-\frac{1}{q_n^{3/2-5\eta}}+\theta_i, \theta_i+\frac{1}{q_n^{3/2-5\eta}}]\times \{s_i\} \right) \cap \cD.$

We continue this procedure until the last possible $h_m\leq H_k$ is defined.

Let us define
 \begin{equation}\label{def:b}
\mathcal B_l:=\bigcup_{1\leq i\leq m} U_i.
\end{equation}
Now, $(B_1)$ and $(B_3)$ follow by construction (notice that the top of $U_i$ is below the base of $U_{i+1}$).  Moreover, by Lemma \ref{koksi} we get that $\vphi_{q_k-1}(\a)\leq cq_{k+1}$, hence $(B_2)$ follows from 
 \begin{equation*}\mu(\mathcal B_l)\leq \vphi_{q_k}(\a) \lambda ( B_i)\leq \frac{1}{q_n^{1/2-6\eta}}.
\end{equation*} 

It remains to prove $(B_4)$ and $(B_5)$, which will be the subject of the next subsection.

\begin{figure}[htb] 
 \centering
  \resizebox{!}{6cm}{\includegraphics[angle=0]{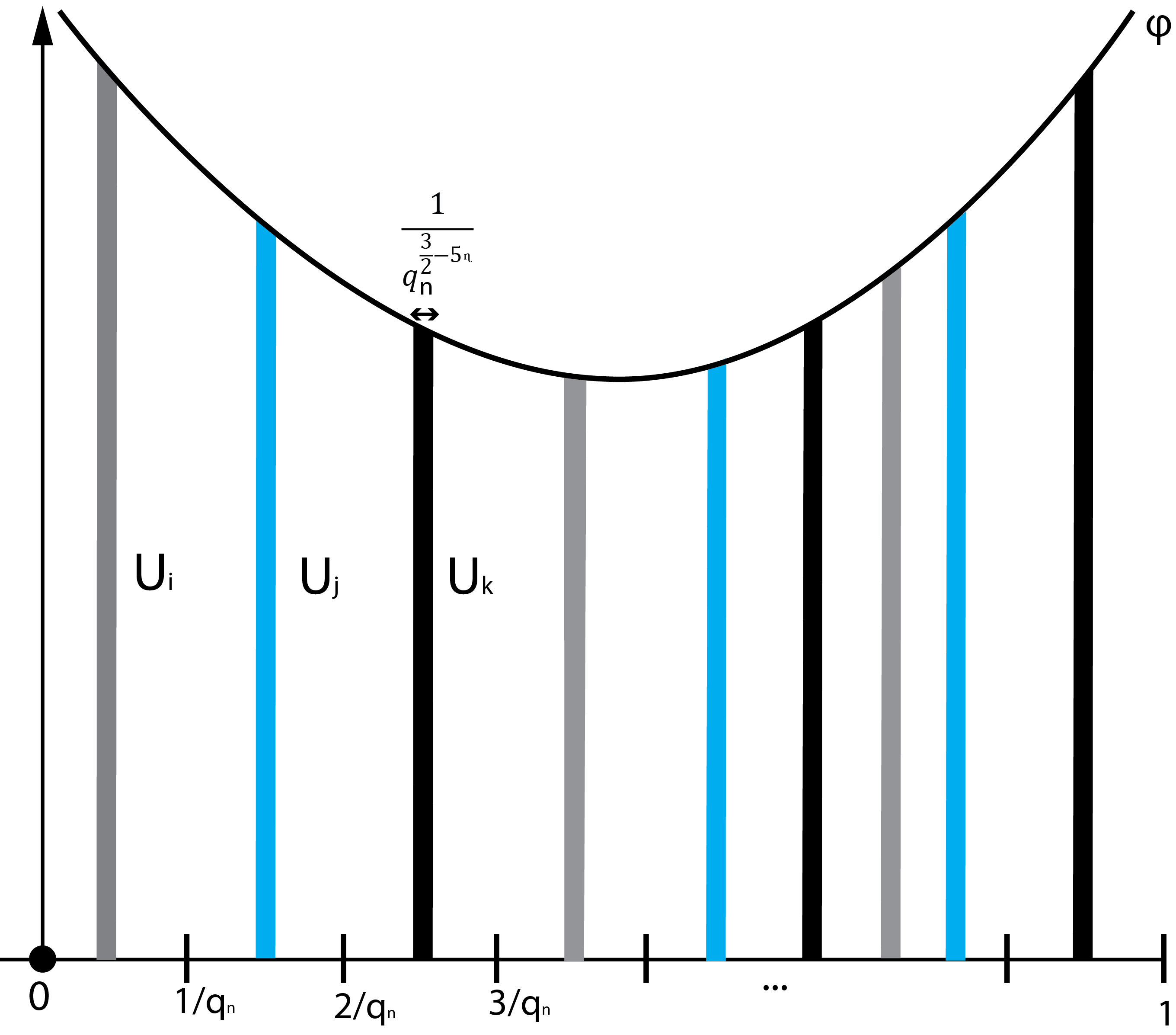}}
\caption{\small The set $\mathcal{B}_l$ is a union of complete towers $U_i$.} 
\label{wykres2}
\end{figure}

\subsection{Proving the properties of the bad set}\label{pr456}

In this section we give the proofs of $(B_4)$ and $(B_5)$ in Proposition \ref{prop.badset}

\medskip
\textbf{Proof of $(B_4)$.}
Fix $t\in [l_0,l_1]$. By the construction of $\mathcal B_l$, whenever for a partition interval $I\subset W$  we have $r^t_I\leq q_n^{3/2+\eta}$, 
then $$
I\cap \mathcal{B}_l=I_{bad},
$$
where $I_{bad}$ is a level of some $U_i$. In fact, otherwise $I\cap \mathcal{B}_l=\emptyset$ and by construction $r_I > 2 q_n^{3/2+\eta}$.
Therefore we need to show $(B_4.ii)$ for  $I\subset W$ such that  $I_{bad}\neq\emptyset$ and $r^t_I<q_n^{3/2+\eta}$.  Then, by definition,  there exists $x^t_{I}\in I$ such that  
\begin{equation}\label{smat}
T^t_{\a,\vphi}(x^t_{I})\in W\subset V
\text{  and  }
\vphi'_{N( x^t_I,t)}(\bar x^t_{I})\leq q_n^{3/2+\eta}.
\end{equation}
Notice that we have
\begin{equation}\label{sec.small}
u_I<q_n^{3-\eta}\log^9q_n.
\end{equation}
Indeed, if not, then by Lemma \ref{sec.fir} we would get by \eqref{xmin} and \eqref{eq:sec} that $\vphi'_{N( x^t_I,t)}(\bar x^t_I)\geq q_n^{2-\eta}$, which is a contradiction with \eqref{smat}.

Notice that by \eqref{smat} and \eqref{sec.small}, the assumptions of Corollary \ref{smder} are satisfied with $x_0=x^t_I$. Therefore, for every $x\in I$ such that $T^t_{\a,\vphi}(x)\in V$  and $|x-x^t_I|\geq \frac{1}{q_n^{3/2-3\eta}}$, we have 
\begin{equation}\label{fans}
|\vphi'_{N( x,t)}(\bar x)|\geq \frac{q_n^{3-\eta}}{2\log^5q_n}|\bar x-\bar x^t_I|.
\end{equation}
We claim that
\begin{equation}\label{dists}
|\bar x^t_I-\bar x_{bad}|\leq q_n^{-3/2+4\eta}.
\end{equation}

Now, \eqref{fans}, \eqref{sec.small} and \eqref{dists}  will finish the proof of  $(B_4.ii)$ since for $x\in J_1\sqcup J_2=I\setminus I_{bad}$, we have that  $|\bar x-\bar x_I^t| \geq |\bar x-\bar x_{bad}|-|\bar x^t_I-\bar x_{bad}| \geq q_n^{-3/2+3\eta}$.

Thus it only remains to show our claim~\eqref{dists}. By construction of  the $U_i$'s, for some $h>0$, we can write
$$x_{bad}=T^h_{\a,\vphi} x_i\,.$$
Moreover, since $U_i$ is a complete tower of height $q_n^{3/5}$ and $T^h_{\a,\vphi} x_i\in U_i$, we have that 
$$
h\leq \vphi_{N( x_i,q_n^{3/5})}(\bar x_i) + \vphi(\bar x_i+N( x_i,h)\a).
$$
Since $x_i\in W$, we get by the definition of special flow
$$\vphi_{N( x_i,q_n^{3/5})}(\bar x_i)\leq 2q_n^{3/5}.$$
Moreover, since $T^h_{\a,\vphi}x_i\in W$, we have 
$$\vphi(\bar x_i+N( x_i,h)\a)\leq 2q_n^{3/5}.
$$
By putting together the above bounds, we get
\begin{equation}\label{sm:h}
h<2q_n^{3/5+1/50}.
\end{equation}

Let $m_i:=\max(t_i,t)$. We will show that 
\begin{enumerate}
\item[a.] $T^{m_i}_{\a,\vphi}(T^h _{\a,\vphi}x_i), T^{m_i}_{\a,\vphi}(x_I^t)\in V$;
\item[b.] $|\vphi'_{N( x_I^t,m_i)}(\bar x^t_I)|\leq 2q_n^{3/2+\eta};$
\item[c.] $|\vphi'_{N(T^h_{\a,\vphi}x_i,m_i)}(\bar T^h_{\a,\vphi}x_i)|\leq 5q_n^{3/2+\eta}$.
\end{enumerate}
The above properties will give \eqref{dists} (and hence $(B4.ii)$), since if $|\bar T^h_{\a,\vphi}x_i-\bar x^t_I|\geq q_n^{-3/2+4\eta}$ then by \eqref{sec.small} and a., b., the assumptions of Corollary \ref{smder} are satisfied with $x_0=x^t_I$, $x=T^h_{\a,\vphi}x_i$ but then c. is in contradiction with estimate \eqref{large.first} stated there.  It remains then to show a., b.,c.\\

\smallskip
 For a. we notice that by \eqref{txi} and \eqref{smat} we have $T^{t_i}_{\a,\vphi}x_i,T^t _{\a,\vphi}x^t_I\in W$. Moreover, by the immediate bound
 $|m_i-t|\leq l_1-l_0<q_n^{1/10}$ and by \eqref{sm:h}, we have the estimate
\begin{equation}
\label{mt}
0\leq m_i-t,m_i-t_i+h\leq 2q_n^{3/5+1/50}+q_n^{1/10} \leq  3 q_n^{3/5+1/50}\,,
\end{equation}
from which we derive that 
$$
\{T^{m_i}_{\a,\vphi}(T^h _{\a,\vphi}x_i), T^{m_i}_{\a,\vphi}(x_I^t)\}=
\{T^{m_i-t_i+h}_{\a,\vphi}(T^{t_i}_{\a,\vphi}x_i),T^{m_i-t}_{\a,\vphi}(T^t_{\a,\vphi}x_I^t)\} \subset V.
$$
This gives a.  

\smallskip
For b. we first notice that since $T^t_{\a,\vphi}(x^t_I)\in W$ and $|m_i-t|\leq l_1-l_0<q_n^{1/10}$, by \eqref{con:ele},
we have 
$$
\vphi'_{N( x^t_I,m_i-t)}(\bar T^t_{\a,\vphi}(x^t_I))|\leq q_n^{3/2+\eta}
$$
and by \eqref{smat}, $|\vphi'_{N( x^t_I,t)}(x^t_I)|\leq q_n^{3/2+\eta}$. By the cocycle identity, we then have 
$$
|\vphi'_{N( x^t_I,m_i)}(x^t_I)|\leq |\vphi'_{N(x^t_I,t)}(\bar x^t_I)|+|\vphi'_{N(x^t_I,m_i-t)}(\bar T_{\a,\vphi}^t(x^t_I))|\leq 2q_n^{3/2+\eta}.
$$
This gives b. 

\smallskip
For c., by cocycle identity, \eqref{txi}, \eqref{mt}  and \eqref{con:ele} (for $T^{t_i}_{\a,\vphi}(x_i)\in W$), we get 
\begin{equation}\label{ew1}
|\vphi'_{N( x_i,m_i+h)}(\bar x_i)|\leq |\vphi'_{N( x_i,t_i)}(\bar x_i)|+
|\vphi'_{N(T_{\a,\vphi}^{t_i}(x_i),m_i+h-t_i)}(\bar T^{t_i}_{\a,\vphi}(x_i))|\leq 2q_n^{3/2+\eta}.
\end{equation}
Since $x_i\in W$, by \eqref{sm:h} and \eqref{con:ele}, we have 
\begin{equation}\label{ew2}
|\vphi'_{N( x_i,h)}(\bar x_i)|\leq 2q_n^{3/2+\eta}.
\end{equation}
Finally from the  cocycle identity, \eqref{ew1} and \eqref{ew2} we conclude that 
$$
|\vphi'_{N(T^h_{\a,\vphi}x_i,m_i)}(\bar T^h_{\a,\vphi}x_i)|\leq |\vphi'_{N( x_i,m_i+h)}(\bar x_i)| +
| \vphi'_{N( x_i,h)}(\bar x_i)|\leq 5q_n^{3/2+\eta}.
$$
This finishes the proof of c. and hence also $(B4.ii)$.

\medskip
\textbf{Proof of $(B_5)$.} 

Let $s_i$ be such that $x_i  \in D(h_i) \subset  \T \times \{s_i\}$  ($D(h_i)$ is the base of $U_i$).
Let $\bt \in[t,t-1]$ be such that for  $z_{t,i} =(\theta_{t,i}, s_{t,i}):=T_{\a,\varphi}^{\bt}x_i$ we have
$$B_{t,i}:=[\theta_{t,i}-\frac{1}{q_n^{3/2-5\eta} },\theta_{t,i}+\frac{1}{q_n^{3/2-5\eta}}] \times \{s_{t,i}\}\subset M,$$

Let $h_{t,i}:=\vphi_{N( x_i,q_n^{3/5})}(x_i)-s_i- (\bt-t)$ and let $\cT_{t,i}$ be the complete tower of height $h_{t,i}$ over $B_{t,i}$. Notice that $s_i\leq q_n^{3/5(1-\eta)}$ and $\vphi_{N( x_i,q_n^{3/5})}(x_i)\geq q_n^{3/5}\log^{-10}q_n$ (by \eqref{koks0}), hence~$h_{t,i}\geq q_n^{3/5-1/50}$.

The difference between $U_i\cap M_\zeta$ and $T^{t}_{\a,\vphi}(\cT_{t,i})\cap M_\zeta$ will come from the stretching of Birkhoff sums of the top and at the base of $\cT_{t,i}$ and from the difference $|\bt-t|\leq 1$. The measure of the symmetric difference between the two sets is twice the maximal stretching times the measure of the base of $\cT_{t,i}$. First let us estimate the maximal stretch. 

For any $z\in B_{t,i}$ there exists 
$\xi_i\in [\bar z, \theta_{t,i}]$ such that 
\begin{equation}\label{le}
|\vphi_{N(\theta_{t,i},t)}(\bar z)-\vphi_{N(\theta_{t,i},t)}(\theta_{i,t})|\leq 
|\vphi'_{N(\theta_{t,i},t)}(\xi_i)||\bar z-\theta_{t,i}|
\end{equation}
Since $t<q_{n+1}$, it follows that for $j=0,\dots,N(\theta_{t,i},t)-1$, we have $\theta_{t,i}+j\a \notin [-\frac{1}{q_n\log^{100}q_n},\frac{1}{q_n\log^{100}q_n}]$ and since  $|\xi_i-\theta_{t,i}|<\frac{1}{q_n^{3/2-5\eta}}$, it follows that for $j=0,\dots,N(\theta_{t,i},t)-1$, we have
$$
\xi_i+j\a\notin \left[-\frac{1}{2q_n\log^{100}q_n},\frac{1}{2q_n\log^{100}q_n}\right].
$$
By the above condition and by \eqref{koks0}, we derive from \eqref{le} the bound
$$|\vphi_{N(\theta_{t,i},t)}(\bar z)-\vphi_{N(\theta_{t,i},t)}(\theta_{t,i})|\leq q_n^{1/2+3\eta}.
$$
Therefore, 
$$\mu((T^{t}_{\a,\vphi}(U_i)\triangle \cT_{t,i})\cap M_\zeta)\leq  \lambda ( B_i) (4q_n^{1/2+3\eta}+|t-\bt|) \leq q_n^{-1+10\eta}.$$

This finishes the proof of $(B5)$ and hence also of Proposition \ref{prop.badset}.

\bg 
\section{From uniform stretching of Birkhoff sums to decay of correlations} \label{sec.correlations}
\blk

\subsection{Uniform stretching of Birkhoff sums and correlations} \label{sec.uniformstretchandcorrelations}

We will adopt below the following notation. 

For all $f\in \cF$ with transfer function $\phi$ and $g \in C^1(M)$, let
$$
\Nzero(f,g) := \Vert \phi \Vert_0 \Vert  g \Vert_0  \quad \text{and} \quad \None(f,g):= 
(\Vert f \Vert_0 +\Vert \phi \Vert_0) \Vert g \Vert_1 +  (\Vert f \Vert_1 +\Vert \phi \Vert_1) \Vert g \Vert_0\,,
$$
where $\Vert \cdot \Vert_0$ and $\Vert \cdot \Vert_1$ respectively denote the $C^0$ and the $C^1$ norm. Moreover, we will denote by the letter
$C$ a generic constant which depends only on on the rotation number $\a$ and on the
ceiling function $\vphi$.

In all what follows $I$ denotes any interval of the partition of $W$ defined in Section \ref{sec.partition}. 

Our main result in this section is the following relation between uniform stretching of the Birkhoff sums and decay of correlations. 

Let us recall the following notation (see \eqref{rti}). For any interval $J \subset I$ denote 
\begin{equation}\label{aj}
r^t_J:=\inf_{x\in J_t}|\vphi'_{N( x,t)}(\bar x)|,
\end{equation}
where $J_t:=J\cap T^{-t}_{\a,\vphi}W$ ($r^t_J=+\infty$ if $J_t=\emptyset$).

\begin{proposition}\label{rem.cse2} For any interval $J=[z,w]\times\{s\} \subset I$, we have 
the following estimate:  
$$
\left|\int_{\bar J}f(T^t_{\a,\vphi}(\theta,s))g(\theta,s)d \theta-p(z,w)\right|\leq  C\left\{ \Nzero(f,g) \frac{\lambda ( J)}{S^t_J}  + 
\None(f,g) \frac{\lambda ( J)}{r^t_J} \right\},$$
where $p(z,w)=\frac{g(z,s)\phi(T_{\a,\vphi}^t(z,s))}{\vphi'_{N(z,t)}(z)}
-\frac{g(w,s)\phi(T_{\a, \vphi}^t(w,s)}{\vphi'_{N(w,t)}(w)}$.
\end{proposition}

To prove Proposition \ref{rem.cse2}, we will need the following lemma that encloses the main estimate  
on the correlation of coboundaries based on the stretching of the Birkhoff sums of the roof function.

 Let $\J:=[u,v] \times \{s\} \subset J$ be such that $v-u\leq t^{-10}$.

\begin{lemma} \label{prop.bI}  Let $r^t_u:=-\vphi_{N(u,t)}'(u)$. For all $f\in \cF$ and for 
all $g\in C^1_0(M)$ and  for all $t>0$ we have
\begin{equation}  \label{eq.bI} \left| \int_{\bar \J} f(T_{\a,\vphi}^t(\theta,s)) g(\theta,s) d \theta  - \Delta(\J,t)  \right| \leq C \None(f,g)
 \frac{\lambda ( \J)}{r^t_I  },  \end{equation}
where $\None(f,g)= (\Vert f \Vert_0 +\Vert \phi \Vert_0) \Vert g \Vert_1 +  (\Vert f \Vert_1 +\Vert \phi \Vert_1) \Vert g \Vert_0$ and 
 
$$\Delta(\J,t) := \frac{1}{r^t_u} \left[ g(v,s) \phi(T_{\a,\vphi}^t(v,s)) - g(u,s) \phi(T_{\a,\vphi}^t(u,s)) \right].$$
\end{lemma}

\begin{proof}
Let $I \subset W \cap (\T \times \{s\})$ be a horizontal  interval  as in Section~\ref{sec.partition}.
 Let $\J=[u,v]\subset I$ such that $v_0-u_0\leq t^{-10}$. If $T_{\a,\vphi}^{-t}\J\subset W^c$ then Lemma \ref{prop.bI} holds trivially. We use the notation 
 $$T_{\a,\vphi}^t(u,s)=(\tu,\ts)=(u+N(u,t)\a,t+s-\vphi_{N(u,t)}(u))\,,$$ 
 where $0\leq \ts \leq \vphi(u+N(u,t) \a)$. We also denote $\tv = v+N(u,t)\a$. 

In the remainder of this proof we will denote for simplicity the integer $N(u,t)$ by $N$. 
We will suppose that $r^t_u=-\vphi_{N}'(u) \geq r^t_I \geq 0$, the case where $r^t_I<0$ being similar. Let us also denote 
$$B^t_I:=\sup_{\theta\in I}\vphi ''_N(\theta).$$

We will use the notation $X = O(Y)$ if there exists a constant $C>0$ such that $X \leq  C Y$.

We have for $\theta \in [0,\lambda (\J)]$ that $T_{\a,\vphi}^t(u+\theta,s) = (\tu+\theta, \ts+\vphi_N(u)-\vphi_N(u+\theta))$. By the intermediate value theorem, since $r^t_I \ll \lambda ( \J)^{-1}$, we have 
\begin{align*}
 \int_{\bar \J} f(T_{\a,\vphi}^t(\theta,s)) &g(\theta,s) d \theta = \int_0^{\lambda ( \J)}  f (\tu+\theta, \ts+\vphi_N(u)-\vphi_N(u+\theta)) g(u+\theta,s) d \theta \\
= &g(u,s)  \int_0^{\bar \lambda (\J)}  f (\tu+\theta, \ts+\vphi_N(u)-\vphi_N(u+\theta))  d \theta +O(\Vert f\Vert_0  \Vert g \Vert_1  \frac{\lambda (\J)}{r^t_I} ). \end{align*}
Now,  since  $ \vphi_N(u)-\vphi_N(u+\theta)\ll1 $ we also have  
$$
\begin{aligned}
   \int_0^{\lambda (\J)}  &f (\tu+\theta, \ts+\vphi_N(u)-\vphi_N(u+\theta))  d \theta \\ =  &\int_0^{\lambda (\J)}  f (\tv, \ts+\vphi_N(u)-\vphi_N(u+\theta)) d \theta+O( \Vert f\Vert_1  \frac{ \lambda ( \J)}{r^t_I}),
 \end{aligned}
 $$
and by the definition of $B^t_I$, we have $|\vphi_N(u)-\vphi_N(u+\theta) - r^t_u \theta | \leq B^t_I \theta^2$. Therefore,
$$
\begin{aligned}
 \int_{\bar \J} f(T_{\a,\vphi}^t(\theta,s)) g(\theta,s) d \theta &= g(u,s)  \int_0^{\lambda (\J)}  f (\tv, \ts+r^t_u\theta)  d \theta  \\
 & +O(\|f\|_0\|g\|_1\frac{\lambda (\J)}{r^t_I})+
O(\|f\|_1\|g\|_0\frac{\lambda ( \J)}{r^t_I}).
\end{aligned}
$$
For simplicity let us denote $w(f,g):= \|f\|_0\|g\|_1 +\|f\|_1\|g\|_0$. A change of variable  then gives
\begin{align*}
 \int_{\bar \J} f(T_{\a,\vphi}^t(\theta,s)) g(\theta,s) d \theta&= \frac{1}{r^t_u}  g(u,s)  \int_0^{r^t_u \lambda ( \J)}  f (\tv, \ts+\theta)  d \theta + 
 O(w(f,g)\frac{\lambda ( \J)}{ r^t_I}) \\
 &= \frac{1}{r^t_u}  g(u,s)  \left[\phi( \tv,\ts+r^t_u\lambda (\J))-\phi(\tv,\ts)\right] + O(w(f,g)\frac{\lambda (\J)}{r^t_I}) \end{align*}
 but $T_{\a,\vphi}^t (v,s)=(\tv,\ts+\vphi_N(u)-\vphi_N(v))=(\tv,\ts+r^t_u \lambda ( \J ) +\cE)$ with $\cE \leq B^t_I \lambda ( \J)^2$, hence 
\begin{align*}
 \int_{\bar \J} f(T_{\a,\vphi}^t(\theta,s)) g(\theta,s) d \theta &= \frac{1}{r^t_u}  g(u,s)  \left[\phi(T_{\a,\vphi}^t (v,s))-\phi(\tv,\ts)\right] \\
 &\qquad\qquad+O(w(f,g)\frac{\lambda ( \J)}{r^t_I} + \|g\|_0\|\phi\|_1\frac{\lambda ( \J)}{r^t_I})  \\
&= \frac{1}{r^t_u}   \left[g(v,s) \phi(T_{\a,\vphi}^t (v,s))-g(u,s) \phi(T_{\a,\vphi}^t(u,s))\right] \\ 
&\qquad\qquad+O(\None(f,g)\frac{\lambda ( \J)}{r^t_I})\,,
\end{align*}
which is precisely formula~\eqref{eq.bI}.  \end{proof}

\begin{proof}[Proof of Proposition \ref{rem.cse2}]
Since the proof is symmetric for $t>0$ and $t<0$, from now  on we will assume that $t>0$.
If $T_{\a,\vphi}^{t}(J)\subset W^c$, then Proposition \ref{rem.cse2} holds trivially. We assume for definiteness that $-\vphi_{N(u,t)}'(u) \geq r^t_J$ on $J$.
Let us decompose $J$ into finitely many subintervals $J=\bigcup_{i=1}^m \Ji$ such that $\Ji=[u_i,u_{i+1})\times \{s\}$ with $|u_{i+1}-u_i|\leq t^{-10}$, and so that $N(\cdot,t)$ is constant on each $\Ji$. 

Then 
\begin{equation}\label{fgh}
\int_{\bar J} f(T_{\a,\vphi}^t(\theta,s)) g(\theta,s) d \theta = \sum_{i=1}^m  \int_{\Ji} f(T_{\a,\vphi}^t(\theta,s)) g(\theta,s) d \theta 
= \sum_{i=1}^m  \Delta(\Ji,t) + \cE\,,
\end{equation}
where, by \eqref{eq.bI}
$$
\cE \leq \None(f,g) \frac{\lambda ( J)}{r^t_J}\,.
$$
Notice that if  $T_{\a, \vphi}^{t}(\Ji)\subset W^c$ then the corresponding integral in \eqref{fgh} is $0$. Therefore we only have to consider those $\Ji$ for which $T_{\a,\vphi}^{-t}(\Ji)\nsubseteq W^c$. By enumeration let us assume that this is the case for all $\Ji$.

Let us  denote $r^t_i:= -\vphi_{N(u_i,t)}'(u_i)$ and $\Theta_i:=  g(u_i,s) \phi(T_{\a,\vphi}^t(u_i,s))$. We then have
\begin{align*} \vert \sum_{i=1}^m  \Delta(\Ji,t)- p(z,w)\vert &= \vert \sum_{i=1}^m  \frac{1}{r^t_i}(\Theta_{i+1}-\Theta_i) - p(z,w)\vert \\ &= \vert \frac{1}{r^t_m} \Theta_{m+1}- \frac{1}{r^t_1} \Theta_{1} + \sum_{i=1}^{m-1} \left(  \frac{1}{r^t_i} -\frac{1}{r^t_{i+1}} \right) \Theta_{i+1} -p(z,w)\vert \\ 
&= \vert\sum_{i=1}^{m-1} \left(  \frac{1}{r^t_i} -\frac{1}{r^t_{i+1}} \right) \Theta_{i+1}\vert\leq \Vert \phi \Vert_0  \Vert g \Vert_0 \left(\frac{1}{r^t_J}+ \sum_{i=1}^{m-1}  \frac{\left| r^t_{i+1} -r^t_i \right|}{r^t_{i+1}r^t_{i}}  \right)\,.
 \end{align*} 

To estimate the quantity $ \sum_{i=1}^{m-1}  \frac{\left| r^t_{i+1} -r^t_i \right|}{r^t_{i+1}r^t_{i}}$,
by the choice of $(u_i)_{i=1}^m$ (since $N(\cdot,t)$ is constant on $\Ji$) and $u_{i+1}-u_i\leq t^{-10}$, we get 
$$|r^t_{i+1}-r^t_i|\leq 2B^t_i \lambda (\Ji)$$
where $B^t_i:=\vphi''_{N(u_i,t)}(u_i)$.
To conclude the argument, we notice that (since $u_{i+1}\sim u_i$) 
\begin{equation}\label{nts}
 \sum_{i=1}^{m-1}  \frac{B^t_i\lambda (\Ji)}{r^t_{i+1}r^t_{i}}\leq  \frac{\lambda ( J)}{ S^t_J}.
\end{equation}
This, by~\eqref{fgh},  finishes the proof of Proposition \ref{rem.cse2}

\end{proof}

Proposition \ref{rem.cse2} has the following corollaries that allow us to deal with the decay of correlations  on good intervals.
In the corollaries below $C$ again denotes a global positive constant which depends only on the rotation number $\a$ and on the
ceiling function $\vphi$. It may be different in each corollary. 

\begin{corollary}\label{correlation.good} For every good interval $J$, we have
\begin{equation}\label{cor:j}
\left|\int_{\bar J}f(T^t_{\a,\vphi}(\theta,s))g(\theta,s)d \theta\right|\leq  C ( \Nzero(f,g) q_n^{-1} + \None(f,g) q_n^{-2}) \, t^{-1/2-\frac{\eta}{4}}.
\end{equation}
\end{corollary} 

\begin{proof} Assume $J\cap T_{\a,\vphi}^{-t}W\neq\emptyset$ (otherwise the LHS is $0$) and let first \eqref{verygood} hold in the definition 
\ref{int.good} of a good interval. Notice that for $x\in T^{-t}_{\a,\vphi}(W)$,  $\vphi''_{N( x,t)}(\bar x)\geq q_n^{3-10\eta}$ (see \eqref{con:ele2})
and hence by \eqref{verygood}, $1/ r^t_J\leq q_n^{-3/2-4\eps}\leq t^{-1/2-2\eps}\lambda ( J)$. Moreover,  $p(z,w)\leq C \Nzero(f,g)  /r^t_J \leq 
C \Nzero(f,g) t^{-1/2-\eps}\lambda ( J)$. An application of Proposition \ref{rem.cse2} for $J$ finishes the proof in this case. If \eqref{good} holds, 
define $J_{weak}:=[x^*-\frac{1}{q_n^{3/2-2\eta}}, x^*+\frac{1}{q_n^{3/2-2\eta}}]\cap J$. Notice that by \eqref{good}, 
$$
 r^t_{J_{weak}}\geq q_n^{3/2+\eta} \quad \text{ and }  \quad S^t_{J_{weak}}\geq q_n^{\frac{5\eta}{2}}.
$$
So by Proposition \ref{rem.cse2} for $J_{weak}$, we have

\begin{equation}\label{eq:cs0}
\left|\int_{\bar J_{weak}}f(T^t_{\a,\vphi}(\theta,s))g(\theta,s)d \theta \right|\leq C ( \Nzero(f,g) q_n^{-1} + \None(f,g) q_n^{-2})    \, t^{-1/2-\frac{\eta}{4}}.
\end{equation}
Therefore it remains to show \eqref{cor:j} with $J\setminus J_{weak}$. Let $J=[z,w] \times\{s\}$  and let $J \setminus J_{weak}= J' \cup J''$, so that
$z\in J'$ (unless $J'=\emptyset$) and $w \in J''$ (unless $J''=\emptyset$). We will show \eqref{cor:j} for $J'$ and $J''$.  We will apply the same procedure to both $J'$ and $J''$, therefore we will explain the argument only in the case of $J''$. Let $m\in \N$ be the unique positive integer s.t. 
$2^m\leq q_n^{3/2-2\eta}(w-x^*)\leq 2^{m+1}$. Let us consider the intervals $J''_i:= [w_i, w_{i+1}] \times \{s\}=[x^*+\frac{w-x^*}{2^{i+1}},
x^*+\frac{w-x^*}{2^i}] \times \{s\}\cap J''$, where $i=0,\dots,m$. Then $J''=\bigcup_{i=0}^m J''_i$ (notice that $J_m$ may be degenerated). Consider only those 
$J''_i$ for which  $T^{-t}_{\a,\vphi}(J''_i)\nsubseteq W^c$. By enumeration assume this is the case for all $i=0,\dots,m$. By \eqref{good} we have 
\begin{equation}\label{taj} 
r^t_{J''}\geq q_n^{3/2+\frac{\eta}{2}}.
\end{equation}

Moreover by \eqref{good},  for every $J''_i$, we have 
$$\sup_{x\in J''_i}\vert \vphi''_{N( x)}(\bar x) \vert \leq q_n^{3-\eta}\log^9 q_n \quad \text{ and } \quad \inf_{x\in J''_i} \vphi'_{N( x)}(\bar x) \geq \frac{q_n^{3-\eta}(w-x^*)}{2^{i+2}\log^5 q_n} \,.$$
 
Therefore, we have the following estimate:
\begin{equation}\label{trj}\sum_{i=0}^{m}  \frac{\lambda ( J''_i)}{S^t_{J''_i}}\leq \frac{\log^{20} q_n}{q_n^{3-\eta}}\sum_{i=0}^m
\frac{2^{2i+4} }{(w-x^*)^2} \lambda ( J''_i)\leq
\frac{8\log^{20} q_n}{(w-x^*)q_n^{3-\eta}}2^{m+1}\leq\frac{1}{q_n^{3/2+\frac{\eta}{2}}}\leq t^{-1/2-\frac{\eta}{3}}\lambda ( J). 
\end{equation}
Notice that by the definition of the function $p(z,w)$ (see Proposition \ref{rem.cse2}), we have  $p(w_0,w_{m+1})=\sum_{i=0}^{m}p(w_i,w_{i+1})$. By Proposition \ref{rem.cse2} for $J''_i$, $i=0,\dots,m$ and by \eqref{taj}, \eqref{trj}, we derive
$$
\begin{aligned}
\left|\int_{\bar J''}f(T^t_{\a,\vphi}(\theta,s))g(\theta,s)d \theta \right| &\leq |p(w_0,w_{m+1})|+ \left|\int_{\bigcup J''_i}f(T^t_{\a,\vphi}(\theta,s))g(\theta,s)d \theta- p(w_0,w_{m+1})\right| \\
&\leq C \{  \Nzero(f,g) \lambda ( J) + \None(f,g) \lambda ( J)^2\}   \, t^{-1/2-\frac{\eta}{4}}.
\end{aligned}
$$
The same estimate is true for $J'$. This completes the proof of Corollary \ref{correlation.good}.
\end{proof}

Moreover, we also have the following crucial corollary for the bootstrap argument in Subsection~\ref{Proof_bn}. Recall that $l,l_0,l_1,n$ and $W$ are chosen as in Section \ref{sec.partition}. 

\begin{corollary}\label{alleb}  For every 
interval   $\bar I \in \cI_k$ and for all $s\in \R^+$  such that  $I:= \bar I \times \{s\} \subset M$,  for  all $t\in [l_0,l_1]$,  we have 
\begin{equation}\label{eq:leb}
\left|\int_{ \bar I }f(T^t_{\a,\vphi}(\theta,s))g(\theta,s)d \theta \right|\leq 
C \{ \Nzero(f,g) \lambda ( I)   + \None(f,g) \lambda ( I)^2 \}
 t^{-1/2+6\eta}.
\end{equation}
\end{corollary}
\begin{proof} If $I \cap W^c \neq \emptyset $, then $I \subset M_\zeta^c$ hence (LHS) is $0$. If $I \subset W$ then let $I=J_1\sqcup J_2\sqcup I_{bad}$ as in Proposition \ref{good.dec}. We  apply Corollary \ref{correlation.good} to $J_1$ and $J_2$ together with the estimates 
$$
\lambda ( I)\geq q_n\log^{-20}q_n \quad \text{ and }\quad \lambda ( I_{bad})<\frac{1}{q_n^{3/2-2\eta}}\,.
$$
For the interval $I_{bad}$ we estimate the integral by the uniform norm of the integrand times the measure $\lambda ( I_{bad})$ of the domain
of integration.
 \end{proof}

\subsection{Summable decay on good intervals. Proof of Proposition \ref{tdec}}
\label{Proof_tdec}

We  now explain how the results of Section \ref{sec.uniformstretchandcorrelations} imply Proposition \ref{tdec}.

In fact, we prove a more general statement that will be relevant in Subsection~\ref{sec6.kochergin},  to complete the proof that 
the spectrum is Lebesgue with countable multiplicity.

 \begin{proposition} \label{tdecbis} For every set $E$, measurable with respect to the partition 
$W$ (see \eqref{defw} for its definition), we have
$$
\left|\int_{E\setminus \mathcal B_l}f(T^t_{\a,\vphi}(x))g(x)d\mu\right| <   C \{ \Nzero(f,g) \mu(E) + \None(f,g) \mu(E)^2\} \,t^{-1/2-\frac{\eta}{5}}\,.
$$
\end{proposition}

\begin{proof}
Since $g=0$ on $M^c_\zeta \supset W^c$, we have
$$\left|\int_{E\setminus \mathcal B_l}f(T^t_{\a,\vphi}(x))g(x)d\mu\right|=\left|\int_{(E\cap W)\setminus \mathcal B_l}f(T^t_{\a,\vphi}(x))g(x)d\mu\right|.
$$
  
By Fubini, it is enough to show that, for every interval $I\subset W$, we have
$$
\left|\int_{\bar I\setminus \bar I_{bad}}f(T^t_{\a,\vphi}(\theta,s))g(\theta,s)d \theta\right|\leq  
C \{ \Nzero(f,g)\lambda ( I) + \None(f,g) \lambda ( I)^2 \} \, t^{-1/2-\eps},
$$
where the subinterval $I_{bad}$ is as in Proposition \ref{good.dec}. It is then enough to apply  Corollary \ref{correlation.good} (to the subintervals $J_1$ and $J_2$) together with the lower bound 
$\lambda ( I)\geq q_n\log^{-20}q_n$.

Proposition \ref{tdecbis} is thus proved, and Proposition \ref{tdec} immediately follows, as 
among the properties of the bad set (see Proposition \ref{prop.badset}) we have the bound 
$\mu (\mathcal B_l) \leq q_n^{-1/2 + 6 \eta}$.
\end{proof}

\subsection{Averaged decay on the bad set. Proof of Proposition \ref{bn}}
\label{Proof_bn}
Notice that as the bad set $\mathcal B_l$ decomposes by \eqref{def:b} as the union of the towers $U_1$, ...,
$U_m$,  Proposition \ref{bn} follows by the proposition below. 

Let $C_{f,g}$ denote a positive constant which depends on the functions $f\in \cF$ and $g\in C^1_0(M)$ only through the quantities  $\Nzero(f,g)$ and $\None(f,g)$. 

\begin{proposition}\label{towdec} For every $i\in\{1,\dots,m\}$, we have 
$$
\int_{l_0}^{l_1}\left|\int_{U_i}f(T^t_{\a,\vphi}(x))g(x)d\mu\right|dt< C_{f,g}\frac{(l_1-l_0)\mu(U_i)}{q_n^{20\eta}}.
$$
\end{proposition}

\begin{proof} Fix $i\in\{1,\dots,m\}$. Let $A:=\{t\in [l_0,l_1]\;:\;\int_{U_i}f(T^t_{\a,\vphi}(x,s))g(x,s)d x>0\}$.
Let $\rho(t)=1$ if $t\in A$ and $\rho(t)=-1$ if $t\in [l_0,l_1]\setminus A$. Then, by Cauchy-Schwarz (H\"older) inequality, we have
\begin{multline*}
\int_{l_0}^{l_1}\left|\int_{U_i}f(T^t_{\a,\vphi}(x))g(x)d\mu\right|dt=\int_{U_i}\left(\int_{l_0}^{l_1}\rho(t)
f(T^t_{\a,\vphi}(x))dt\right)g(x)d\mu \\ \leq
\left(\int_{U_i}\left(\int_{l_0}^{l_1}\rho(t)
f(T^t_{\a,\vphi}(x))dt\right)^2d\mu\right)^{1/2}
\left(\int_{U_i}g(x)^2d\mu\right)^{1/2}\\ \leq 
 \Vert g \Vert_0 \mu(U_i)^{1/2}
\left(\int_{U_i}\left(\int_{l_0}^{l_1}\rho(t)f(T^t_{\a,\vphi}(x))dt\right)^2d\mu\right)^{1/2}.
\end{multline*}
Moreover we have  
\begin{multline*}
\left(\int_{U_i}\left(\int_{l_0}^{l_1}\rho(t)
f(T^t_{\a,\vphi}(x))dt\right)^2d\mu\right)\leq \Vert f \Vert_0^2\, (l_1-l_0)^{3/2}\mu(U_i)\\ +
\left(\int_{U_i}\left(\int_{l_0}^{l_1}\left(\int_{r\in[l_0,l_1]\;:\;|r-t|\geq (l_1-l_0)^{1/2}}
\rho(r)\rho(t)f(T^t_{\a,\vphi}(x))f(T^r_{\a,\vphi}(x))dr\right)dt\right)d\mu\right). 
\end{multline*}
   Therefore,  
to finish the proof of Proposition  \ref{towdec} it is enough to show that  there exists a constant $C>0$ such that, 
for every $t\leq r$ with $t,r \in [l_0,l_1]$ s.t. $|t-r|\geq (l_1-l_0)^{1/2}$, we have
\begin{equation}\label{cortow}
\left|\int_{T_{\a, \vphi}^{t}(U_i)}f(x)f(T^{r-t}_{\a,\vphi}(x))
d\mu\right|\leq C \None (f, f)
\frac{\mu(U_i)}{q_n^{40\eta}}.
\end{equation} 

Note that $\bt:=r-t \in [q_n^{\frac{1}{41}},q_n^{\frac{1}{19}}]$. Let us then fix such a $\bt \in  [q_n^{\frac{1}{41}},q_n^{\frac{1}{19}}]$. Following the notation of Section~\ref{sec.partition} we then let 
$\bl=[\bt]$ 
and  $\bn$ be the unique integer such that $q_\bn \leq \bl<q_{\bn+1}$.

Let $\bk$ be any integer such that  $q_{\bk} \in [q_\bn\log ^{15} q_\bn, q_\bn\log^{20} q_\bn]$.
It follows by construction that we have $q_\bk \in [q_n^{\frac{1}{41}},   q_n^{\frac{1}{19}} \log^{20} q_n]$

Observe now that by Corollary \ref{alleb} there exists a constant $C>0$ such that,  for any interval  $\bar I \in \cI_{\bk}$ and for all $s\in \R^+$  such that  $I:= \bar I \times \{s\} \subset M$, we have  
\begin{equation}\label{fint3}
\left|\int_{\bar I}f(T^{\bt}_{\a,\vphi}(\theta,s))f(\theta,s)
d \theta \right|\leq C  \{\Nzero(f,f) + \None(f,f)\lambda ( I)\} \frac{\lambda ( I)}{q_n^{\frac{1}{100}}}.
\end{equation}
 Thus, it only remains to be seen that the integral in \eqref{cortow} decomposes into integrals over the sets of the form 
 $T_{\a, \vphi}^{t}(U_i) \cap I, \bar I \in \cI_\bk$, and that each is roughly equal to the product of  $\frac{\lambda(U_i\cap I)}{\lambda ( I)}$ times the integral in \eqref{fint3}. This is what we will now derive from Proposition~\ref{prop.badset}, 
 namely from the property that $T_{\a,\varphi}^{t}(U_i)$ is almost equal to the tower $\cT_{t,i}$ of $(B_5)$. In fact, by properties $(B_1)$, $(B_2)$ in Proposition~\ref{prop.badset}, we have the bound $m\leq q_n^{2/5 + \eta}$, hence by property $(B_5)$ we conclude that
\begin{equation}
\label{eq:tower_error}
\sum_{i=1}^m \mu (\cT_{t,i} \triangle T_{\a,\varphi}^{t}(U_i)) \leq  q_n^{-3/5 + 15\eta}  \,.
\end{equation}
The intersection of each tower $\cT_{t,i}$ with $I$ is a regular union of equally separated small intervals (see Figure \ref{wykres1}).  In this 
situation  the interpolation between the integrals is possible.   To  carry it out, we introduce the following 

\begin{definition} \label{def.independence} Let $\nu, \gamma\in (0,1)$. We will say that a collection  $\cS:=K_1\sqcup \ldots \sqcup K_{H}  \subset \T \times \{s\}$  of pairwise disjoint horizontal intervals of equal lengths is $(\nu,\gamma)$-uniformly distributed in the interval~$I$  if there exists a decomposition of~$I$ into a disjoint union of $L \leq \gamma H$ intervals $I_1,\ldots, I_L$ of equal length $\ell \in [\nu,2\nu]$ such that, for all $j\in [1,L]$, we have 
 $$\# \{i \in [1,H] :  K_i \subset  I_j \} \in [( 1-\gamma)\frac{H}{L}, (1+\gamma)\frac{H}{L}]\,.$$
 \end{definition}

 This definition is useful in the following straightforward lemma.
 \begin{lemma} \label{lemma.extrapol}  If $\cS$ and $I$ are as in Definition \ref{def.independence}, then for any $C^1$ real function $G$ defined over the interval $I:= \bar I \times \{s\}$, we have  
 $$\left| \int_{ \bar\cS\cap \bar I} G(\theta,s) d\theta - \frac{\lambda ( \cS \cap  I)}{\lambda(I)} \int_{\bar I} G(\theta,s) 
 d\theta \right| \leq C \left( \nu \|G\|_1 + \gamma \|G\|_0\right)  \lambda(\cS \cap I )\,.  $$
\end{lemma}
 
  \begin{lemma}  \label{lemma.independence} For any complete tower $\cT$ of height $h\geq q_n^{3/5-1/50}$ above any horizontal  interval of the the form
  $B_\cT=[-\frac{1}{q_n^{3/2-5\eta}}+\theta_\cT,\theta_\cT+\frac{1}{q_n^{3/2-5\eta}}] \times \{s_\cT\}$, we have the following:
 
\begin{itemize}
\item[$(I_1)$] if $N(\theta_\cT,h)\leq q_n^{1/3}$, then $\mu(\cT \cap M_\zeta)\leq q_n^{1/2-3/5}\mu(\cT)$\,;
 \item[$(I_2)$]  if $N(\theta_\cT,h)\geq q_n^{1/3}$, then for any $\bar I \in  \cI_\bk$ such that $I := \bar I \times \{s\}
 \subset M_\zeta$,   the set $ \cT \cap I$ is contained in a collection of disjoint intervals of equal  size 
 $(q_n^{-1/4},q_n^{-1/100})$-uniformly distributed in the interval $I$.
 \end{itemize}
 \end{lemma}

\begin{figure}[htb] 
 \centering
  \resizebox{!}{6cm}{\includegraphics[angle=0]{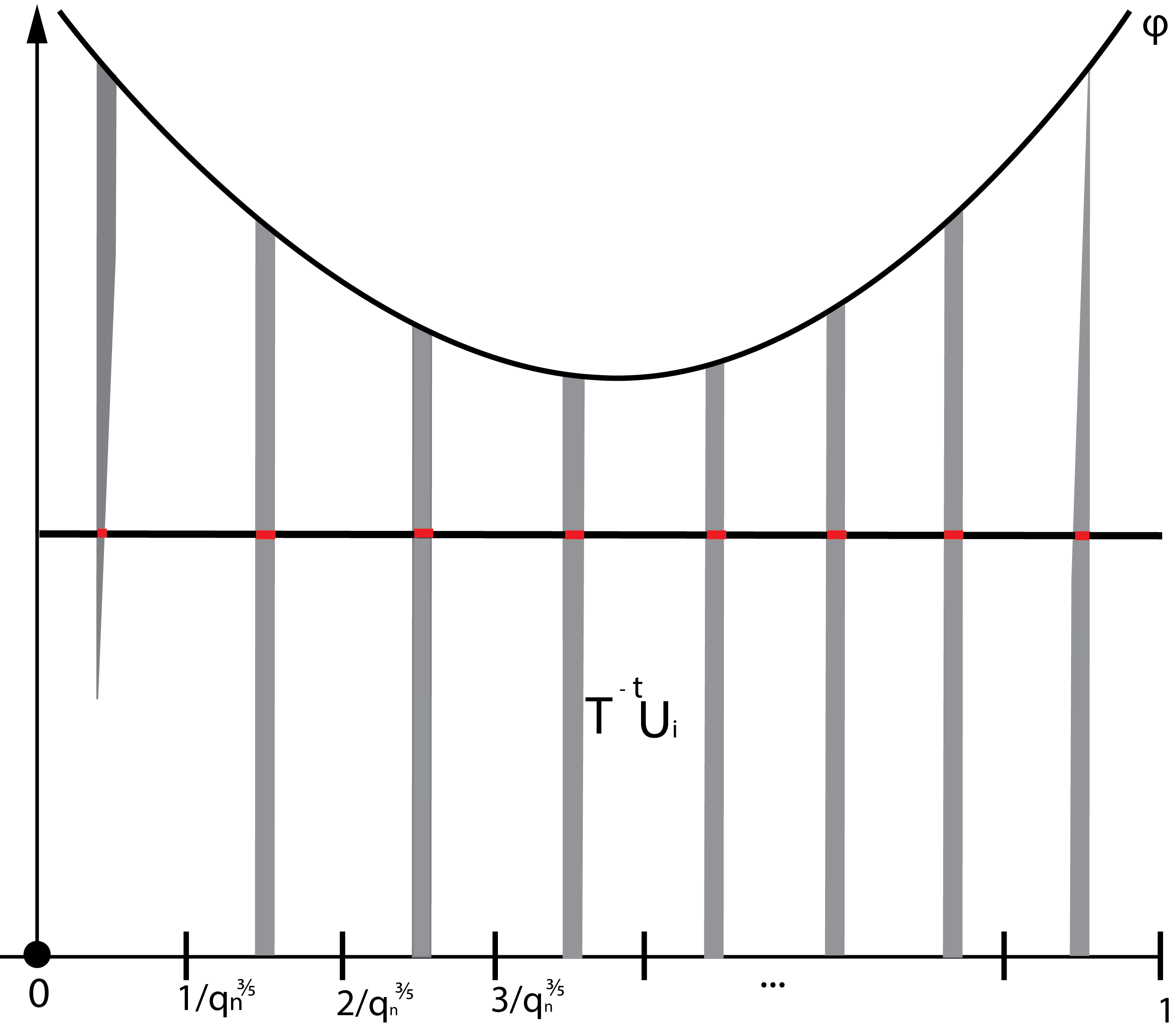}}
\caption{\small The image of the set $U_i$ under the flow. The intersection with any horizontal interval is a union of equispaced intervals.} 
\label{wykres1}
\end{figure}

Before proving Lemma \ref{lemma.independence}, we  show how it implies \eqref{cortow}. By \eqref{eq:tower_error}, it suffices to show that there exists a constant $C>0$  such that 
\begin{equation}\label{cortow2}
\left|\int_{\cT_{t,i} }f(x)f(T^{\bt}_{\a,\vphi}(x)
d\mu \right|\leq C \None(f,f) 
\frac{\mu(\cT_{t,i}) }{q_n^{50\eta}}.
\end{equation} 

\smallskip
If $(I_1)$ holds, then  since $f$ is supported on $M_\zeta$ we have
\begin{equation}
\label{eq:small_tower}
\left|\int_{\cT_{t,i}  }f(x)f(T^{\bt}_{\a,\vphi}(x))
d\mu \right| \leq \Vert f \Vert_0^2 \,
\mu(\cT_{t,i}\cap M_\zeta)\leq \Vert f \Vert_0^2 \,
\frac{\mu(\cT_{t,i})}{q_n^{\frac{1}{10}}},
\end{equation}
hence the proof is finished in this case.
Notice that by Fubini's theorem \eqref{cortow2} follows from the following claim: there exists a constant 
$C>0$ such that, for any $I:= \bar I \times \{s\}$ with $\bar I \in \cI_{\bk}$, we have 
\begin{equation}\label{cortow3}
\left|\int_{\overline{\cT_{t,i} \cap  I}}f(\theta,s)f(T^{\bt}_{\a,\vphi}(\theta,s))
d\theta\right|\leq C \{\Nzero(f,f) + \None(f,f) \lambda ( I)\}
\frac{\lambda(\cT_{t,i} \cap  I)}{q_n^{50\eta}}.
\end{equation} 
In fact, the above bound is stronger than what we need to prove the absolute continuity of the spectrum.
The precise dependence of the constant on the function $f$ and on the interval $I\in \cI_{\bk}$ will be
crucial in the proof, in Subsection~\ref{sec6.kochergin}, that the spectrum is Lebesgue with countable multiplicity .

Now, if $I\subset M_\zeta^c$ 
then the integral in \eqref{cortow3} is zero. Notice that, since $\bt \leq q_n^{1/19}$, by Lemma~\ref{fi} the function 
$G: I\to \R$ defined as $G(\cdot) = f(\cdot)f(T_{\a,\vphi}^{\bt}(\cdot))$ satisfies $\|G\|_{1} \leq q_n^{1/8} \Vert f \Vert_0 \Vert f \Vert_1$, thus $(I_2)$ and Lemma \ref{lemma.independence} imply that 
$$\left| \int_{\overline{\cT_{t,i} \cap  I}} G(\theta,s) d\theta - \frac{\lambda(\cT_{t,i} \cap  I)}{\lambda(I)} \int_{\bar I} G(\theta,s)) d\theta \right| \leq C \Vert f\Vert_0 \{\Vert f \Vert_0  + \Vert f \Vert_1 
\lambda ( I) \} \frac{\lambda(\cT_{t,i} \cap I)}{q_n^{\frac{1}{200}}},$$
and therefore \eqref{cortow3} follows from \eqref{fint3}. The proof of the derivation of the bound in \eqref{cortow}
from Lemma~\ref{lemma.independence}  is complete.

\smallskip

It only remains to give the
\begin{proof}[Proof of Lemma \ref{lemma.independence}]

Let us first consider the case $N:=N(\theta_\cT,h)\geq q_n^{1/3}$.  Let $\{K_1, \dots, K_H\}$ be the smallest
collection of disjoint intervals of equal length such that 
$$
I \cap \cT \subset  K_1\sqcup K_2 \sqcup \dots \sqcup  K_{H}.
$$
Notice that for every $i\in\{1,\dots,H\}$, the interval $\bar K_i$ is centered at the point  $\theta_\cT+k_i\a$, 
for some $k_i\in [0,N]$. In fact, there is an injective map from the set of  $k\in [0,N]$ such that  $\theta_\cT+k \a 
\in \bar I$ to the collection of intervals $\{K_1, \dots, K_H\}$ which misses at most  $2$ intervals. By 
Lemma~\ref{koksma2}  for $\bar J=\bar I$ and  $\theta=\theta_\cT$, we have
\begin{equation}
\label{eq:DK_1}
\vert  H -  N \lambda ( I) \vert  \leq   2+ 2 C^{-1} \log N^{2+\xi}\,.
\end{equation}
Let us then divide $I$ into equal intervals $I_1,\dots, I_L$ of equal length $\ell\in [q_n^{-1/4},2q_n^{-1/4}]$ 
and let us consider $I_j\subset I$. The map from the set $\{i \in [1,H] :  K_i \subset I_j \}$ to the set of  $k\in [0,N]$ such that  $\theta_\cT+k \a \in \bar I_j$, which sends every interval $\bar K_i$ to its center, is injective and misses at most  $2$ elements. From Lemma~\ref{koksma2}  for $\bar J=\bar I_j$ and $\theta=\theta_\cT$, 
it follows that
\begin{equation}
\label{eq:DK_2}
\vert  \# \{i \in [1,H] :  K_i \subset I_j \} - N \lambda ( I_j) \vert \leq 2+ 2 C^{-1} \log N^{2+\xi}\,.
\end{equation}
Notice that since $I \in \cI_k$ by the bound~\eqref{eq:DK_1}, it follows that $H \geq q_n^{1/3-1/20}$ and 
by construction we have $L \leq q_n^{1/4-1/40}$, hence in particular $H/L \geq  q_n^{1/12- 1/40}$. We then derive the estimate
$$
\vert N \lambda ( I_j) - \frac{H}{L} \vert = \vert \frac{N \lambda ( I)}{L} - \frac{H}{L} \vert 
 \leq  \frac{ 2+ C^{-1} \log N^{2+\xi}} {L} \leq   q_n^{-1/10} \frac{H}{L}  \,,
$$
which in turn by the bound \eqref{eq:DK_2} implies that
$$
 \# \{i \in [1,H] :  K_i \subset I_j \} \in \left[(1-q_n^{-1/100})\frac{H}{L}, (1+q_n^{-1/100}) \frac{H}{L}\right].
$$
This shows that the collection $\cS=K_1\sqcup\dots\sqcup K_H$ is $(q_n^{-1/4},q_n^{-1/100})$-uniformly distributed 
in $I$. The proof of Lemma \ref{lemma.independence} is finished in case $(I_2)$.

Assume now that $N(\theta_\cT,h)\leq q_n^{1/3}$. Notice that, since the height of the complete tower $\cT$ is $h\geq q_n^{3/5-1/10}$, we have
$$
\vphi_{N(\theta_\cT,h)+1}(\theta_\cT)\geq q_n^{3/5-1/50}.
$$
But then 
$$
\mu(\cT\cap M_\zeta)\leq q_n^{1/3}\zeta^{-1}\lambda ( B_\cT)\leq q_n^{1/2-3/5}q_n^{3/5-1/50}\lambda ( B_\cT)\leq q_n^{1/2-3/5}\mu(\cT).
$$
This finishes the proof of Lemma \ref{lemma.independence}.
\end{proof}

\end{proof}

\bg 
\section{Countable Lebesgue Spectrum} \blk
\label{subsec:CLS} 

In this section we prove a general criterion for establishing the countable Lebesgue spectral property
for smooth flows with square-integrable correlations of smooth coboundaries. From our criterion we derive
that our Kochergin flows have countable Lebesgue spectrum, thereby completing the proof of our main result,
Theorem~\ref{main} (the precise formulation of Theorem \ref{statement}). We also derive that time-changes of horocycle flows have countable Lebesgue spectrum, thereby completing the proof of the Katok-Thouvenot conjecture (see~\cite{KT}, Conjecture 6.8). In fact, it was proved in  \cite{FU}  that smooth time changes of the horocycle flow have Lebesgue maximal spectral type, but the multiplicity question was left open. 

{The section will be divided in three parts. In the first Subsection~\ref{secCILS}, we give in 
 Theorem~\ref{thm:Criterion}  an abstract Criterion for Infinite Lebesgue Spectrum (CILS),  that guarantees infinite Lebesgue multiplicity for a strongly continuous group of unitary operators on a Hilbert space having an absolutely continuous spectral type. 

 To guarantee that the multiplicity of the Lebesgue component in the spectrum is at least $n+1$, for some $n\geq 0$, the criterion requires, for any given  positive measure and bounded subset $C$ of the real line, the construction of $n+1$- functions such that the $(n+1)\times (n+1)$ matrix of Fourier transforms of their square-integrable mutual correlations has maximal rank equal to $n+1$ on $C$. The latter would indeed contradict that all the (equivalence classes of the) spectral measures in the decreasing spectral decomposition be zero on $C$ starting from the ${n+1}^{\rm st}$ measure. 

An equivalent way of presenting the hypotheses of the CILS, is to require the existence, for any $n\geq 0$,  of $n+1$ functions, such that each function is almost orthogonal to the cyclic space of any other one, and such that the spectral measures of the functions can be chosen to be not too small on any fixed bounded measurable set 
of $\R$.

In Subsection~\ref{sec6.decay}, we state in Theorem~\ref{thm:CLS}  a criterion that guarantees infinite Lebesgue multiplicity for a flow, based on the control of the decay of correlations for functions supported on tall flow-boxes with an arbitrarily thin base.  When mixing between such functions is effectively obtained at times that compares to the height of the flow-boxes, it is then possible to construct the functions as in the CILS and conclude infinite Lebesgue multiplicity. 
Indeed, we show in the same subsection  how the hypotheses of Theorem \ref{thm:CLS}  immediately imply the hypotheses of the abstract criterion in Theorem \ref{thm:Criterion}.

In Subsection~\ref{sec6.kochergin}, elaborating on the mixing estimates of Sections \ref{sec.partition} and \ref{sec.correlations}, we show that Kochergin flows (with a sufficiently degenerate singularity) typically satisfy the hypotheses of our criterion, thus completing the proof of Theorem \ref{main} (hence of Theorem~\ref{statement}), our main result. 

\smallskip
We also explain  how to derive from \cite{FU} that smooth time changes of horocycle flows satisfy the hypothesis of Theorem \ref{thm:CLS}, our criterion for infinite Lebesgue multiplicity for flows.  Since by the results of \cite{FU}  the maximal spectral type is Lebesgue, we conclude that the smooth time changes of horocycle flows also have a Lebesgue spectrum with infinite  multiplicity.

 \subsection{The Criterion for Infinite Lebesgue Spectrum (CILS)}
 \label{secCILS} }

Our criterion for countable Lebesgue spectrum of smooth flows is based on the following abstract criterion for strongly continuous one-parameter unitary groups on Hilbert spaces. 

Let $\mathcal F : L^2(\R,dt) \to L^2(\R,d\tau)$ denote the Fourier transform, given by the formula
$$
\mathcal F(f)(\tau) = \int_\R  f(t) e^{-2\pi \imath t\tau }   dt  \,, \quad \text{ for all } f\in L^2(\R, dt)\,.
$$
\begin{theorem} \label{thm:Criterion} Let $\{\phi_\R\}$ be a strongly continuous one-parameter unitary group on a Hilbert space $H$ with absolutely 
continuous spectrum. {For a fixed $n\in \N$, let us assume that  for every compact set $C\subset \R\setminus \{0\}$ of positive Lebesgue measure there exists $\epsilon_{n,C}>0$ such that 
the following holds.  For every $\epsilon \in (0, \epsilon_{n,C})$ 
there exist vectors $f_1, \dots, f_{n+1} \in H$  such that
$$
\begin{aligned} 
&\Vert \langle  \phi_t (f_i),  f_j\rangle  \Vert_{L^2(\R, dt)} \leq  \delta_{ij} +  \epsilon\,, \quad 
\text{ for all } i, j \in {1, \dots, n+1}\,; \\
&\Vert \prod_{i=1}^{n+1} \mathcal F(\langle \phi_t(f_i),  f_i\rangle) \Vert_{L^{\frac{2}{n+1}}(C)} > 
 (n+1) !   (1+ \epsilon)^n  \epsilon    
 \,.
\end{aligned}
$$
}
Then the spectral type of $\{\phi_\R\}$ is Lebesgue with multiplicity at least $n+1$.
\end{theorem}

The proof of the theorem is based on the following lemma. 

\begin{lemma}
\label{lemma:spectral_det}
Let $H$ be a Hilbert space  and let $H^{(n)}:=\oplus_{k=1}^n H_k \subset H$ denote an orthogonal, invariant decomposition into cyclic subspaces of a strongly continuous one-parameter unitary group $\{\phi_\R\}$ with absolutely continuous spectrum.  Let $f_1, \dots, f_{n+1} \in H^{(n)}$ be vectors such that the correlations
functions  $\langle  \phi_t (f_i), f_j \rangle  \in L^2(\R, dt)$. Let $\mathcal F:L^2(\R) \to L^2(\R)$ denote the
the Fourier transform. We have 
$$
\det ( \mathcal F \langle \phi_t(f_i), f_j \rangle ) = 0   \quad \text{almost everywhere. }\,
$$
\end{lemma}
\begin{proof}
Let us begin to illustrate the argument in the case $n=1$. 
Then we can assume (up to a unitary equivalence) that there is a function   $m \in L^1(\R, d\tau)$ 
such that $f_1, f_2 \in L^2(\R, md\tau)$. We therefore assume $f_1= f_1(\tau)$, and $f_2=f_2(\tau) \in L^2(\R,
md\tau)$.  The flow acts on $L^2(\R, m d\tau)$ by multiplication by $e^{2\pi \imath \tau t}$. So we have  
$$
\begin{aligned}
{\mathcal F} (\langle  \phi_t(f_i), f_i \rangle)  &=   m \vert f_i\vert^2 \,,  \quad  \text{for } i=1,2\,,   \\  
{\mathcal F} (\langle  \phi_t(f_1), f_2 \rangle) &=  m f_1 \bar{f_2} \,.
\end{aligned}
$$
We have the identity between  functions in $L^1(\R)$:
$$
\begin{aligned}
\mathcal F(   \langle  \phi_t(f_1), f_1 \rangle) \mathcal F  (\langle  \phi_t(f_2), f_2 \rangle) &= m^2 \vert f_1\vert^2 \vert f_2\vert^2
= (m f_1 \bar {f_2} ) (m f_2 \bar {f_1}    ) \\ &= \mathcal F(   \langle  \phi_t(f_1), f_2 \rangle) \mathcal F (  \langle  \phi_t(f_2) ,f_1 \rangle)\,.
\end{aligned}
$$
In the general case, let $f_{ik}$ denote the projection of the vector $f_i$  on the cyclic space 
$H_k$, for all $i\in \{1, \dots, n+1\}$ and $k\in \{1, \dots, n\}$. Since the cyclic spaces are invariant and
orthogonal, for all $i, j\in \{1, \dots, n+1\}$ we have
$$
\langle \phi_t (f_i), f_j \rangle = \sum_{k=1}^n  \langle  \phi_t(f_{ik}), f_{jk} \rangle \,.
$$
Since the spaces $H_k$ are cyclic and the group has absolutely continuous spectrum, there exist  functions $m_1, \dots, m_k \in L^1(\R,d\tau )$ and, 
for each $k\in \{1, \dots, n\}$ and all $i\in \{1, \dots, n+1\}$, there exist functions  $u_{ik} \in L^2(\R, m_k d\tau )$  
such that for all $i, j\in \{1, \dots, n+1\}$ we have
$$
\mathcal F \langle \phi_t (f_i), f_j \rangle  =   \sum_{k=1}^n u_{ik} \bar u_{jk} m_k \,.
$$
By the above formula, every column of the $(n+1)\times (n+1)$ matrix $\left(\mathcal F \langle \phi_t (f_i), f_j \rangle (\tau )\right)$ can be written as the sum of $n$ vectors as follows. For each $j\in \{1, \dots, n+1\}$, we have
$$
\begin{pmatrix} \mathcal F \langle  \phi_t( f_1), f_j \rangle (\tau ) \\ \dots \\ \mathcal F \langle  \phi_t(f_{n+1}), f_j \rangle(\tau )  \end{pmatrix} = \sum_{k=1}^n  m_k(\tau ) \begin{pmatrix}  u_{1k}(\tau ) \bar u_{jk} (\tau )  \\ \dots \\ 
u_{(n+1)k}(\tau ) \bar u_{jk}(\tau )   \end{pmatrix} =  \sum_{k=1}^n  m_k(\tau ) \bar u_{jk} (\tau ) 
 \begin{pmatrix}  u_{1k}(\tau )  \\ \dots \\ 
u_{(n+1)k}(\tau )   \end{pmatrix}    \,.
$$
Since the matrix $\left(\mathcal F \langle \phi_t (f_i), f_j \rangle(\tau )\right)$ is $(n+1)\times (n+1)$, its determinant is
a sum of determinants of matrices containing at least two columns proportional to the same vector
$(u_{1k} (\tau ), \dots, u_{(n+1) k}(\tau ) ) $. This proves that the determinant vanishes almost everywhere.
The argument is complete.

\end{proof}

\begin{proof} [Proof of Theorem \ref{thm:Criterion}]
Let $\oplus_{n\in \N} H_n$ denote an orthogonal, invariant decomposition into cyclic subspaces
such that, for all $n\in \N$, we have $H_n \approx L^2(\R, \mu_n)$  with
$$
\mu_1:= m_1 (\tau) d\tau \gg \mu_2:= m_2 (\tau) d\tau \gg  \dots \gg \mu_n:= m_n (\tau) d\tau \gg \dots
$$
Let $\{f_i\}$ be a sequence of vectors in $H$ and, for each $i$, $j\in \N$, let $f_{ij}$ denote the
orthogonal projection onto $H_j$. Since the spaces $H_k$ are $\phi_\R$-invariant and mutually orthogonal,
we have, for all $i, j \in \N$, 
$$
\langle  \phi_t(f_i),  f_j \rangle = \sum_{k\in\N} \langle  \phi_t(f_{ik}),  f_{jk}\rangle \,,
$$ 
hence after taking the Fourier transform
$$
\mathcal F(\langle  \phi_t( f_i),  f_j \rangle) = \sum_{k\in\N} \mathcal F(\langle  \phi_t(f_{ik}),  f_{jk}\rangle)\,.
$$
Let us assume by contradiction that the spectrum is Lebesgue with multiplicity at most $n$. Then
there exists a  compact set $C \subset \R\setminus\{0\}$ of positive Lebesgue measure such that
$$
m_{n+1}  (\tau) = m_{n+2}(\tau)  = \dots =0 \,, \quad \text{ for all } \tau \in C\,.
$$
Let $f_1, \dots, f_{n+1} \in H$ be vectors given by the assumptions of the theorem 
Let $\bar f_1, \dots, \bar f_{n+1}  \in H^{(n)} :=H_1 \oplus \dots \oplus H_n$
denote, respectively, the orthogonal projections of vectors $f_1, \dots, f_{n+1} \in H$ onto the subspace $H^{(n)}$. Since for each $k\in \N$ the subspace $H_k$ is cyclic, the Fourier transform of the correlation $\langle  \phi_t(f_{ik}),  f_{ik} \rangle$ is absolutely continuous (as a density) with respect to the measure
$\mu_k$ on $\R$. Hence we derive, for all $i, j\in \{1, \dots, n+1\}$ the identity
$$
\mathcal F (\langle    \phi_t (f_i),   f_j \rangle) (\tau) =  \mathcal F (\langle  \phi_t( \bar f_{i}) ,  \bar f_{j}\rangle)(\tau)\,,
\quad \text{ for almost all } \tau \in C\,.
$$
It follows that, by Lemma \ref{lemma:spectral_det},  for all $i,j\in \{1, \dots, n+1\}$, we have that
\begin{equation}
\label{eq:det_id}
\det ( \mathcal F ( \langle  \phi_t (f_i),  f_j \rangle)(\tau   )  = 
\det ( \mathcal F ( \langle \bar \phi_t (f_i),  \bar f_j \rangle)(\tau   ) =0 \,, \quad \text{ for almost all } \tau \in C\,.
\end{equation}
{  By H\"older inequality, for any $p>1$ the product of $n+1$ functions in $L^p(\R)$ belongs to  $L^{\frac{p}{n+1}} (\R)$ and we have
\begin{equation}
\label{eq:Holder}
\Vert    \prod_{i=1}^{n+1}  g_i   \Vert_{L^{\frac{p}{n+1}}(\R)} \leq \prod_{i=1}^{n+1}  \Vert   g_i \Vert_{L^p(\R)}\,.
\end{equation}
Since  the determinant of a $(n+1) \times (n+1)$ matrix is a polynomial of degree $n+1$ in the entries
of the matrix, the determinant of the matrix $(\mathcal F \langle \phi_t (f_i), f_j \rangle)$ 
belongs to $L^{\frac{2}{n+1}} (\R)$. }

By the assumptions on the vectors $f_1, \dots, f_{n+1}$ we have
$$
\Vert \mathcal F \langle \phi_t (f_i), f_j\rangle \Vert_{L^2(C, d\tau)} \leq  
\Vert \mathcal F \langle \phi_t (f_i), f_j \rangle \Vert_{L^2(\R, d\tau)} = \Vert \langle \phi_t (f_i), f_j \rangle\Vert_{L^2(\R, dt)} \leq \delta_{ij} + \epsilon\,,
$$
hence, by  formula \eqref{eq:det_id}, by the expansion of the determinant, and by the estimate in formula
\eqref{eq:Holder}, 
\begin{equation}
\label{eq:epsilon}
\begin{aligned}
\Vert  \prod_{i=1}^{n+1} \mathcal F ( \langle  \phi_t (f_i),  f_i\rangle) \Vert_{L^{\frac{2}{n+1}}(C)}&= \Vert \det  \mathcal F ( \langle  \phi_t (f_i),  f_j \rangle) -  \prod_{i=1}^{n+1} \mathcal F ( \langle  \phi_t (f_i),  f_i\rangle) \Vert_{L^{\frac{2}{n+1}}(C)} \\ &\leq  (n+1) !   (1+ \epsilon)^n  \epsilon\,.
\end{aligned} 
\end{equation}
However, by assumption we also have
$$
\Vert  \prod_{i=1}^{n+1} \mathcal F ( \langle  \phi_t (f_i),  f_i\rangle) \Vert_{L^{\frac{2}{n+1}}(C)} 
> (n+1) !   (1+ \epsilon)^n  \epsilon   \,,
$$
a contradiction with the upper bound in formula \eqref{eq:epsilon}. The argument is thus complete.
\end{proof}

We give now a version of the CILS { that is well adapted to derive countable Lebesgue spectrum from mixing estimates for Kochergin flows and horocycle flows} (that is, from Theorem \ref{thm:CLS} below). 

\begin{corollary} \label{cor:criterion} Let us assume that for every $n \in \N$, for any even functions $\omega_1, \dots, \omega_{n+1} \in \mathcal S(\R)$ (the Schwartz space), and for any any~$\epsilon >0$,  there exists vectors $f_1, \dots, f_{n+1} \in H$ such that, for all $i, j \in \{1, \dots, n+1\}$, we have
$$
\Vert \langle  \phi_t(f_i),  f_j \rangle  -  \frac{d^2}{dt^2} \omega_i \ast \omega_i(t) \delta_{ij}  \Vert_{L^2(\R)} \leq \epsilon\,.
$$
Then the spectral type of the strongly continuous one-parameter unitary group $\phi_\R$ is Lebesgue with countable multiplicity.
\end{corollary} 
\begin{proof}  Let $C$ be a given compact subset of $\R\setminus \{0\}$ of positive Lebesgue measure.
By the Lebesgue density theorem, it is not restrictive to assume that there exists an interval $[a, b]$ with $0<a<b$
 such that $\text{Leb} (C \cap [a,b]) \geq (b-a)/2$.  The case when $C \cap \R^+ =\emptyset$ is similar.
 Let $\chi_C: \R \to [0, 1]$ denote any  smooth odd function with compact support in $[-2b, -a/2] \cup [a/2, 2b]$ such that $\chi^2_C \equiv 1$  on $[-b,-a] \cup [a,b]$.  For all $i\in \{1, \dots, n+1\}$ let  $\omega_i$ be the function determined by the identity
 $$
 \mathcal F ( \omega_i ) (\tau) = \frac{1}{\tau}   \frac{ \chi_C (\tau)}{\Vert \chi^2_C\Vert^{1/2}_{L^2(\R)}} \,, \quad \text{ for all } i\in \{1, \dots, n+1\}\,.
 $$
 The functions $\omega_i$ are all even, and we can take $f_1,\ldots,f_{n+1}$ as in the statement of the corollary. We then verify  that the hypotheses of Theorem \ref{thm:Criterion}  hold.  In fact, we have
$$
\begin{aligned}
\Vert   \mathcal F( \langle  \phi_t(f_i),  f_j \rangle)  &-  \mathcal F(\frac{d^2}{dt^2} \omega_i \ast \omega_i) \delta_{ij}  \Vert_{L^2(\R)}  \\ &= \Vert \mathcal F(  \langle  \phi_t(f_i),  f_j \rangle)  -  \frac{\chi_C^2 }{ 
\Vert \chi_C^2 \Vert_{L^2(\R)}}  \delta_{ij}  \Vert_{L^2(\R)}  \leq \epsilon\,,
\end{aligned}
$$
hence in particular
$$
\Vert \mathcal F(\langle   \phi_t(f_i),  f_j \rangle) \Vert_{L^2(\R)} \leq \delta_{ij} + \epsilon\,, \quad
\text{ for all } i, j\in \{1, \dots, n+1\}\,.
$$
By the construction and by the H\"older inequality bound of formula \eqref{eq:Holder} we have
$$
\Vert \prod_{i=1}^{n+1}  \mathcal F( \langle  \phi_t(f_i),  f_i \rangle) -  \left( \frac{\chi_C^2 }{ 
\Vert \chi_C^2 \Vert_{L^2(\R)}}   \right)^{n+1}  \Vert_{L^{\frac{2}{n+1}} (\R)} \leq  2^n (1+\epsilon)^{n-1}
\epsilon\,,
$$
hence by convexity we derive that
$$
\begin{aligned}
\Vert \prod_{i=1}^{n+1}  \mathcal F( \langle  \phi_t(f_i),  f_i \rangle) \Vert^{\frac{2}{n+1}}_{L^{\frac{2}{n+1}} (C)}
&\geq   \Vert  \left( \frac{\chi_C^2 }{ 
\Vert \chi_C^2 \Vert_{L^2(\R)}}   \right)^{n+1} \Vert ^{\frac{2}{n+1}}_{L^{\frac{2}{n+1}} (C)} - 
[2^n(1+\epsilon)^{n-1} \epsilon]^{\frac{2}{n+1}  } \\  & \geq
\left(\frac{ \Vert \chi^2_C \Vert_{L^2(C)} }{ \Vert \chi^2_C \Vert_{L^2(\R)}   } \right)^2
  - [2^n(1+\epsilon)^{n-1} \epsilon]^{\frac{2}{n+1}  }\,.
\end{aligned}
$$
From the above estimate we conclude that 
$$
\Vert \prod_{i=1}^{n+1}  \mathcal F( \langle  \phi_t(f_i),  f_i \rangle) \Vert_{L^{\frac{2}{n+1}} (C)} >  
(n+1) !   (1+ \epsilon)^n  \epsilon \,,
$$
for all $\epsilon>0$ such that
$$
\left(\frac{ \Vert \chi^2_C \Vert_{L^2(C)} }{ \Vert \chi^2_C \Vert_{L^2(\R)}   } \right)^2
  > [2^n(1+\epsilon)^{n-1} \epsilon]^{\frac{2}{n+1}} +  (n+1) !   (1+ \epsilon)^n  \epsilon \,.
$$
By Theorem \ref{thm:Criterion} it follows that the strongly continuous one-parameter unitary group $\{\phi_\R\}$ has Lebesgue spectrum
with multiplicity at least $n$. Since $n\in \N$ is arbitrary, it has Lebesgue spectrum with countable multiplicity.
\end{proof}

 \subsection{Decay of correlations and infinite Lebesgue multiplicity.}
\label{sec6.decay}

As explained in the introduction of this section, we now give a criterion based on decay of correlations that allows to construct the functions that as required in the CILS to guarantee infinite Lebesgue multiplicity. The idea is to guarantee mixing between functions supported on tall flow-boxes with thin base $J$ after a time that is comparable to the height of the flow-boxes.  Indeed by fixing such a flow-box with base $J\subset M$ and  height $T_J>0$, we can choose functions supported on this flow-box in an arbitrary way so as to guarantee the satisfaction of the CILS conditions up to some finite time comparable to $T_J$. After time $T_J$ it is the effective mixing between functions supported on such flow-boxes that insures the complete satisfaction of the CILS conditions. 

One additional technical point is that our mixing estimates only hold for coboundaries, hence we have to define  corresponding classes of functions supported on tall and thin flow-boxes. 
For Kochergin flows, one extra technical difficulty is that mixing is effectively controlled only away from the singularity (and for technical reasons related to our proof, away from the ceiling function). Hence the family of functions we need to consider are not just supported on tall flow-boxes with thin bases, but also have to vanish on a small measure set inside these flow-boxes that correspond to a small neighborhood of the origin (and of the ceiling function). {The latter difficulty is not present in our application of the CILS to time changes of horocycle flows. }

Let $\{T^t\}$ be a smooth aperiodic flow on a smooth manifold $M$, preserving a smooth volume form of finite
total volume.  For any given transverse embedded closed  multi-dimensional interval $J \subset M$,  let $T_J$ be the maximal real number $T>0$ such that the map
\begin{equation}
\label{eq:flow_box}
F^T_J (x,t) = T^t (x,0)         \,, \,\, \text{ \rm for all }  (x,t) \in J \times (-T,T)\,,
\end{equation}
is a flow-box for the flow $\{T^t\}$.   The flow-box $F_J := F^{T_J}_J$ will be called a {\it maximal flow-box}
over the (basis) interval $J\subset M$.
Since the flow $\{T^t\}$ has no periodic orbits,  for any $T_0>0$ there exists an interval $J$ such that $T_J>T_0$. 

Let $\cal M:=\{M_\zeta \vert \zeta\in (0,1)\}$ be a fixed family of open subsets of $M$ such that 
$\cap_{\zeta>0} M^c_\zeta$ is a closed subset $M_0$ of zero-measure.  In the case of Kochergin flows
this is the family  introduced in formula~\eqref{Mzeta} of Section~\ref{sec.decay}.
Given a flow-box  {$F_J^T$}, we define, for any $\zeta>0$, the set $S^T_\zeta(J) \subset \R$ 
as follows
$$
S^T_\zeta(J):=\{ t\in (-T,T) \;:\;  T^t (J) \cap M_\zeta^c =\emptyset\}\,.
$$ 
 By definition we have that $S^T_\zeta(J)$ is an open subset (which in general may be empty).

\smallskip
In the sequel, we will focus our attention  on very long and thin maximal flow boxes, that spend most of the initial time away from the bad sets $M_\zeta^c$, for $\zeta>0$ sufficiently small. This motivates the following definition. 

\begin{definition}
\label{def:admissible}
A family $\Phi= \{F_J\}$ of maximal flow-boxes  is called {\em admissible} if for every $T>0$ and $\nu>0$, there exist $N \in \N$ and $\tau>0$, and $\zeta >0$  such that  for all maximal flow-boxes $F_J\in \Phi$  with $T_J > \tau$, the set $S^T_\zeta(J)$ has at most $N$ connected components, and 
 \begin{equation}
 \label{eq:admissible}
{\text{Leb} ((-T,T)\setminus S^T_\zeta(J))} \leq  \nu\,.
 \end{equation}
 \end{definition} 
 

\begin{remark}
 It should be noted that, in the special case when the bases of the maximal flow-boxes of a family 
 $\Phi= \{F_J\}$ form a decreasing sequence $\{J\}$ of intervals with respect to the inclusion, then in order to establish that $\Phi$ is admissible it is enough to verify the conditions for all the degenerate flow-boxes (the orbits segments) over the singleton equal to the intersection of all of their bases. Our construction below of admissible families of maximal flow-boxes for Kochergin flows (in~Section~\ref{sec6.kochergin}) is based on this principle.
\end{remark}

 \smallskip
For any $k\in \N \setminus \{0\}$, and for any constants $C>0, \zeta>0$,  we define $G_k(J,T,C,\zeta)$ to be the set of all functions $\psi_J \in C^\infty_0 (J \times (-T, T)    )$  defined as  follows.  Let  $\chi_J \in C_0^\infty(J)$ be any smooth function such that $\int_J \chi^2_J d\lambda =1$,  with $C^s(J)$ norm bounded above by  $C/\lambda (J)^{s+1/2}$ for all $s\in \{0, \dots, k\}$, and let $\psi \in C^\infty_0 ( S^T_\zeta(J))$ is any smooth function with $C^k$ norm on $\R$ bounded above by $C$. 
 We can now define the functions  supported on flow-boxes that we will be working with. 

\begin{definition}
\label{def:test_functions}
Given a flow-box $F_J^T$ and constants $C,\zeta>0$, we define $\mathcal G_k(F_J^T,C,\zeta)$ to be the class  of all functions  $g_J\in C^\infty(M)$, defined on the range $R^T_J$ of the flow-box map $F^T_J$ as 
\begin{equation}
\label{eq:u_def1}
(g_J \circ   F^T_J) (x, t):=   \chi_J(x)  \psi (t) \,,  \quad  \text{ \rm if }  (x,t) \in J \times (-T,T) \,,\\
\end{equation} 
for any $\psi  \in G_k(J,T,C,\zeta)$, and  defined as $g_J:=0$ on $M \setminus R^T_J$.

\smallskip
The class $\mathcal F_k(F^T_J,C,\zeta) \subset \mathcal G_k(F_J^T,C,\zeta)$ { consists of all functions $f_J
\in \mathcal G_k(F_J^T,C,\zeta)$ which are derivatives along the flow.}
\end{definition} 

We can now state our general criterion, based on correlation decay, for countable Lebesgue spectrum.

\begin{theorem} \label{thm:CLS}  Let $\{T^t\}$ be a smooth, aperiodic, volume-preserving ergodic flow with absolutely continuous spectrum on a smooth manifold $M$ of finite total volume.   Assume that there exists
an admissible  family of maximal  flow-boxes $\Phi:=\{F_J\}$  for the flow, such that $\inf \lambda (J)=0$ (hence 
$\sup  T_J=+ \infty$) and there exists $k \in \N \setminus \{0\}$ such that given any  
$T>0$, $C>0$ and $\zeta >0$, for any family $\{(f_J, g_J)\}$  of pair of functions such that $f_J \in \mathcal F_k(F_J^T,C, \zeta)$ and $g_J \in \mathcal G_k(F_J^T,C,\zeta)$ we have 
$$
\inf_{F^J\in \Phi}   \int_{\R\setminus [-T_J, T_J]} \vert \langle f_J \circ T^t , g_J \rangle\vert^2  dt =0\,.
$$
Then the flow $\{T^t\}$ has countable Lebesgue spectrum. 
\end{theorem} 

We will derive the criterion from Corollary \ref{cor:criterion}. Since we only control the decay of correlations for functions  in the classes $\mathcal F_k(F_J^T,C, \zeta)$ and $\mathcal G_k(F_J^T,C,\zeta)$, we first prove below a simple approximation lemma to approximate the target even functions $\omega_1, \dots, \omega_{n+1} \in \mathcal S(\R)$ of Corollary~\ref{cor:criterion} by (even) functions supported inside sets of the type $S^T_\zeta(J)$. For technical reasons that will appear below in the proof of the approximation lemma, we prefer to first symmetrize the set $S^T_\zeta(J)$ and  consider instead functions supported in $S^T_\zeta(J) \cap (-S^T_\zeta(J))$.

\begin{lemma} \label{lemma.psi} Let $\Phi= \{F_J\}$ be {\em admissible}. Then, for every $k\in \N$,  $\epsilon>0$,  
for every even function $\omega \in \mathcal S(\R)$, there exist $\tau>0$ such that for every $T\geq \tau$, there exists $\zeta>0$ such that  if $J$ is such that $T_J>T$, there exists an even function $\psi \in C^\infty_0 (-T,T) $ such that  $\frac{d\psi}{dt} \in  C^\infty_0 ( S^T_\zeta(J)\cap (- S^T_\zeta(J)))$ with  $C^{k}$ norm \ignore{uniformly} bounded  above by a constant $C:=C(k,\epsilon,\omega,T)>0$ (crucially) independent of the flow-box $F_J\in \Phi$, such that
\begin{equation}  
\label{eq:conv_approx_2}
\Vert \frac{d^2}{dt^2} (\psi \ast \psi) - \frac{d^2}{dt^2} (\omega \ast \omega)  \Vert_{L^2(\R)} <  \epsilon\,. 
\end{equation} 

\end{lemma}
\begin{proof} By properties of convolution we can write
$$
\begin{aligned}
\frac{d^2}{dt^2} (\psi \ast \psi) - \frac{d^2}{dt^2} (\omega \ast \omega)&= 
\frac{d\psi}{dt} \ast \frac{d\psi}{dt}  -  \frac{d\omega}{dt} \ast \frac{d\omega}{dt}
\\ &=(\frac{d\psi}{dt} +  \frac{d\omega}{dt}) \ast (\frac{d\psi}{dt} -\frac{d\omega}{dt}   ) \,,
\end{aligned} 
$$
hence, by Young's convolution inequality,  we have 
\begin{equation}
\label{eq:Young_est}
\Vert \frac{d^2}{dt^2} (\psi \ast \psi) - \frac{d^2}{dt^2} (\omega \ast \omega)  \Vert_{L^2(\R)}  \leq
\Vert  \frac{d\psi}{dt} +  \frac{d\omega}{dt} \Vert_{L^1(\R)} \Vert  \frac{d\psi}{dt} -  \frac{d\omega}{dt} \Vert_{L^2(\R)} \,.
\end{equation}
It is therefore enough to construct functions $\psi$ such that $ \frac{d\psi}{dt}$ are $L^2$ approximations of the function $ \frac{d\omega}{dt}$ with bounded $L^1$ norm, and are supported in $ S^T_\zeta(J)$.

\smallskip
{ By the definition of an admissible family of maximal flow-boxes, for every $T>0$ and $\nu>0$, there exist $\tau, \zeta >0$ and $N \in \N$ such that for all maximal flow-boxes $F_J\in \Phi$  with $T_J>\tau$ we have that $J\subset M_\zeta$, the symmetric set $S^T_\zeta(J) \cap (- S^T_\zeta(J)) =I_1 \cup \ldots \cup I_N$ is a union of (at most) $N$ open intervals  and $\text{Leb}((-T, T) \setminus I_1 \cup \ldots \cup I_N )\leq  \nu/4$.  
Let  then $\{I'_1, \dots, I'_N\}$ be a symmetric family of closed subintervals such that $I'_i \subset I_i$ , for each $i\in \{1, \dots, N\}$,  and 
\be\label{eq:bouv}
\text{Leb}((-T, T) \setminus (I'_1 \cup \ldots \cup I'_N ))\leq  \nu/2\,.
\ee
In order to control the norm of higher derivatives of the function $\psi$, we also choose the family
$\{I'_i\}$ such that, for all $i\in \{1, \dots, N\}$, 
$$
\text{ \rm dist} (\partial I_i, \partial I'_i)  \geq   \frac{\nu}{10 N} \,.
$$
We claim that there exists an even function $\psi \in C_0^\infty ((-T, T))$ such that
\begin{itemize}
\item $\frac{d\psi}{dt}$ is an odd smooth function supported inside $S^T_\zeta(J) \cap (-  S^T_\zeta(J) )$,
\item $\frac{d\psi}{dt}=\frac{d\omega}{dt}$ on $ (I'_1 \cup \dots \cup I'_N) \cap [-T+\nu/4 , T-\nu/4]$,
\item $\vert \frac{d\psi}{dt} (t) \vert \leq  \vert \frac{d\omega}{dt}(t) \vert$,  for all $t\in \R$,
\item $\| \frac{d\psi}{dt} \|_{C^k(\R)} \leq C_k  \| \omega \|_{C^{k+1}(\R) } (\frac{\nu}{10 N})^{-k}$,  for all $k\geq 1$ 
\end{itemize}
(with $C_k>0$ a constant depending only on $k\in \N$).}

It follows from formula \eqref{eq:Young_est}  and by H\"older inequality that  we have
$$
\begin{aligned}
\Vert \frac{d^2}{dt^2} (\psi \ast \psi) - \frac{d^2}{dt^2} (\omega \ast \omega)  \Vert_{L^2(\R)}  \leq
2 \Vert  \frac{d\omega}{dt} \Vert_{L^1(\R)}
\left(  2 \Vert  \frac{d\omega}{dt} \Vert_{C^0(\R)} \nu^{1/2}   +   \Vert \frac{d\omega}{dt} \Vert_{L^2(\R\setminus(-T,T))} \right) \,.
\end{aligned}
$$
Hence for every $\omega \in \mathcal S(\R)$ and for every $\epsilon >0$, there exists
$T_\epsilon>0$ and $\nu_\epsilon>0$ such that, for all $T > T_\epsilon$ and all $\nu<\nu_\epsilon$, we have
$$
\Vert \frac{d^2}{dt^2} (\psi \ast \psi) - \frac{d^2}{dt^2} (\omega \ast \omega)  \Vert_{L^2(\R)} <\epsilon\,.
$$
We have thus reduced the proof of the lemma to that 
of the above claim. 

In order to prove the claim we consider an even function $\phi_\nu \in C_0^\infty((-T,T))$  such that 
\begin{itemize}
\item $\phi_\nu (t) \in [0,1]$, for all $t\in \R$,
\item $\phi_\nu(t) =0$, for all $t \not\in S^T_\zeta(J) \cap (-  S^T_\zeta(J) ) =
I_1 \cup \dots \cup I_N$,
\item $\phi_\nu(t)= 1$, for all $t \in (I'_1 \cup \dots \cup I'_N) \cap [-T+\nu/4, T-\nu/4]$,
\item $\| \phi_\nu \|_{C^{k}(\R)} \leq C_k  (\frac{\nu}{10N})^{-k}$ \,,
\end{itemize}
and then we define
$$
\psi(t) = \begin{cases} \int_{-T}^t  \phi_\nu(s) \frac{d\omega}{dt}(s) ds \,, \quad &\text{ for all } t \in (-T,T)\,, \\
0    \,, \quad  &\text{ for all } t \in \R\setminus (-T,T)\,.
 \end{cases} 
$$
The function $\psi \in C^\infty_0(-T,T)$ since the function $\phi_\nu \frac{d\omega}{dt}$ is odd, hence all of its primitives are even, and has compact support in $(-T,T)$, so that $\psi$ is the unique primitive which vanishes on the complement of $(-T,T)$. Finally, it is straightforward to verify that $\psi$ satisfies all the properties listed in the claim. 
\end{proof} }

\medskip 
\begin{proof}[Proof of Theorem \ref{thm:CLS}]
Let us fix $\eps>0$ and any given number $n+1\in \N \setminus\{0\}$ of even Schwartz functions $\omega_1, \dots, \omega_{n+1} \in \mathcal S(\R)$.   
Let $\Phi=\{F_J\}$ be an admissible family of maximal flow-boxes. 

By Lemma \ref{lemma.psi}  there exist $T$, $\tau>0$ (large) and  $\zeta>0$ (small) such that  if $J$ is such that $T_J>\tau$, there exist even functions $\psi_i \in  C^\infty_0 ( (-T,T)), i=1,\ldots,n+1,$ such that  $\frac{d\psi_i}{dt}  \in  C^\infty_0 ( S^T_\zeta(J))$ with  $C^{k+1}$ norm uniformly bounded  above by a constant $C':=C'(k,\epsilon,\omega_1,\ldots,\omega_{n+1},T)>0$ such that 

\begin{equation}  
\label{eq:conv_approx_3}
\Vert \frac{d^2}{dt^2} (\psi_i\ast \psi_i) - \frac{d^2}{dt^2} ( \omega_i \ast  \omega_i)  \Vert_{L^2(\R)} <  \epsilon/2\,. 
\end{equation} 
Let  now $\chi^{(1)} _J, \dots, \chi^{(n+1)} _J \in C^\infty(J)$ be functions  such that 
$$
\int_J   \chi^{(i)}_J  \chi^{(j)}_J   d\lambda  =   \delta_{ij}  \,,  \quad  \text{for all } i,j \in \{1, \dots, n+1\}\,,
$$
with $C^s$ norm bounded above by $C''/\lambda ( J)^{s+1/2}$ on $J$ for all $s\in \{0, \dots, k\}$  (this is possible provided that the constant $C''$ is taken to be larger than some constant that only depends on $n$). 

Let $C>\max \{C', C''\}$.
For every $i\in \{1, \dots, n+1\}$,  let $f^{(i)}_J \in \mathcal F_k(F_J^T,C,\zeta)$  be the function defined
on the range $R^T_J$ of the flow-box map $F^T_J$ as 
\begin{equation*}
\label{eq:u_def2}
f^{(i)}_J \circ   F^T_J (x, t) :=   \chi^{(i)}_J (x) \frac{d}{dt} \psi_i (t) \,, \quad 
\text{ \rm if } \,\, (x,t) \in J \times (-T, T)\,,
\end{equation*} 
and defined as  $f^{(i)}_J=0$ on $M\setminus R^T_J$. 

\smallskip
We then compute the correlations. Let $T_J /2 > \max \{T, \tau/2\}$. For all $t \in [-T_J,T_J]$ we have (since the functions $\psi_1, \dots, \psi_{n+1}$ are all even)
\begin{equation}
\label{eq:correlations}
\begin{aligned}
\langle f^{(i)}_J\circ T^t, f^{(j)}_J \rangle  &=\int_J \int_{-T}^T   \chi^{(i)}_J(x) \chi^{(j)}_J(x)  
\frac{d\psi_i}{dt} (\sigma+t)   \frac{d\psi_j}{dt}(\sigma) d\sigma  dx
\\ &=  (\frac{d\psi_i}{dt} \ast   \frac{d\psi_j}{dt}) (t)   \, \delta_{ij}  = \frac{d}{dt^2} (\psi_i\ast \psi_j)(t) \delta_{ij}  \,.\end{aligned}
\end{equation}
and by the assumption of the theorem,  if $\lambda(J)$ is small enough,  for every $i,j\in \{1,\ldots,n+1\}$ we have:
\begin{equation} \label{eq.outside1}
 \Vert \langle f^{(i)}_J\circ T^t, f^{(j)}_J \rangle  \Vert_{L^2(\R\setminus [-T_J, T_J])} \leq \epsilon/2\,.\end{equation}
Note that, since the functions $\psi_i$ are supported in $[-T, T]$ and $T<T_J/2$, we also have
\begin{equation} \label{eq.outside2} \frac{d}{dt^2} (\psi_i\ast \psi_j)(t) \delta_{ij}  =0, \text{ for } t\in \R\setminus [-T_J, T_J]. 
\end{equation}
By putting together formulas~ \eqref{eq:conv_approx_3}--\eqref{eq.outside2}, 
it follows that  if $\lambda(J)$ is small enough (hence $T_J$ is large enough), the functions $f^{(i)}_J$, with $i \in \{1, \dots, n+1\}$, satisfy the assumptions of Corollary \ref{cor:criterion}:
$$
\Vert \langle f^{(i)}_J\circ T^t, f^{(j)}_J \rangle  -\frac{d^2}{dt^2}(\omega_i \ast \omega_j) \delta_{ij}  \Vert_{L^2(\R)} \leq \epsilon\,.
$$
It follows then by Corollary~\ref{cor:criterion} that, under the hypotheses of the theorem, the flow $\{T^t\}$ has countable Lebesgue spectrum,  hence the argument
is completed. \end{proof}

\subsection{CILS for Kochergin flows and time changes of horocycle flows} \label{sec6.kochergin}

We prove below that the hypotheses of Theorem~\ref{thm:CLS}  are verified for Kochergin flows $\{T_{\alpha, \varphi}^t\}$.  

 Let $\cal M= \{M_\zeta\vert \zeta>0\}$ denote the family introduced in formula \eqref{Mzeta} of Section~\ref{sec.decay}.

\begin{theorem} \label{thm.koc}  For every Kochergin flow $\{T_{\alpha, \varphi}^t\}$ with $\a \in D_{\log,\xi}$, $\xi<\frac{1}{10}$,  there exists an admissible family of maximal flow-boxes $\{F_J\}$, over 
a decreasing sequence $\{J\}$ of intervals satisfying $\lim \lambda (J) \to 0^+$ (hence $T_J\to + \infty$), such that the following holds.  For any $T>0$, for any $C$ and $\zeta>0$,  for any sequence of pair of functions 
$\{(f_J, g_J)\}$ such that $f_J \in \mathcal F_1(F_J^T,C,\zeta)$ and $g_J \in \mathcal G_1(F_J^T,C,\zeta)$ we have 
$$
\lim_{\lambda(J) \to 0^+}  \int_{\R\setminus [-T_J, T_J]} \vert \langle f_J \circ T_{\alpha, \varphi}^t , g_J \rangle\vert^2  dt =0\,.
$$
\end{theorem} 
\begin{proof} 
\noindent  We consider the following family of maximal flow boxes. We fix $\theta_0$ that is not in the orbit of $0$ by $R_\a$ on the circle and take any sequence of basis intervals $\{J\} \subset \{J_m\}$ with
$$
 J_m := [\theta_0- \frac{1}{10q_m}, \theta_0+ \frac{1}{10q_m}] \times \{0\} \,, \quad \text{ for all } m\in \N.
$$
Since the sequence $\{J_m\}$ is decreasing with respect to the inclusion, it is immediate to prove that the
family $\{F^{J_m}\}$ of maximal flow-boxes is admissible by verifying that the conditions hold for all the degenerate flow-boxes (orbit segments) $F^T _{J_\infty}:= \{T_{\alpha, \varphi}^t(\theta_0, 0) \vert t\in [-T, T]\}$ over the degenerate interval $J_\infty:=\{(\theta_0, 0)\} \subset M$.   Indeed, by the definition of the family 
$\{M_\zeta\}$ (in formula~\eqref{Mzeta} of Section~\ref{sec.decay})  for any $T>0$ and $\nu >0$ there exists $\zeta>0$ such that  $(\theta_0, 0) \in M_\zeta$, the orbit segment $F^T _{J_\infty}$ does not intersect the interval $[-\zeta,\zeta]\times \{0\}$. Therefore the set {$S_\zeta^T (J_\infty)=F^T _{J_\infty}$ is equal to a finite union of $N:= N_T$ intervals such that $\text{Leb} ((-T,T)\setminus S^T_\zeta(J_\infty)) \leq (2T)(2\zeta)$, hence the estimate in formula~\eqref{eq:admissible} of the definition of an admissible family of maximal flow-boxes holds if $4T\zeta<\nu$.}

We also observe that  by construction there exists a constant $c>0$ such that $T_{J_m}\geq cq_m$, for all $m\in \N$,  which implies that for any $J \in \{J_m\}$ we have
\begin{equation}
\label{eq:T_J}
\lambda (J)  \geq      \frac{ c} {5 T_{J }}\,.
\end{equation}
 For any $J \in \{J_m\}$, let us fix {$T \in (0,T_{J}^{1/2})$.} We want to prove a bound on the correlations for any pair of functions $f_J \in \mathcal F_1(F_J^T,C,\zeta)$ and $g_J \in \mathcal G_1(F_J^T,C,\zeta)$, where $C$ is a fixed constant, and derive the  vanishing of the limit in the statement of the theorem.
In the sequel  the symbols $C'$, $C''$  will denote generic universal constants independent of $J$ (but dependent on $T>0$), and that  also  depend on the constant $C$ in the classes of functions $\mathcal F_1(F_J^T,C,\zeta)$ and $\mathcal G_1(F_J^T,C,\zeta)$.
Let then $f_J \in \mathcal F_1(F_J^T,C,\zeta) $ and $g_J \in \mathcal G_1(F_J^T,C,\zeta) $. By Definition \ref{def:test_functions}, the functions $f_J$ and $g_J$ are given by 
$$
f_J \circ F_J^T (x,t) = \phi_J  (x,t)= \chi_J (x) \phi(t)  \quad \text{ and }  \quad g_J \circ F_J^T (x,t) = \psi_J  (x,t)= \chi_J (x) \psi(t) 
$$ 
on $R_J^T$, with $\phi, \psi \in G_1 (J,T, C, \zeta)$ and $\phi$ a derivative, and $f_J=g_J=0$  on $M\setminus \R_J^T$. 

Since the function $g_J$ is supported on the range $R^T_J$ of the flow-box map $F^T_J$  it is enough to prove bounds on 
$$
\int_{R^T_J}  f_J\circ T^t_{\alpha, \vphi} (x) g_J(x) d\mu \,.
$$
Let then $\vert t  \vert \geq T_J$.  WLOG we can assume $t>0$ since the argument for $t<0$ is similar. As in Subsection~\ref{Proof_tdec} we split the estimate into two parts: the integral
over the complement of the bad set $\mathcal B_l$ (see formula~\eqref{def:b} in Subsection~\ref{been} for its definition), and the integral over the bad set.  
\begin{claim}
\label{Koch_Claim1}
There exists $C>0$ such that, for some $\eps>0$ and for all $J \in \{J_m\}$, we have (recalling that $T$ is fixed and $T_J\to +\infty$) 
\begin{equation}
\label{eq:Koch_CILS_good}
\vert \int_{R^T_J\setminus \mathcal B_l}  f_J\circ T^t_{\alpha, \vphi} (x) g_J(x) d\mu \vert  \leq C  \,t^{-1/2-\eps}\,.
\end{equation}
\end{claim}

\begin{proof} Let us recall the partitions $\cI_k$ of $\T$ into intervals with endpoints $\{-i\a\}_{i=0}^{q_k-1}$ and 
$W$ introduced (see formula \eqref{defw}) at the beginning of  Section \ref{est.birk}. By the assumption 
that $t \geq T_J$, by formula \eqref{eq:T_J} there  exists a constant such that $ t \geq C'/2\lambda ( J)$. 
It follows that there exists a product set  $E^T_{J,k}\in W$ with base $\bar E^T_{J,k}$ measurable with respect to the partition $\mathcal I_k$, such that  $R^T_J \subset  E^T_{J,k}$ and we have
$$
\mu (E^T_{J,k}) \leq C' \mu (R^T_J) = C' T \lambda ( J)\,.
$$
By construction there exists a constant $C''>0$ such that  
$$
\begin{aligned}
\Nzero(f_J, g_J) &= \Vert f_J \Vert_0 \Vert g_J \Vert_0 \leq \frac{C''}{\lambda ( J)} \Vert \phi\Vert_0 \Vert \psi\Vert_0  \,; \\
\None (f_J, g_J)&=  (\Vert f_J \Vert_0 + \Vert \phi_J\Vert_0) \Vert g_J \Vert_1 +(\Vert f_J \Vert_1 + \Vert \phi_J\Vert_1) \Vert g_J \Vert_0 \leq \frac{C''}{\lambda ( J)^2}
\Vert \phi\Vert_2 \Vert \psi\Vert_1 \,.
\end{aligned}
$$
Hence it  follows from Proposition~\ref{tdecbis} that
$$
\left| \int_{E^T_{J,k}\setminus \mathcal B_l}f_{J}(T^t_{\a,\vphi}(x,s))g_{J}(x,s)d\mu \right| <   C''\left (C' T  \Vert \phi\Vert_0 \Vert \psi\Vert_0  +  (C' T)^2 \Vert \phi\Vert_2 \Vert \psi\Vert_1  \right) t^{-1/2-\eps}\,.
$$
The bound in formula~\eqref{eq:Koch_CILS_good} is therefore proved and the proof of Claim~\ref{Koch_Claim1}
is completed.
\end{proof}

\smallskip
It remains to estimate the integral on the bad set $\mathcal B_l \cap R^T_J$.  Let $t\in [l^{21/20}, (l+1)^{21/20}]$ with $l\in \N$.  Let us recall the notation $l_0=l^{21/20}$, $l_1= (l+1)^{21/20}$ and let $n\in \N$ be the unique natural number such that $q_n < l_0 <q_{n+1}$. Let then $k\in \N$ be such the $q_k\in [q_n \log^{15} q_n, q_n \log^{20} q_n ]$.

\begin{claim}
\label{Koch_Claim2} There exists a constant $C>0$ such that, for all $l_0 >T_J$, we have
\begin{equation}
\label{eq:spec_type_bad_0}
\int_{l_0}^{l_1}\left|\int_{\mathcal B_l \cap R^T_J  }f_J(T^t_{\a,\vphi}(x))g_J(x)d \mu\right| dt
\leq  C \frac{l_1-l_0}{q_n^{1/2+ 15\eta}} \,.
\end{equation}
\end{claim} 
\begin{proof}
We follow the proof of Proposition \ref{towdec} in Subsection \ref{Proof_bn}.  Let 
$$A_J:=\{t\in [l_0,l_1]\;:\;\int_{\mathcal B_l}f_J(T^t_{\a,\vphi}(x))g_J(x)d \mu>0\}$$ 
and let $\rho_J(t)=1$ if $t\in A_J$ and $\rho_J(t)=-1$ if $t\in [l_0,l_1]\setminus A_J$. Let $F^T_J$ denote as above a flow-box map  and let $R^T_J\subset M$ denote its range. Then, by Cauchy-Schwarz (H\"older) inequality, we have
\begin{multline}
\label{eq:Koch_CILS_0}
\int_{l_0}^{l_1}\left|\int_{\mathcal B_l}f_J(T^t_{\a,\vphi}(x))g_J(x)d \mu\right|dt=\int_{\mathcal B_l \cap R^T_J}\left(\int_{l_0}^{l_1}\rho_J(t) f_J(T^t_{\a,\vphi}(x))dt\right)g_J(x)d \mu \\ \leq
\left(\int_{\mathcal B_l \cap R^T_J}\left(\int_{l_0}^{l_1}\rho_J(t)
f_J(T^t_{\a,\vphi}(x))dt\right)^2d\mu\right)^{1/2}
\left(\int_{\mathcal B_l\cap R^T_J}g_J(x)^2d\mu\right)^{1/2} \\ \leq
 \Vert g_J \Vert_0 \mu(\mathcal B_l \cap R^T_J)^{1/2}
\left(\int_{\mathcal B_l}\left(\int_{l_0}^{l_1}\rho(t)f_J(T^t_{\a,\vphi}(x))dt\right)^2d\mu\right)^{1/2}.
\end{multline}
We split the auto-correlation integral on the RHS of formula \eqref{eq:Koch_CILS_0}, as follows:
\begin{multline}
\label{eq:Koch_CILS_0_bis}
\int_{\mathcal B_l  }\left(\int_{l_0}^{l_1}\rho_J(t)
f_J(T^t_{\a,\vphi}(x))dt\right)^2d\mu  \\ \leq \left(\int_{\mathcal B_l }\left(\int_{l_0}^{l_1}
\left(\int_{r\in[l_0,l_1]\;:\;|r-t|\leq 10T}
\rho(r)\rho(t)f_J(T^t_{\a,\vphi}(x))f_J(T^r_{\a,\vphi}(x))dr\right)dt\right)d\mu\right)\\ +
\left(\int_{\mathcal B_l }\left(\int_{l_0}^{l_1}\left(\int_{r\in[l_0,l_1]\;:\;|r-t|\geq 10T}
\rho(r)\rho(t)f_J(T^t_{\a,\vphi}(x))f_J(T^r_{\a,\vphi}(x))dr\right)dt\right)d\mu\right). 
\end{multline}
By invariance of the measure, since $f_J$ is supported on $R^T_J$, we can write
\begin{equation}
\label{eq:push_fwd}
\begin{aligned}
\int_{\mathcal B_l  }  f_J(T^t_{\a,\vphi}(x))    f_J(T^r_{\a,\vphi}(x)) d\mu  &= \int_{T^{t}_{\a,\vphi}(\mathcal B_l  )  }  f_J(x)   f_J(T^{r-t}_{\a,\vphi}(x)) d\mu
\\ &= \int_{ T^{t}_{\a,\vphi}(\mathcal B_l  ) \cap R^T_J }  f_J(x)   f_J(T^{r-t}_{\a,\vphi}(x)) d\mu \,.
\end{aligned}
\end{equation}
hence we have the immediate estimate
\begin{multline}
\label{eq:Koch_CILS_1}
 \left(\int_{\mathcal B_l }\left(\int_{l_0}^{l_1}\left(\int_{r\in[l_0,l_1]\;:\;|r-t|\leq 10T}
\rho(r)\rho(t)f_J(T^t_{\a,\vphi}(x))f_J(T^r_{\a,\vphi}(x))dr\right)dt\right)d\mu\right) \\ \leq  20 T \Vert f_J \Vert_0^2 (l_1-l_0) \mu (T^{t}_{\a,\vphi}(\mathcal B_l ) \cap R^T_J   ) \,.
\end{multline}
We then have the following crucial fact. For any $(r,t) \in [l_0, l_1]^2$ such that $10 T \leq \vert r-t\vert  \leq T_J/10$,  either  $x'_t:= T^t_{\a,\vphi}(x) \in R^T_J$ or $T^r_{\a,\vphi}(x) = T^{r-t}_{\a,\vphi}(x'_t)  \in R^T_J$ (but not both), hence, by taking into account that the function $f_J$ is supported in $R^T_J$, we have
\begin{multline}
\label{eq:Koch_CILS_1_bis}
\left(\int_{\mathcal B_l}\left(\int_{l_0}^{l_1}\left(\int_{r\in[l_0,l_1]\;:\;|r-t|\geq 10T}
\rho(r)\rho(t)f_J(T^t_{\a,\vphi}(x))f_J(T^r_{\a,\vphi}(x))dr\right)dt\right)d\mu\right)
\\ = \left(\int_{\mathcal B_l}\left(\int_{l_0}^{l_1}\left(\int_{r\in[l_0,l_1]\;:\;|r-t|\geq T_J /10}
\rho(r)\rho(t)f_J(T^t_{\a,\vphi}(x))f_J(T^r_{\a,\vphi}(x))dr\right)dt\right)d\mu\right)\,.
\end{multline}
By formula~\eqref{eq:push_fwd}, our goal is now to estimate for $r-t \geq T_J/10$ the integral 
$$
\int_{ \mathcal B_l}  f_J(T^t_{\a,\vphi}(x))    f_J(T^r_{\a,\vphi}(x)) d\mu = \int_{ T^{t}_{\a,\vphi}(\mathcal B_l ) \cap R^T_J }  f_J(x)   f_J(T^{r-t}_{\a,\vphi}(x)) d\mu \,.
$$
Let then $\bt= r-t$ (which without of generality we can assume
positive) and recall  the notation established in Subsection \ref{Proof_bn}: let $\bl=[\bt]$ and  $\bn$ to be the unique integer such that $q_\bn \leq \bl<q_{\bn+1}$.  Let $\bk$ be any integer such that  $q_{\bk} \in [q_\bn\log ^{15} q_\bn, q_\bn\log^{20} q_\bn]$. We recall that by construction we have $q_\bk \in [q_n^{\frac{1}{41}},  q_n^{\frac{1}{19}} \log^{20} q_n]$.   By the lower bound in formula~\eqref{eq:T_J}, since $t^* \geq T_J/10$, it follows that $\lambda ( J) \geq 1/q_{\bk}$ and that, for any interval $\bar I \in \cI_{\bk}$, 
we have $$ \lambda ( I) \leq 1/q_{\bk}  \leq 1/q_n^{1/41}\,.$$ 
We recall that the set $\mathcal B_l$  was defined (in formula \eqref{def:b} of Subsection~\ref{been}) as a union of a finite number of disjoint complete towers $U_1, \dots, U_m$.  By property $(B_5)$ in Proposition~\ref{prop.badset}, for every $t \in [l_0, l_1]$ there exist complete towers $\cT_{t,1}, \dots, \cT_{t,m}$ which are approximations of the sets $T^{t}_{\a,\vphi}(U_1), \dots, T^{t}_{\a,\vphi}(U_m)$ respectively, and such that 
It therefore suffices to estimate
 $$
\sum_{i=1}^m\left|\int_{\cT_{t,i}  \cap R^T_J}f_J(T^{t^*}_{\a,\vphi}(x))f_J(x)d \mu\right| \,.
 $$
Following Lemma~\ref{lemma.independence},  we distinguish two cases. In the first case we have $N(\theta_{t,i}, h_{t,i}) \leq q_n^{1/3}$. By the bound in~\eqref{eq:small_tower}  we then have
\begin{equation}
\label{eq:Koch_CILS_2}
\left|\int_{\cT_{t,i}  \cap R^T_J  }f_J(x)f_J(T^{\bt}_{\a,\vphi}(x))
d\mu \right| \leq \Vert f_J \Vert_0^2 \,
\frac{\mu(\cT_{t,i}  \cap R^T_J)}{q_n^{\frac{1}{10}}} \leq    \frac{C'}{q_n^{\frac{1}{10}}}  \frac{ \mu(\cT_{t,i}  \cap R^T_J)}{ \lambda(J)} .
\end{equation}
In the second case we have $N(\theta_{t,i}, h_{t,i}) \geq q_n^{1/3}$. From the bound in~\eqref{cortow3}, for 
all $I:= \bar I \times \{s\}$ with $\bar I \in \cI_{\bk}$ we have
\begin{equation}
\label{eq:Koch_CILS_3}
\left|\int_{\overline{\cT_{t,i} \cap  I}}f_J(T^{\bt}_{\a,\vphi}(\theta,s)) f_J(\theta,s)
d\theta\right|\leq C' \{\Nzero(f_J,f_J) + \None(f_J,f_J) \lambda ( I)\}
\frac{\lambda(\cT_{t,i} \cap  I)}{q_n^{50\eta}}.
\end{equation}
By the lower bound \eqref{eq:T_J} it follows that 
$\lambda ( J) \geq 1/q_{\bk}$, hence there exists a product set $E^T_{J,{\bk}}\in W$ with base  $\bar E^T_{J,{\bk}}$ measurable with respect to the  partition $\cI_{\bk}$, such that $R^T_J \subset E^T_{J,\bk}$ and we have
$$
\mu (E^T_{J,\bk}) \leq C \mu (R^T_J) = C' T \lambda ( J)\,.
$$
Thus, from the bound in formula~\eqref{eq:Koch_CILS_3} we derive the following estimate:
\begin{equation}
\label{eq:Koch_CILS_4}
\left|\int_{\cT_{t,i} \cap R^T_J  }f_J(x)f_J(T^{\bt}_{\a,\vphi}(x)) 
d\mu \right |\leq \frac{C'}{q_n^{50 \eta}}   \frac{\mu(\cT_{t,i} \cap  R^T_J)}{ \lambda ( J) } \,,
\end{equation}
hence, by formulas~\eqref{eq:Koch_CILS_2} and~\eqref{eq:Koch_CILS_4}, we have
\begin{equation}
\label{eq:Koch_CILS_4bis}
\left|\int_{T^{t}_{\a,\vphi}(U_i ) \cap R^T_J  }f_J(x)f_J(T^{\bt}_{\a,\vphi}(x)) 
d\mu \right |\leq \frac{C'}{q_n^{50 \eta}}   \frac{\mu( \cT_{t,i} \cap  R^T_J) }{ \lambda ( J)} + C'' \frac{\mu(\cT_{t,i}\triangle T^{t}_{\a,\vphi}(U_i )}{\lambda (J)} \,.
\end{equation}
After summing over all towers of $\mathcal B_l$, that is, over $i\in \{1, \dots, m\}$, by the measure bound in~\eqref{eq:tower_error}, we derive that
\begin{multline}
\label{eq:Koch_CILS_5}
\int_{  T^t_{\a,\vphi}(\mathcal B_l)\cap R^T_J  }f_J(x)f_J(T^{\bt}_{\a,\vphi}(x)) 
d\mu   \leq   \sum_{i=1}^m\left|\int_{T^{t}_{\a,\vphi}(U_i ) \cap R^T_J  }f_J(x)f_J(T^{\bt}_{\a,\vphi}(x)) 
d\mu \right |  \\ \leq  \frac{C'}{q_n^{50 \eta}}   \frac{\mu( \cup_{i=1}^m \cT_{t,i} \cap  R^T_J) }{ \lambda ( J)}    
+  C'' \frac{q_n^{-3/5+ 15\eta} }{\lambda (J)}  \,.
\end{multline} 
By the equidistribution properties of the base rotation under the Diophantine assumption on the rotation number, there exist constants $C', C''>0$  such that
\begin{equation}
\label{eq:bad_equidist}
\max \left( \mu(\cup_{i=1}^m \cT_{t,i} \cap  R^T_J )  , \mu(\mathcal B_l \cap R^T_J)  \right) \leq C' \mu (R^T_J) \frac{\log^2 q_n}{q_n^{1/2-4\eta}} \leq      C''  \frac{\lambda ( J)}{q_n^{1/2-5\eta}} \,,
\end{equation}
hence by formulas \eqref{eq:Koch_CILS_0_bis},  \eqref{eq:push_fwd}, \eqref{eq:Koch_CILS_1}, \eqref{eq:Koch_CILS_1_bis}   and \eqref{eq:Koch_CILS_5} we derive that there exist constants $C', C''>0$
such that
\begin{multline}
\label{eq:Koch_CILS_6}
\int_{\mathcal B_l }\left(\int_{l_0}^{l_1}\rho_J(t)
f_J(T^t_{\a,\vphi}(x))dt\right)^2d\mu \leq   C'(\frac{1}{q_n^{1/2 +45 \eta}}+ \frac{1 }{q_n^{3/5- 15\eta}\lambda (J)})
(l_1-l_0)^2 \\ +  C'' \frac{l_1-l_0}{\lambda (J)} ( \frac{\lambda ( J)}{q_n^{1/2-5\eta}} +  \frac{1}{q_n^{3/5- 15\eta}})\,. 
\end{multline} 
By taking into account that $\lambda (J) \geq 1/ q_{k*} \geq  1/q_n^{1/18}$ and that that $l_1-l_0 \geq q_n^{80\eta}$,  we also have
$$
\frac{l_1-l_0}{q_n^{1/2-5\eta}}  \leq    \frac{(l_1-l_0)^2}{q_n^{1/2+45\eta}}   \quad \text{ and} \quad    \frac{1}{q_n^{3/5- 15 \eta}}   \leq  \min \left( \frac{\lambda ( J)}{q_n^{1/2+45\eta}} , \frac{ (l_1-l_0)\lambda(J)}{ q_n^{1/2+ 45\eta} }  \right)\,,
$$
hence from formula~\eqref{eq:Koch_CILS_6} we derive the following bound on self-correlations:
\begin{equation}
\label{eq:Koch_CILS_6_bis}
\int_{\mathcal B_l }\left(\int_{l_0}^{l_1}\rho_J(t)
f_J(T^t_{\a,\vphi}(x))dt\right)^2d\mu \leq    C \frac{(l_1-l_0)^2}{q_n^{1/2 +45 \eta}}   \,.
\end{equation} 
Finally, from formulas~\eqref{eq:Koch_CILS_0}, \eqref{eq:bad_equidist} and~\eqref{eq:Koch_CILS_6_bis} we derive the bound
\begin{multline}
\label{eq:Koch_CILS_7}
\int_{l_0}^{l_1}\left|\int_{\mathcal B_l }f_J(T^t_{\a,\vphi}(x))g_J(x)d \mu\right|dt 
\\ \leq  \Vert g_J \Vert_0  \mu(\mathcal B_l \cap R^T_J)^{1/2}  \left( \int_{\mathcal B_l }\left(\int_{l_0}^{l_1}\rho_J(t)f_J(T^t_{\a,\vphi}(x))dt\right)^2d\mu    \right)^{1/2}
 \\ \leq C' \left(\frac{\mu( \mathcal B_l \cap R^T_J)}{\lambda (J)}\right)^{1/2}  
   \frac{ (l_1-l_0)}{ q_n^{1/4 +20\eta}}    \leq  C''   \frac{ (l_1-l_0)}{ q_n^{1/2 +15\eta}}     \,.
\end{multline} 
Claim~\ref{Koch_Claim2} is therefore proved.  
\end{proof}
From Claim~\ref{Koch_Claim2}, together with the immediate estimate 
$$
\left|\int_{\mathcal B_l \cap R^T_J}f_J(T^t_{\a,\vphi}(x))g_J(x)d\mu\right|  \leq C' \,
 \frac{\mu(\mathcal B_l \cap R^T_J)}{\lambda ( J)}  
$$
and the bound \eqref{eq:bad_equidist}, we derive our final estimate on the bad  set, that is, as soon as $l_0\geq T_J$, 
\begin{equation}
\label{eq:Koch_CILS_bad}
\int_{l_0}^{l_1}\left|\int_{\mathcal B_l \cap R^T_J  }f_J(T^t_{\a,\vphi}(x))g_J(x)d \mu\right|^2 dt
\leq C'   \frac{l_1-l_0}{ q_n^{1 + \eta}} \leq C' \frac{l_1-l_0}{l_0^{1+\eta/2}} \leq  \frac{C''}{l^{1+\eta/3}} \,.
\end{equation}
The statement of Theorem \ref{thm.koc}  then follows  from the estimates in formulas~\eqref{eq:Koch_CILS_good} and~\eqref{eq:Koch_CILS_bad}.
\end{proof}

 The hypotheses of the criterion for countable Lebesgue spectrum stated in Theorem~\ref{thm:CLS} also hold for all smooth time-changes of a horocycle flow $\{h^t\}$ on the unit tangent bundle $M$ of a compact hyperbolic surface, as we will explain below. 
 
 \medskip
 Let $\cal M= \{M_\zeta\vert \zeta>0\}$ denote in this case the  trivial family such that $M_\zeta=M$ for
all $\zeta>0$. For such a trivial family, all  families of flow-boxes are admissible and, for all  $T>0$ and $C>0$, the families of functions $\mathcal F_k (F^T_J, C,\zeta):= \mathcal F_k (F^T_J, C)$ and $\mathcal G_k (F^T_J, C,\zeta):= \mathcal G_k (F^T_J, C)$, introduced in Definition \ref{def:test_functions}, are independent of $\zeta>0$. 

\smallskip
\begin{proposition} \label{prop:spect_type_tail_horo}
For any sufficiently smooth time change $\{h^t_\vphi\}$ of a horocycle flow  $\{h^t\}$, there exists a family $\Phi:=
\{F_J\}$ of maximal  flow-boxes with $\inf \lambda (J) = 0$ (and $\sup T_J = +\infty$), such that for any $T>0$ and $C>0$, and for any family of pairs  of functions $\{(f_J, g_J)\}$ such that $f_J \in \mathcal F_1 (F^T_J, C)$ and $g_J \in \mathcal G_1(F^T_J, C)$ we have 
$$
\inf_{F^J\in \Phi}  \int_{\R\setminus [-T_J, T_J]} \vert <f_J\circ h^t_{\vphi}, g_J>\vert^2 dt =0 \,.
$$
\end{proposition}

\begin{proof}  The estimates required to prove this assertion are carried out in Subsection 6.3 of 
 \cite{FU}  where the authors prove that sufficiently smooth time-changes of the horocycle flow have Lebesgue maximal spectral type (see the Remark~\ref{rem:smoothness_horo} below about the smoothness assumptions).   The base of the flow-boxes are $2$-dimensional intervals of uniform (fixed) size in the geodesic
 direction and arbitrarily small size in the direction of the complementary horocycle.  Such flow-boxes and the
 test functions of the classes $\mathcal F_1 (F^T_J, C)$ and  $\mathcal G_1 (F^T_J, C)$ are introduced in Lemma 28 of \cite{FU} (where the relevant estimates on their derivatives are proved).  The key estimates on correlations of functions in the classes $\mathcal F_1 (F^T_J, C)$ and  $\mathcal G_1 (F^T_J, C)$ (supported on flow-boxes) are stated  in \cite{FU} as formulas $(40)$ and $(41)$ in the proof of Lemma 28. In fact, estimates on correlations 
 follow from those formulas by taking into account the formula of Lemma 9 of \cite{FU}, which reduces estimates on correlations for the time changes to  estimates on curvilinear integrals along push-forwards of geodesic arcs.
\end{proof} 

\begin{remark} \label{rem:smoothness_horo} The relevant estimates on correlations from  \cite{FU}  (see in particular Lemma 24) are stated
for time-change functions of $L^2$ Sobolev regularity  $r>11/2$. However, the argument that establishes that the maximal spectral type is Lebesgue, which also implies the hypotheses of our criterion for countable Lebesgue spectrum,  hold under the milder hypothesis that the time-change function is 
$C^1$ and $C^2$ in the geodesic direction. 
\end{remark}

We are now ready to derive our conclusions. 
By Theorem \ref{cob}, Theorem~\ref{thm:CLS}, and Theorem \ref{thm.koc} 
we derive the completion of the
proof of our main result:

\begin{corollary}
\label{cor:main} 
For $\a \in D_{\log,\xi}$, $\xi<\frac{1}{10}$, the dynamical system $(T^t_{\alpha,\vphi}, M, \mu)$ has 
Lebesgue spectral type with countable multiplicity.
\end{corollary}

By Theorem 25 of~\cite{FU}, \S 6.3,  Theorem~\ref{thm:CLS} of Subsection~\ref{sec6.decay}, and Proposition~\ref{prop:spect_type_tail_horo}, we derive a similar result  for smooth time changes of the horocycle
flow, thereby completing the proof of the Katok-Thouvenot conjecture (\cite{KT}, Conjecture 6.8):

\begin{corollary}
\label{cor:horo_time_changes} 
Any flow obtained by a sufficiently smooth time change from a horocycle flow has Lebesgue spectral
type with countable multiplicity.
\end{corollary}

  \appendix
  
  \bg 
  
  \section{Birkhoff sums estimates}\label{app.birkhoff} \blk 
  \begin{proof}[Proof of Lemma \ref{sec.fir}] By the definition of $u_I$ in \eqref{ui}, we know that  there exist
  $x_0\in I\cap T^{-t}_{\a,\vphi}(W)$ and $t_0\in [l_0,l_1]$ such that  $\vphi''_{N(x_0,t_0)}(\bar{x}_0)\geq 
  q_n^{3-\eta}\log^9q_n$. 
Since $N(x_0,t_0)<cq_{n+1}$, by \eqref{koks2}, we get 
$$(q_n)^{3-\eta}\log^9q_n<
\vphi''_{N(x_0,t_0)}(\bar{x}_0)<2c^3q_{n+1}^{3-\eta}+\frac{1}{x^{N(x_0,t_0)}_{min}}
<q^{3-\eta}_n\log^4q_n+\frac{1}{x^{N(x_0,t_0)}_{min}},$$
which means that there exists $j\in[0,N(x_0,t_0)-1]$ s.t. 
\begin{equation}\label{cusp}\bar{x}_0+j\alpha\in [-\frac{1}{q_n\log ^{3} q_n},\frac{1}{q_n\log ^{3} q_n}].
\end{equation}
We will show that, for every $t\in [l_0,l_1]$ and every $x\in I\cap T^{-t}_{\a,\vphi}(W)$, we have
\begin{equation}\label{nkept}
N(x,t)>j.
\end{equation}
Let us first show how \eqref{nkept} implies \eqref{xmin} and \eqref{eq:sec}. Since $\lambda ( I)\leq \frac{1}{q_n\log^{15}q_n}$
it follows by \eqref{nkept} that for every $t\in [l_0,l_1]$ and every $x\in I\cap T^{-t}_{\a,\vphi}(W)$ 
$$
x_{min}^{ N(x,t)}\leq d(\bar x+j\alpha,0)\leq d(\bar{x}+j\a,0)+\lambda ( I)\leq \frac{1}{2q_n\log^3q_n}.
$$
This gives \eqref{xmin}. For \eqref{eq:sec}, we have by \eqref{koks1} and \eqref{koks2}
$$
|\vphi'_{N(x,t)}(\bar x)|\geq \left(\frac{2}{3x_{min}^{N(x,t)}}\right)^{2-\eta}-4q^{2-\eta}_{n+2}
\geq \left(\frac{1}{2x_{min}^{N(x,t)}}\right)^{2-\eta},
$$
and
$$
|\vphi ''_{N( x,t)}(\bar x)|\leq \left(\frac{3}{2x_{min}^{N( x,t)}}\right)^{3-\eta}+4q^{3-\eta}_{n+2}
\leq \left(\frac{2}{x_{min}^{N( x,t)}}\right)^{3-\eta}.
$$
This gives \eqref{eq:sec}. Therefore it remains to show \eqref{nkept}. Notice that for $x\in I\cap T^{-t}_{\a,\vphi}(W)$, \eqref{nkept} is equivalent to
\begin{equation}\label{nkept2}
N( x,t)\geq j
\end{equation}
(since $T^{t}_{\a,\vphi}(x)=(\bar x+N( x,t)\a,s')\in W$). Notice also that if the lower bound
\begin{equation}\label{nkept3}
N( x,t_0)\geq j,
\end{equation}
holds, then \eqref{nkept2} follows for all $t\in [l_0,l_1]$. Indeed, otherwise we have
\begin{multline*}(4q_{n+1})^{1-\eta}\geq \vphi(\bar x+N( x,t)\a)\geq t+s-\vphi_{N( x,t)}(\bar x)\geq\\ \vphi_{N( x,t_0)}(\bar x)-\vphi_{N( x,t)}(\bar x)\geq \vphi(\bar x+j\a)\geq q_n^{1-\eta}\log^2 q_n,
\end{multline*}
a contradiction. Hence it remains to show \eqref{nkept3}. Assume by contradiction that $N( x,t_0)<j$ for some 
$x\in I \cap T^{-t}_{\a,\vphi}(W)$. Then, by the definition of $j$, we have 
\begin{equation}\label{disj}
\bigcup_{i=0}^{N( x,t_0)}R_\a^i(\bar I)\cap \left[-\frac{1}{5q_{n+2}},\frac{1}{5q_{n+2}}\right]=\emptyset.
\end{equation}
Therefore, for every $\theta \in \bar I$ by \eqref{koks1} we have
\begin{equation}\label{fvd}
|\vphi'_j(\theta)|<10q_{n+2}^{2-\eta}.
\end{equation}
Hence, by~\eqref{cusp}, \eqref{disj}, and~\eqref{fvd}, for some $\theta \in \bar I$, we get
\begin{multline*}
(5q_{n+2})^{1-\eta}\geq \max(\vphi(\bar x+N(x,t_0)\a),\vphi(\bar{x}_0+N(x_0,t_0)\a)) \geq\\ |\vphi_{N(x,t_0)}(\bar x) -\vphi_{N(x_0,t_0)}(\bar{x}_0)|\geq \vphi(\bar{x}_0+j\a)-|\vphi_{j}'(\theta)|\lambda ( I)\geq 1/2\left(q_n\log^3 q_n\right)^{1-\eta}-(q_{n+1})^{1-\eta}, 
\end{multline*}
which yields a contradiction since $q_{n+2}<q_n\log^{2+3\xi}q_n$. So \eqref{nkept3} holds. This completes the proof of Lemma \ref{sec.fir}. \end{proof}

  \begin{proof}[Proof of Lemma \ref{xze}]
  
Notice that for some $\theta\in [\bar x,\bar x_0]$ we have
$$
\vphi'_{N(x)}(\bar x)-\vphi'_{N( x_0)}(\bar x_0)=\vphi''_{N( x_0)}(\theta)(\bar x-\bar x_0)+
\vphi'_{N( x)-N( x_0)}(\bar x+N(x_0)\a).
$$
Since 
$|\vphi''_{N( x_0)}(\bar x_0)|\leq q_n^{3-\eta}\log^{10}q_n$, by \eqref{koks2} for $N=N( x_0)$ it follows that 
\begin{equation}\label{nin}
\{\bar x_0,\dots,\bar x_0+(N(x_0)-1)\a\}\cap [-\frac{1}{q_n\log^4 q_n},\frac{1}{q_n\log^4 q_n}]=\emptyset.
\end{equation}
Notice that since $x_0\in W$, for some constant $c>0$, we have
$$\vphi_{N( x_0)}(\bar x_0)\geq t-q_n^{3/4}\geq cq_n.
$$
So by \eqref{nin}, by \eqref{koks0} for $N=N( x_0)$ and by the Diophantine condition on $\a$, we have 
$q_{r+1}\geq \frac{c q_n}{10}$ (where $r$ is such that $q_r\leq N( x_0)\leq q_{r+1}$ ).
But then by \eqref{koks0} for $N=N( x_0)$ and $x=\theta$ and again by the Diophantine condition on $\a$, we have 
$$\vphi''_{N( x_0)}(\theta)\geq \frac{q_n^{3-\eta}}{\log^{5} q_n}.$$
 Define $A_{x,x_0}:=\vphi''_{N( x_0)}(\theta)$. We will show that 
\begin{equation}\label{wnc}|\vphi'_{N( x)-N( x_0)}(\bar x+N( x_0)\a)|\leq \frac{A_{x,x_0}}{10}|\bar x-\bar x_0|.
\end{equation}
By the definition of $N( x)$, $N( x_0)$ and since $T^{t}_{\a,\vphi}(x)\in V$, for some $z\in [\bar x, \bar x_0]$ we have
\begin{multline*}
2q_n^{3/4(1-\eta)}\geq |(t-\vphi_{N( x_0)}(\bar x_0))-(t-\vphi_{N( x)}(\bar x))|\geq
|\vphi_{N( x)}(\bar x)-\vphi_{N( x_0)}(\bar x_0)|=\\ |\vphi'_{N(x_0)}(\bar x_0)(\bar x-\bar x_0)+
\vphi''_{N( x_0)}(\bar z)(\bar x-\bar x_0)^2+
\vphi_{N( x)-N( x_0)}(\bar x+N( x_0)\a)|.
\end{multline*}
Moreover, we have the following:

\smallskip
{\bf Claim. }{\it  If $\vphi ''_{N( x)}(\bar x)<q_n^{3-\eta}\log^{10}q_n$, then for every $z\in I$
$$\vphi ''_{N( x)}(\bar z)<30q_n^{3-\eta}\log^{10}q_n.$$} 

Therefore
$$
|\vphi_{N( x)-N( x_0)}(\bar x+N( x_0)\a)|\leq 2q_n^{3/4(1-\eta)}+q_n^{7/4+\eta}|\bar x-\bar x_0|+
q_n^{3-\eta}\log^5 q_N(\bar x-\bar x_0)^2,
$$
so by Lemma \ref{fi},
\begin{multline}\label{fip}|\vphi'_{N( x)-N( x_0)}(\bar x+N( x_0)\a)|\leq \\3\left( 4q_n^{3/2(1-\eta^2)}+q_n^{(7/2+2\eta)(1+\eta)}|\bar x-\bar x_0|^2+
q_n^{(6-2\eta)(1+\eta)}\log^{10+2\eta} q_n(\bar x-\bar x_0)^4\right).
\end{multline}
Notice however that since $\frac{1}{q_n\log ^{15} q_n}\geq \frac{1}{q_k}\geq \lambda ( I)\geq |\bar x-\bar x_0|\geq \frac{1}{q_n^{3/2-2\eta}}$, we have  
\begin{multline*}
\frac{q_n^{3-\eta}}{\log^{10} q_n}|\bar x-\bar x_0|\geq\\ 100\max\left(q_n^{3/2(1-\eta^2)},q_n^{(7/2+2\eta)1+\eta}|\bar x-\bar x_0|^2,
q_n^{(6-2\eta)(1+\eta)}\log^{10+2\eta} q_n(\bar x-\bar x_0)^4\right).
\end{multline*}
Therefore and using \eqref{fip} we get \eqref{wnc}
which completes the proof of Lemma \ref{xze}.

\smallskip
We just have to give the proof of the claim.

\begin{proof}[Proof of the Claim] We know that $N( x)\leq q_{n+2}$. If $\vphi ''_{N( x)}(\bar z)\geq 30q_n^{3-\eta}\log^{10}q_n$, by \eqref{koks2} it follows that $z^{N( x)}_{min}\leq \frac{1}{3q_n\log^{\frac{10}{3-\eta}}q_n}$. But since $x,z\in I$ and $\lambda ( I)<\frac{1}{q_n\log^{15}q_n}$, we would have $x^{N( x)}_{min}\leq \frac{1}{2q_n\log^{\frac{10}{3-\eta}}q_n}$. So by applying \eqref{koks2} for $N(x)$ and $x$, we would get
$\vphi ''_{N( x)}(\bar x)\geq 2q_n^{3-\eta}\log^{10}q_n$, a contradiction. \end{proof}

\end{proof}

\section*{Acknowledgments}
The authors are very grateful to Anatole Katok and to Jean-Paul Thouvenot for valuable discussions and suggestions. B.~Fayad was supported by  ANR-15-CE40-0001 and by the project BRNUH. 
G.~Forni was supported by NSF Grants DMS 1201534 and 1600687, and by a Simons Fellowship. He would 
also like to thank the Institut de Math\'ematiques de Jussieu (IMJ) for its hospitality during the academic 
year 2014-15 when work on this paper began, and during the academic year 2018-19 when the current
version of the paper was completed.

\end{document}